\documentclass[12pt]{article}
\usepackage{geometry}                
\usepackage{graphicx,color,mathtools}
\usepackage{amssymb,amsmath,amsthm,mathrsfs}
\usepackage[all,cmtip]{xy}
\usepackage{epstopdf, comment, url}
\usepackage{bm} 
\usepackage{enumitem}

\usepackage[pdftex,bookmarks,pdfnewwindow,plainpages=false,unicode,pdfencoding=auto]{hyperref}

 \hypersetup{
pdfauthor={Oleg Ivrii, Mariusz Urba\'nski},
pdftitle={Inner Functions, Composition Operators, Symbolic Dynamics and Thermodynamic Formalism},
pdfsubject={Dynamical systems in the unit disk},
bookmarksdepth={4}
}

\newtheorem{theorem}{Theorem}[section]
\newtheorem{corollary}[theorem]{Corollary}
\newtheorem{lemma}[theorem]{Lemma}

\theoremstyle{remark}
\newtheorem*{remark}{Remark}
\newtheorem*{example}{Example}

\numberwithin{equation}{section}

\DeclareMathOperator{\diam}{diam}

\DeclareMathOperator{\aut}{Aut}
\DeclareMathOperator{\hol}{Hol}
\DeclareMathOperator{\hyp}{hyp}
\DeclareMathOperator{\dist}{dist}
\DeclareMathOperator{\loc}{loc}

\DeclareMathOperator{\sph}{\hat{\mathbb{C}}}

\DeclareMathOperator{\re}{Re}
\DeclareMathOperator{\im}{Im}

\DeclareMathOperator{\per}{Per}

\DeclareMathOperator{\id}{Id}
\DeclareMathOperator{\BV}{BV}

\DeclareMathOperator{\Cl}{Cl}

\DeclareMathOperator{\Interior}{Int}

\DeclareMathOperator{\bounded}{b}
\def\N{\mathbb N}
\def\C{\mathbb C}

\title{\bf \Large{Inner Functions, Composition Operators, Symbolic Dynamics and Thermodynamic Formalism}}
\author{Oleg Ivrii and Mariusz Urba\'nski}

\date{}                                           

\begin{document}


\maketitle

\abstract{In this paper, we use thermodynamic formalism to study the dynamics of inner functions $F$ acting on the unit disk. If the Denjoy-Wolff point of $F$ is in the open unit disk, then without loss of generality, we can assume that $F(0) = 0$ so that 0 is an attracting fixed point of $F$ and the Lebesgue measure on the unit circle is invariant under $F$. Utilizing the connection between composition operators, Aleksandrov-Clark measures and Perron-Frobenius operators, we develop a rudimentary thermodynamic formalism which allows us to prove the Central Limit Theorem and the Law of the Iterated Logarithm for Sobolev multipliers and H\"older continuous observables. 

Under the more restrictive, but natural hypothesis that $F$ is a one component inner function, we develop a more complete thermodynamic formalism which is sufficient for orbit counting, assuming only the $(1+\varepsilon)$-integrability of $\log|F'|$. As one component inner functions admit countable Markov partitions of the unit circle, we may work in the abstract symbolic setting of countable alphabet subshifts of finite type. Due to the very weak hypotheses on the potential, we need to pay close attention to the regularity of the complex Perron-Frobenius operators $\mathcal L_s$ with $\re s > 1$ near the boundary.

 Finally, we discuss inner functions with a Denjoy-Wolff point on the unit circle. We assume a parabolic type behavior of $F$ around this point and we introduce the class of parabolic one component inner functions. By making use of the first return map, we deduce various stochastic laws and orbit counting results from the aforementioned abstract symbolic results.}

\tableofcontents

\section{Introduction}

A {\em finite Blaschke product} is a  holomorphic self-map of the unit disk which extends to a continuous dynamical system
on the unit circle. It is determined by the location of its zeros up to a rotation:
$$
F(z) = e^{i\theta} \prod_{i=1}^d \frac{z-a_i}{1-\overline{a_i}z}, \qquad a_i \in \mathbb{D}.
$$

Loosely speaking, an {\em inner function} is a  holomorphic self-map of the disk which extends to a measure-theoretic dynamical system
on the unit circle. More precisely, we require that for a.e.~$\theta \in [0, 2\pi)$, the radial boundary value $F(e^{i\theta}) := \lim_{r \to 1} F(re^{i\theta})$ exists and has absolute value 1.

One can classify inner functions according to the location of the Denjoy-Wolff fixed point. We mostly deal with the case when the Denjoy-Wolff fixed point is in the open unit  disk. After a conjugation, we may assume that $F(0) = 0$. We call such inner functions {\em centered}\/. By the Schwarz Lemma, under iteration, all points in the unit disk tend to 0.
We write $m = d\theta/2\pi$ for the normalized Lebesgue measure on the unit circle.

To set the stage, we recall a number of dynamical properties of $m$ such as invariance, ergodicity and mixing due to Sullivan-Shub \cite{SS} and Pommerenke \cite{pom}:

\begin{lemma}
Suppose that $F(z)$ is an inner function such that $F(0) = 0$. 

{\em (i)} The Lebesgue measure on the unit circle is invariant under $F$, that is,
$m(F^{-1}(E)) = m(E)$ for any measurable set $E \subset \partial \mathbb{D}$.

{\em (ii)} $m$ is ergodic, that is,
any invariant set $E \subset \mathbb{D}$ has measure 0 or 1.

{\em (iii)} The map $F: \partial \mathbb{D} \to \partial \mathbb{D}$ is mixing, that is, for any measurable sets $A, E \in \partial \mathbb{D}$,
$$
m(A \cap F^{-n}(E)) \to m(A) \cdot m(E), \qquad \text{as }n \to \infty.
$$
\end{lemma}

The above properties may be proved by examining Poisson extensions of characteristic functions.

In this paper, we employ thermodynamic formalism to show stronger statistical properties of inner functions such as the Central Limit Theorem, the Law of the Iterated Logarithm and the Almost Sure Invariance Principle. We also obtain an Orbit Counting Theorem for one component inner functions satisfying a mild integrability condition. Later, we also discuss Orbit Counting for parabolic one component inner functions.

\subsection{Thermodynamic Formalism and Stochastic Laws for Centered Inner Functions}

Our first objective is to study stochastic properties of arbitrary inner functions $F$ with $F(0) = 0$ acting on the unit circle. We begin by surveying a number of results on the spectral properties of the {\em composition operator} 
$$
C_F: h \to h \circ F
$$ 
acting on various spaces of real and holomorphic functions on the unit circle. We then define  the {\em Perron-Frobenius} or {\em transfer operator} $L_F$ with potential $-\log|F'|$ as the dual of the composition operator with respect to the $L^2$ pairing on the unit circle.

In Section \ref{sec:transfer-operators}, we will see that the operator $L_F$ admits a convenient representation in terms of the Aleksandrov-Clark measures:
 \begin{equation*}
(L_F g)(\alpha) = \int_{\partial \mathbb{D}} g(z) d\mu_\alpha(z), \qquad \alpha \in \partial \mathbb{D},
\end{equation*}
which generalizes the classical definition
\begin{equation*}
(L_F g) (x) = \sum_{F(y) = x} |F'(y)|^{-1} g(y),
\end{equation*}
valid for finite Blaschke products. In fact, the two definitions agree precisely when the Alexandrov-Clark measures $\{\mu_\alpha\}$ are purely atomic. More generally, one can define Perron-Frobenius operators for other potentials by taking adjoints of weighted composition operators.

We use duality arguments to show that for a large class of potentials, the Perron-Frobenius operator has a spectral gap when acting on spaces of Sobolev and H\"older functions. From here, stochastic laws follow quite easily: the Central Limit Theorem is an immediate consequence of Gordin's results \cite{gordin} while the Law of the Iterated Logarithm and the Almost Sure Invariance Principle follow from the work of Gou\"ezel  \cite{gouezel}.

We now describe the Central Limit Theorem  (CLT) in greater detail.
 As is standard in ergodic theory and dynamical systems, we use the notation
$$
S_n h(x) := h(x) + h(F(x))+ \dots + h(F^{\circ (n-1)})(x)
$$
to repesent the sum of a function $h$ along the forward orbit of $x$. 
One says that the Central Limit Theorem holds for an observable $h: \partial \mathbb{D} \to \mathbb{R}$ and a measure $\mu$ on $\partial \mathbb{D}$ if the random variables
$\frac{S_n h}{\sqrt{h}}$ on the probability space $(\partial \mathbb{D}, \mu)$ converge in distribution to the Gaussian distribution
$$
\mathcal N(\overline{h}, \, \sigma^2(h - \overline{h})), \qquad \overline{h} = \int_{\partial \mathbb{D}} h(x) d\mu(x),
$$
for some constant $\sigma^2 = \sigma^2(h) \ge 0$ depending on $h$.

\begin{theorem}
\label{main-thm}
If $F$ is a centered inner function, then the Central Limit Theorem holds with respect to the Lebesgue measure on the unit circle for 
\begin{itemize}
\item any Sobolev multiplier $h \in  \mathcal M(W^{1/2,2}(\partial \mathbb{D}))$, 
\item any H\"older continuous function $h \in  C^\alpha(\partial \mathbb{D})$ with $\alpha > 0$.
\end{itemize}
Furthermore, if $F$ is not a finite Blaschke product and   $h$ is non-constant, then $\sigma^2(h) \ne 0$.
\end{theorem}

The space of H\"older continuous functions $C^\alpha(\partial \mathbb{D})$ is an algebra since the product of two functions in $C^\alpha(\partial \mathbb{D})$ is again a function in $C^\alpha(\partial \mathbb{D})$.

Loosely speaking, for $\beta > 0$, the Sobolev space $W^{\beta,2}(\partial \mathbb{D})$ consists of $L^2(\partial \mathbb{D})$ functions which have $\beta$ derivatives in $L^2(\partial \mathbb{D})$. In terms of Fourier coefficients, this means that
$$
\sum_{n=-\infty}^\infty |n|^{2\beta} |\hat{h}(n)|^2 < \infty.
$$ 
For $0 < \beta \le 1/2$, the space $W^{\beta,2}(\partial \mathbb{D})$ is not an algebra, so we work with the multiplier space
$\mathcal M(W^{\beta,2}(\partial \mathbb{D}))$ instead.
A function $\gamma$ is called a  $W^{\beta,2}(\partial \mathbb{D})$ {\em multiplier} if for any function $g \in W^{\beta,2}(\partial \mathbb{D})$, the product $g\gamma \in W^{\beta,2}(\partial \mathbb{D})$, in which case,
\begin{equation}
\label{eq:multiplier-def}
\| g\gamma \|_{W^{\beta,2}(\partial \mathbb{D})} \le  C \| g \|_{W^{\beta,2}(\partial \mathbb{D})},
\end{equation}
for some $C > 0$. The infimum of constants $C$ one can put in the right hand side of (\ref{eq:multiplier-def}) is called the {\em multiplier norm} and is denoted by $\| \gamma \|_{\mathcal M(W^{\beta,2}(\partial \mathbb{D}))}$. 

For concreteness, we take $\beta = 1/2$. It is well known that any Sobolev multiplier is a bounded function, and for any $\varepsilon > 0$,
$$
W^{1/2+\varepsilon, 2}(\partial \mathbb{D})  \, \subset \, \mathcal M(W^{1/2,2}(\partial \mathbb{D})) \, \subset \, W^{1/2,2}(\partial \mathbb{D}).
$$
The first inclusion shows that any $C^1$ function is a $W^{1/2,2}$ multiplier. For more on multiplication in Sobolev spaces, we refer the reader to the survey \cite{sobolev-multipliers}.

In additional to the Lebesgue measure, we also establish the Central Limit Theorem with respect to equilibrium measures associated to potentials of the form $-\log|F'| + \gamma$, where $\gamma$ is a positive function in $C^\alpha(\partial \mathbb{D})$  of small norm.

\subsection{Orbit Counting for Centered and Parabolic One Component Inner Functions}

To delve deeper into understanding the dynamical properties of inner functions, we restrict to a natural class of inner functions called {\em one component inner functions} introduced by W.~Cohn in \cite{cohn}. According to the original definition, an inner function $F:\mathbb{D}\to\mathbb{D}$ is a one component inner function if the set $\{ z \in \mathbb{D} : |F(z)| < c \}$ is connected for some $c\in (0,1)$. From the point of view of dynamical systems, it is more useful to say that a one component inner function is an inner function whose set of singular values is compactly contained in the unit disk. 

The equivalence of the two definitions, as well as a number of basic properties of one component inner functions will be established in Section \ref{sec:one-component}. In particular, we will see that one component inner functions admit sufficiently nice Markov partitions, which allows us  to view them as conformal graph directed Markov systems -- a notion introduced in \cite{MU} and further developed in \cite{URM22, KU1}. We may therefore study the dynamics of one component inner functions on the unit circle by working in the abstract symbolic setting of countable alphabet subshifts of finite type.

Let $F:\mathbb{D}\to\mathbb{D}$ be a centered one-component inner function. For $x \in \partial \mathbb{D}$ and $T > 0$, consider the {\em counting function}
$$
n(x, T):= \# \bigl \{ (n \ge 0, y \in \partial \mathbb{D}) : F^{\circ n}(y) = x \text{ and } \log | (F^{\circ n})'(y)| < T \bigr \}.
$$
More generally, for a Borel set $B \subset \partial \mathbb{D}$, one can form the function $n(x,T,B)$ which counts the number of iterated pre-images $y$ that lie in $B$. We show:

\begin{theorem}
\label{orbit-counting1a}
Let $F$ be a one component inner function with $F(0) = 0$. Assume that $F$ has infinite degree or is a finite Blaschke product of degree $d \ge 2$ other than  $z \to z^d$.
Under the hypothesis
\begin{equation}
\label{eq:orbit-counting1}
\int_{\partial \mathbb{D}} \bigl ( \log |F'(z)| \bigl )^{1+\varepsilon} dm < \infty, \qquad \text{for some }\varepsilon > 0,
\end{equation}
we have
$$
n(x, T, B) \sim \frac{e^T}{\int_{\partial \mathbb{D}} \log |F'(z)| dm}  \cdot m(B) \qquad \text{as }T \to \infty,
$$
where $B \subset \partial \mathbb{D}$ is a Borel set with $m(\partial B) = 0$.
\end{theorem}

\begin{theorem}
\label{orbit-counting2a}
Let $F$ be a one component inner function such that $F(0) = 0$, which is not a rotation.
If the Lebesgue measure of $F$ has finite entropy, i.e.~
\begin{equation}
\label{eq:orbit-counting2}
\int_{\partial \mathbb{D}} \log |F'(z)| dm < \infty,
\end{equation}
then
$$
 \frac{1}{T} \biggl \{ \int_0^T \frac{n(x, t, B)}{e^t} \, dt \biggr \} \sim   \frac{1}{\int_{\partial \mathbb{D}} \log |F'(z)| dm}  \cdot m(B) 
  \qquad \text{as }T \to \infty.
$$
\end{theorem}
 
 \smallskip
 
 The above theorems were obtained for finite Blaschke products in \cite[Section 7]{ivrii-wp} based on renewal-theoretic considerations of S.~Lalley.
In order to study orbit counting for inner functions, we use the approach pioneered in \cite{PU} which involves examining the Poincar\'e series
$$
\eta_x^B(s) \, = \,   \sum_{n \ge 0} \sum_{\substack{F^{\circ n}(y)=x \\ y \in B}} |(F^{\circ n})'(y)|^{-s} \, = \, \int_0^\infty e^{-sT} \, dn(x,T, B).
$$
Following the strategy outlined in \cite{PU}, we show:
\begin{lemma}
\label{eta}
Let $F$ be a one component inner function with $F(0) = 0$, which satisfies the integrability condition (\ref{eq:orbit-counting1}) for some $0 < \varepsilon < 1$. Suppose $B \subset \partial \mathbb{D}$ is a Borel set with $m(\partial B) = 0$. For any $x \in \partial \mathbb{D}$, the  function $\eta^B_x(s)$ is

{\em (1)} holomorphic from $\mathbb{C}_1^+ = \{\re s > 1\} \to \mathbb{C}$, 

{\em (2)} continuous from $\overline{\mathbb{C}_1^+} \setminus \{1\} \to \mathbb{C}$.

{\em (3)} near $s = 1$,
$$
 \eta_x^B(s) - \frac{1}{\int_{\partial \mathbb{D}} \log |F'(z)| dm} \cdot \frac {1}{s-1} = O  \biggl (\frac{1}{|s-1|^{1-\varepsilon}} \biggr).
$$
If we merely assume that the inner function $F$ has finite entropy, then the above conclusions hold if we replace
 {\em (3)} with

{\em ($3'$)} near $s=1$,
$$
 \eta_x^B(s) - \frac{1}{\int_{\partial \mathbb{D}} \log |F'(z)| dm} \cdot \frac {1}{s-1} = o \biggl (\frac{1}{|s-1|} \biggr).
$$
\end{lemma}

\smallskip

Once we show the above lemma, Theorem \ref{orbit-counting1a} and Theorem \ref{orbit-counting2a} follow follow from the Tauberian theorems of Wiener-Ikehara and Hardy-Littlewood respectively. For the convenience of the reader, we recall these theorems below:

\begin{theorem}[Wiener-Ikehara]
\label{wiener-ikehara}
Suppose $\mu \ge 0$ is a locally finite measure on $(0,+\infty)$ and $$\eta(s) = \int_0^\infty e^{-sT} d\mu$$ converges in the right half-plane $\mathbb{C}^+_1 = \{\re s > 1\}$. If
$$
\eta(s) - \frac{c}{s-1}
$$
extends to an $L^1_{\loc}$ function on the vertical line $\{\re s = 1 \}$, then
$$
\mu([0,T]) \sim c \, e^T  \qquad \text{as }T \to \infty.
$$
\end{theorem}

\begin{theorem}[Hardy-Littlewood]
\label{hardy-littlewood}
If
$$
\eta(s) \sim \frac{c}{s-1}
$$
as $s \to 1^+$ along the real axis, then
$$
\frac{1}{T} \int_0^T \frac{d\mu(t)}{e^t} \to c \qquad \text{as }T \to \infty.
$$
\end{theorem}

Note that the Hardy-Littlewood Tauberian Theorem has weaker hypotheses than
Wiener-Ikehara Tauberian theorem but also a weaker conclusion. The Hardy-Littlewood theorem is usually stated in the half-plane $\{\re s > 0\}$ rather than $\{ \re s > 1 \}$, which explains the $e^t$ in the denominator. For the proofs, we refer the reader to \cite[Corollary 8.7]{montgomery-vaughan} and \cite[Theorem 98]{hardy}.

\begin{remark}
(i) In the treatise \cite{PU}, one asks that the Poincar\'e series admits a meromorphic continuation to a slightly larger half-plane
$\mathbb{C}_{1-p}^+ = \{\re s > 1-p \}$ with a simple pole at $s = 1$, which translates to the more stringent requirement
$$
\int_{\partial \mathbb{D}}   |F'(z)|^p dm < \infty, \qquad \text{for some }p > 0.
$$
Due to the weak integrability hypothesis on the inner function $F$, the Poincar\'e series usually does not have a meromorphic continuation beyond the line 
vertical line $\{\re s=1\}$. As such, we are forced to deal with the subtle issues of differentiability and holomorphy of the complex Perron-Frobenius operators $\mathcal L_s$ with $\re s \ge 1$ near the vertical line $\{\re s=1\}$.

(ii) Although we will not pursue this here, one can also count periodic orbits. Using the methods of \cite{PU}, one can show that
$$
n_{\per}(T, B) \sim \frac{e^T}{\int_{\partial \mathbb{D}} \log |F'(z)| dm}  \cdot m(B) \qquad \text{as }T \to \infty,
$$
where
$$
n_{\per}(T) = \# \bigl \{ y \in \partial \mathbb{D} : F^{\circ n}(y) = y \text{ for some }n \ge 1 \text{ and } \log | (F^{\circ n})'(y)| < T \bigr \}.
$$

(iii) For the finite Blaschke products $F(z) = z^d$, with $d \ge 2$, the counting function $n(x,T)$ is a step function, so the asymptotics in Theorem \ref{orbit-counting1a} do no hold. This exceptional behaviour is caused by the failure of D-genericity. Showing that other centered one component inner functions satisfy this mixing condition will be an important ingredient in the proof of Theorem \ref{orbit-counting1a}. 
\end{remark}

In Section~\ref{POCIF}, we define the class of parabolic one component inner functions and prove an analogue of Theorem~\ref{orbit-counting1a}.  It is convenient to view parabolic inner functions as holomorphic self-maps of the upper half-plane $\mathbb{H}$, in which case, the absolutely continuous invariant measure is just the Lebesgue measure $\ell$ on the real line $\mathbb{R}$. In this setting, the Orbit Counting Theorem reads as follows:

\begin{theorem}
\label{parabolic-orbit-counting}
Let $F: \mathbb{H} \to \mathbb{H}$ be a doubly parabolic one component inner function with a doubly parabolic fixed point at infinity. Suppose that
$$
\int_{ \mathbb{R}} \bigl ( \log |F'(z)| \bigl )^{1+\varepsilon} d\ell < \infty,
$$
for some $\varepsilon > 0$.
If $B \subset \mathbb{R}$ is a bounded Borel set with $\ell(B) < \infty$ and $\ell(\partial B) = 0$, then for any $x \in \mathbb{R}$,
$$
n(x, T, B) \sim \frac{e^T}{\int_{\mathbb{R}} \log |F'(z)| d\ell}  \cdot \ell(B) \qquad \text{as }T \to \infty.
$$
\end{theorem}

\subsection{Orbit Counting in Symbolic Dynamics}

Since our approach to orbit counting is very general, it is applicable to a plethora of other dynamical systems, not just to inner functions. We will deduce Theorems \ref{orbit-counting1a}, \ref{orbit-counting2a} and \ref{parabolic-orbit-counting} from a general orbit counting theorem, valid in the setting of countable alphabet subshifts of finite type.
A complete explanation of the terms will be given in  Section~\ref{sec:SFTTFSL}.

\begin{theorem}
\label{orbit-counting-tdf}
Let $E$ be a countable set, $A: E \times E \to \{0,1\}$ be a finitely irreducible incidence matrix and $\psi: E_A^\infty \to \mathbb{R}$ be a robust potential such that
$$
\int_{E_A^\infty} \psi^{1+\varepsilon} d\mu_\psi < + \infty,
$$
for some $\varepsilon > 0$.
Given $\xi \in E_A^\infty$ and a Borel set $B \subset E_A^\infty$ with $m_\psi(\partial B) = 0$, the counting function
$$
 N^{B}_\xi(T) := \# \bigl \{ \omega \in E^*_\xi:\ \omega \xi \in B \ \text {and }\,  S_{|\omega|} (-\psi) (\omega \xi) \le T \bigr \}
 $$
 satisfies
$$
\lim_{T \to \infty} \frac{N^{B}_\xi(T)}{e^T} = \frac{\rho_\psi(\xi)}{\int_{E_A^\infty} (-\psi) d\mu_\psi } \cdot m_\psi(B).
$$
\end{theorem}

\begin{theorem}
\label{orbit-counting-tdf2}
Let $E$ be a countable set, $A: E \times E \to \{0,1\}$ be a finitely irreducible incidence matrix and $\psi: E_A^\infty \to \mathbb{R}$ be a normal potential such that
$$
\int_{E_A^\infty} \psi d\mu_\psi < + \infty.
$$
Given $\xi \in E_A^\infty$ and a Borel set $B \subset E_A^\infty$ with $m_\psi(\partial B) = 0$, we have
$$
\lim_{T \to \infty} \frac{1}{T} \int_0^T \frac{N^{B}_\xi(t)}{e^t} dt = \frac{\rho_\psi(\xi)}{\int_{E_A^\infty} (-\psi) d\mu_\psi } \cdot m_\psi(B).
$$
\end{theorem}

\subsection{Notes and references}
 
\paragraph*{Blaschke-lacunary series.} In \cite{NS, NS2},  A.~Nicolau and O.~Soler i Gibert proved a Central Limit Theorem for ``Blaschke-lacunary series''
$$
\sum_{n=1}^N a_n F^{\circ n}(z),
$$
where $\{a_n\}_{n=0}^\infty$ is a sequence of complex numbers satisfying
\begin{equation}
\label{eq:blaschke-lacunary}
a_n \notin \ell^2, \qquad \lim_{N \to \infty} \frac{|a_N|^2}{\sum_{n=1}^N |a_n|^2} = 0.
\end{equation}
More precisely, they show that for an appropriate choice of scaling factors $\sigma_N$, the distribution of the functions
$$
\frac{\sqrt{2}}{\sigma_N} \cdot \sum_{n=1}^N a_n F^{\circ n}(z),
$$
on the unit circle converges to the standard complex Gaussian. As explained in \cite{NS2}, the condition (\ref{eq:blaschke-lacunary}) is sharp in the sense if (\ref{eq:blaschke-lacunary}) is not satisfied, then the Central Limit Theorem cannot hold.
The proof leverages that iterates of the observable $h(z) = z$ possesses a certain ``uniform indepedence property'' which is made precise by BMO estimates.
For further results in this direction, we refer the reader to \cite{linear-combinations, donaire-nicolau}.

\part{General Centered Inner Functions}
\label{GCIF}

\section{Background on Composition Operators}
\label{sec:composition-operators}

In this section, we assume that $F$ is a holomorphic self-map of the unit disk with $F(0) = 0$. We survey a number of results on the spectral properties of the composition operator $C_F: g \to g \circ F$ acting on various spaces of real and holomorphic functions. We first recall a classical lemma due to Koenigs from 1884:

\begin{lemma}
\label{holomorphic-eigenvalues}
Suppose $F$ is a holomorphic self-map of the unit disk with $F(0) = 0$.
If $F'(0) \ne 0$, then the eigenvalues of $C_F: \hol(\mathbb{D}) \to \hol(\mathbb{D})$ are 1 and $F'(0)^k$, $k = 1, 2, \dots$. All have multiplicity 1. 
On the other hand, if $F'(0) = 0$, then $C_F$ only has the constant eigenfunction ${\bf 1}$.
\end{lemma}

A proof can be found in \cite[Section 6.1]{shapiro-book}.
The Koenigs eigenfunction
$$
\varphi(z) = \lim_{n \to \infty} F'(0)^{-n} \cdot F^{\circ n}(z)
$$
is well known in complex dynamics. It serves as a linearizing coordinate near an attracting fixed point with a non-zero multiplier:
$$
\varphi(F(z)) = F'(0) \cdot \varphi(z), \qquad \varphi'(0) = 1.
$$
 The other eigenfunctions are powers of the Koenigs eigenfunction.

\subsection{Hardy Spaces}

The Hardy space $H^p$  consists of holomorphic functions on the unit disk which satisfy
$$
\| g \|_{H^p}^p = \sup_{0 < r < 1} \frac{1}{2\pi} \int_{0}^{2\pi} |g(re^{i\theta})|^p d\theta < \infty.
$$
To study the action of $C_F$ on the space $H^2$, one uses the following two facts, which one can find in \cite[Sections 2 and 3]{shapiro-inner}:

\begin{lemma}[Littlewood-Paley formula]
\label{littlewood-paley}
For any function $g \in H^2$,
\begin{equation}
\| g\|_{H^2}^2 =  |g(0)|^2 + \frac{2}{\pi} \int_{\mathbb{D}} |g'(z)|^2 \log \frac{1}{|z|} |dz|^2.
\end{equation}
\end{lemma}

\begin{lemma}
\label{nevanlinna}
Let $F:\mathbb{D}\to\mathbb{D}$ be a holomorphic self-map of the unit disk which vanishes at the origin and
$$
N_{F}(z) = \sum_{F(w) = z} \log \frac{1}{|z|}, \qquad  z\in \mathbb{D},
$$
be its {\em Nevanlinna counting function}. We have:

{\em (i)} The inequality $N_F(z) \le \log \frac{1}{|z|}$ holds for any $z \in \mathbb{D}$.
 
{\em (ii)} If $F$ is an inner function, then $N_F(z) = \log \frac{1}{|z|}$ outside a set of logarithmic capacity zero.

{\em (iii)} If the equality $N_F(z) = \log \frac{1}{|z|}$ holds at a single point, then $F$ is an inner function.
\end{lemma}

Using Lemmas \ref{littlewood-paley} and \ref{nevanlinna}, it is easy to see that $C_F$ is bounded on $H^2$:
\begin{align*}
\| C_F g \|^2_{H^2} 
& = |g(0)|^2 + \frac{2}{\pi} \int_{\mathbb{D}} |(g \circ F)'(z)|^2 \log \frac{1}{|z|} |dz|^2 \\
& =  |g(0)|^2 + \frac{2}{\pi} \int_{\mathbb{D}} |g'(z)|^2 N_F(z) |dz|^2 \\
 & \le  |g(0)|^2 + \frac{2}{\pi} \int_{\mathbb{D}} |g'(z)|^2 \log \frac{1}{|z|} |dz|^2 \\
 & = \| g \|^2_{H^2}.
\end{align*}
If $F$ is an inner function and $g(0) = 0$, then the inequalities in the computation above are equalities and we have 
$\| C_F g \|^2_{H^2} = \| g \|^2_{H^2}$. As explained in \cite[Theorem 5.1]{shapiro-inner}, if $F$ is not inner,
then $C_F g$ is a strict contraction on the orthogonal complement of the constant functions:

\begin{lemma}
\label{contraction-on-h2}
Let $F$ be a holomorphic self-map of the unit disk with $F(0) = 0$. If $F$ is an inner function, then $C_F$ is an isometry on $H^2$. Otherwise, $F$ is a strict contraction on $H^2_0 = \{ g \in H^2 : g(0) = 0 \}$, i.e.~$\| C_F \|_{H^2_0 \to H^2_0} < 1$.
\end{lemma}

The {\em spectrum} $\sigma$ of an operator consists of all $\lambda \in \mathbb{C}$ for which $\lambda I - F$ does not have a bounded inverse.
The {\em spectral radius} is the radius of the smallest closed disk centered at the origin which contains the spectrum. Since $C_F$ is a bounded operator acting on a Banach space, its spectral radius may be computed by the formula 
$$
r_{H^2}(C_F) = \lim_{n\to\infty}\| C_F^n \|^{1/n}. 
$$

We are more interested in the {\em essential spectum} of $C_F$, which is the part of the spectrum that survives under compact perturbations. More precisely,
$$
\sigma_{e}(C_F) := \bigcap_K \sigma(C_F - K),
$$
where $K$ ranges over compact operators. The {\em essential spectral radius} is the radius of the smallest disk centered at the origin which contains the essential spectrum.
Outside $\overline{B(0, r_{e}(C_F))}$, the spectrum $\sigma(C_F)$ can consist only of isolated eigenvalues of finite multiplicity. 

In the literature, one often says that an operator has a {\em spectral gap} if its essential spectral radius $r_e$ is strictly smaller than its spectral radius $r$. For our purposes, it is preferable to use a more restrictive definition: we say that an operator has a spectral gap if it has a simple eigenvalue $\lambda$ with $|\lambda| = r$, and the rest of the spectrum is contained in a ball $B(0, \rho)$ for some $\rho < r$.

The {\em essential norm} of $C_F$ is defined as
$$
\| C_F \|_{e,\, H^2 \to H^2} := \inf_K \bigl \{\| C_F - K \|_{H^2 \to H^2} \bigr \}.
$$
The spectral radius formula shows that $r_{e,H^2}(C_F) \le \| C_F \|_{e, H^2 \to H^2}$.   In an important work \cite{shapiro}, J.~H.~Shapiro computed the essential norm of $C_F$ when acting on $H^2$:

\begin{theorem}
\label{shapiro-thm}
If $F$ is a holomorphic self-map of the unit disk, then the essential norm of the composition operator $C_F$ acting on $H^2$ is 
$$\| C_F \|_{e,\, H^2 \to H^2}^2 = \limsup_{|z| \to 1} \frac{N_F(z)}{\log(1/|z|)}.$$
\end{theorem}

As a consequence of Lemma \ref{contraction-on-h2}, we see that the spectral properties of the composition operator $C_F$ on $H^2$ can detect whether or not a holomorphic self-map of the unit disk is an inner function:
\begin{corollary}
\label{spectral-gap-on-H^2}
Let $F$ be a holomorphic self-map of the unit disk with $F(0) = 0$. If $F$ is an inner function, then $C_F$ is an isometry on $H^2$. Otherwise, $$r_{e,\, H^2}(C_F) \, \le \, \| C_F \|_{e,\, H^2 \to H^2} \, < \, 1$$ and the eigenvalue 1 with the eigenfunction ${\bf 1}$ is isolated and simple.
\end{corollary}

For a general exponent $1 \le p < \infty$, Bourdon and Shapiro  \cite{bourdon-shapiro} showed the following theorem which implies that Corollary \ref{spectral-gap-on-H^2} holds for all Hardy spaces $H^p$ with $1 \le p < \infty$:
\begin{theorem}
\label{bourdon-shapiro}
If $F$ is a holomorphic self-map of the unit disk, then
$$
r_{e, H^p}(C_F)^p \le r_{e, H^2}(C_F)^2, \qquad 1 \le p < \infty,
$$
with equality holding whenever $1 < p < \infty$.
\end{theorem}

\subsection{Weighted Bergman Spaces}

While the results of the previous section are quite elegant, for the purposes of studying the dynamics of inner functions, they are slightly disappointing since the composition operators $C_F$ do not exhibit a spectral gap in this case. We therefore turn our attention to Bergman spaces.

For $\alpha > -1$, the weighted Bergman space $A^p_\alpha$ consists of holomorphic functions that satisfy
$$
\| g \|^p_{A^p_\alpha} \, = \,  \int_{\mathbb{D}} |g(z)|^p \, dA_\alpha \, < \, \infty, \qquad dA_\alpha = 
\frac{1}{\pi} (\alpha+1) (1-|z|^2)^\alpha |dz|^2.
$$
As in \cite{HKZ}, we normalize the area form $dA_\alpha$ so that $\| {\bf 1} \|^2_{A^p_\alpha} = 1$. When the exponent $p=2$,
the above condition can be rephrased in terms of the coefficients of the power series expansion at 0:
\begin{equation}
\label{eq:bergman-power-series}
\| g \|^2_{A^2_\alpha} \, = \, \sum_{n=0}^\infty \frac{n! \, \Gamma(2+\alpha)}{\Gamma(n+2+\alpha)} |a_n|^2
\, \asymp \, \sum_{n=0}^\infty \frac{|a_n|^2}{n^{1+\alpha}} \, < \, \infty.
\end{equation}

For weighted Bergman spaces, the spectral properties of $C_F$ are related to the behaviour of the weighted Nevanlinna counting function
 $$
N_{F, \gamma}(z) = \sum_{F(w) = z} \biggl ( \log \frac{1}{|z|} \biggr )^\gamma, \qquad \text{with }\gamma=\alpha+2.
$$
As explained in \cite[Section 6]{shapiro}, the proof of the upper bound in Theorem \ref{shapiro-thm} can be easily adapted to weighted Bergman spaces.
Several years later, P.~Poggi-Corradini \cite{pietro} showed that Shapiro's bound was sharp  for $\alpha = 0, 1$:
\begin{theorem}\label{t120230206}
If $F$ is a holomorphic self-map of the unit disk and $\alpha > -1$, then the essential spectral radius of the composition operator $C_F$ acting on $A^2_\alpha$ satisfies 
$$\| C_F \|^2_{e,\, A^2_\alpha \to A^2_\alpha} \le \limsup_{|z| \to 1} \frac{N_{F,\alpha+2}(z)}{\log(1/|z|)^{\alpha+2}}.$$
Equality holds for $\alpha = 0, 1$.
\end{theorem}

Using the above theorem, it is easy to see that if $F$ is not a rotation, then $C_F$ has a spectral gap:

\begin{corollary}
\label{spectral-gap-bergman-space}
Let $F$ be a holomorphic self-map of the unit disk with $F(0) = 0$, which is not a rotation. For $-1 < \alpha < \infty$, the essential spectral radius of the composition operator $C_F$ acting on $A^2_\alpha$ is strictly less than 1, while the eigenvalue 1 with eigenfunction ${\bf 1}$, is isolated and simple.
\end{corollary}

\begin{proof}
The proof is based on the following elementary observation: if $\{a_i\}$ is a collection of positive real numbers with $\sum a_i \le x$, then 
$$
\sum a_i^\gamma \le x^\gamma, \qquad \gamma = \alpha + 2 > 1.
$$
In addition, equality holds if and only if one of the $a_i$'s is equal to $x$ and all the remaining $a_i$ are equal to $0$. If one can guarantee that $a_i \le \theta x$ for some $0 < \theta < 1$ and all $i\in\mathbb N$, then one has the stronger estimate 
$$
\sum_{i\in\mathbb N} a_i^\gamma\le \bigl \{ \theta^\gamma + (1-\theta)^\gamma \bigr \} x^\gamma.
$$

In light of this observation and Theorem~\ref{t120230206}, to show that $r_{e, A^2_\alpha}(C_F) < 1$, it suffices to prove that
\begin{equation}
\label{eq:poggi-gap}
\limsup_{|z| \to 1,\, F(w) = z} \frac{\log(1/|w|)}{\log(1/|z|)} < 1.
\end{equation}
Since $F$ is not a hyperbolic isometry,
$$
\gamma \, := \, 1 - \sup_{w_1 \in \partial B_{\hyp}(0,1)} \big\{d_{\mathbb{D}}(0,F(w_1))\big\} \, > \, 0,
$$
where $d_{\mathbb{D}}(\cdot, \cdot)$ denotes hyperbolic distance on ${\mathbb{D}}$. For a point $w \in \mathbb{D}$ with $d_{\mathbb{D}}(0,w) \ge 1$,
let $w_1$ be the point on the hyperbolic geodesic $[0, w]$ which is located a hyperbolic distance 1 from the origin.
By the triangle inequality and the Schwarz lemma,
\begin{align*}
d_{\mathbb{D}}(0,F(w)) & \le  d_{\mathbb{D}}(0, F(w_1)) + d_{\mathbb{D}}(F(w_1), F(w)) \\
& \le  (1-\gamma) +  ( d_{\mathbb{D}}(0,w) - 1) \\
& = d_{\mathbb{D}}(0,w) - \gamma,
\end{align*}
which implies that
$$
1-|F(w)| < c(1-|w|), \qquad w \notin B_{\hyp}(0,1),
$$
for some $0 < c < 1$. The inequality (\ref{eq:poggi-gap}) follows since $\log(1/|w|) \sim 1-|w|$ as $|w| \to 1$. 

Since we have proved that the essential spectral radius of $C_F$ is less than 1, the eigenvalue 1 is isolated and simple by Lemma \ref{holomorphic-eigenvalues}.
\end{proof}

A similar situation holds for the weighted Bergman spaces $A^p_\alpha$ with general exponents $p$.
In \cite[Theorem 5 and Proposition 7]{maccluer-saxe}, MacCluer and Saxe proved an analogue of Theorem \ref{bourdon-shapiro} for the unweighted Bergman spaces $A^p$. Their argument was generalized to the weighted case in \cite[Theorem 2.8]{powers-composition}:

\begin{theorem}
\label{maccluer-saxe}
Let $F$ be a holomorphic self-map of the unit disk. The essential spectral radius of $C_F$ acting on $A^p_\alpha$ satisfies
$$
r_{e,\, A^p_\alpha}(C_F)^p \le r_{e,\, A^2_\alpha}(C_F)^2, \qquad 1 \le p < \infty,
$$
with equality for $p > 1$. In particular, unless $F$ is a rotation, then
$$
r_{e,\, A^p_\alpha}(C_F) < 1,  \qquad 1 \le p < \infty.
$$
\end{theorem}

From the above theorem, it follows that Corollary~\ref{spectral-gap-bergman-space} also holds for the weighted Bergman spaces $A^p_\alpha$ with $\alpha > -1$ and $1 < p < \infty$.
 
\subsection{Sobolev Spaces}
\label{sec:sobolev}

Below, we identify distributions on the unit circle with their harmonic extensions to the unit disk.
Since the composition operator $C_F$ is a {\em real} operator, it preserves the anti-holomorphic spaces $\overline{H^2}$ and $\overline{A^2_\alpha}$ in addition to their holomorphic counterparts. It follows that $C_F$ also preserves 
$L^2 = H^2 + \overline{H^2}$ and the Sobolev spaces $W^{-\beta,2} = (W^{\beta,2})^* = A^2_\alpha + \overline{A^2_\alpha}$ for any $\beta = (1+\alpha)/2 > 0$. The intersection $H^2 \cap \overline{H^2} = A^2_\alpha \cap \overline{A^2_\alpha} = \mathbb{C} \cdot {\bf 1}$ consists of the constant functions. One can check these decompositions by comparing power series expansions.
 
\section{Transfer Operators}
\label{sec:transfer-operators}

For a finite Blaschke product $F$ with $F(0) = 0$, the {\em transfer operator} with potential $-\log|F'|$ acts on functions on the unit circle by
\begin{equation}
\label{eq:classical-transfer}
(L_F g) (x) = \sum_{F(y) = x} |F'(y)|^{-1} g(y).
\end{equation}
Since the Lebesgue measure on the unit circle is invariant under $F$, 
$$
L_F {\bf 1} = {\bf 1}.
$$ 
The invariance of the Lebesgue measure also shows that
the transfer operator is the adjoint of the the composition operator with respect to the $L^2$-pairing on the unit circle:
\begin{equation}
\int_{\partial \mathbb{D}} L_F g(x) \overline{h(x)} dm(x) = \int_{\partial \mathbb{D}} g(x) \overline{h(F(x))} dm(x).
\end{equation}
For a general inner function $F$, we simply {\em define} the transfer operator as the adjoint of the composition operator.

In this section, we describe a number of spaces on which $L_F$ acts, represent $L_F$ in terms of Aleksandrov-Clark measures and discuss a few simple consequences of the spectral gap such as the exponential decay of correlations.

\subsection{Duals of some Function Spaces}
\label{sec:duals-of-some-function-spaces}

We now describe the duals of some function spaces of harmonic and holomorphic functions with respect to the $L^2$-pairing on the unit circle.
As in Section \ref{sec:sobolev}, we identify functions $f, g$ in the unit disk with their boundary distributions $\tilde f, \tilde g$ on the unit circle.
The $L^2$-pairing can be interpreted in terms of Fourier series expansions:
$$
\langle \tilde f, \tilde g \rangle \, = \, \Bigl \langle \sum a_n e^{in\theta}, \, \sum b_m e^{im\theta} \Bigr \rangle \,= \, \sum a_n \overline{b_n}. 
$$
Alternatively, one can dilate the functions $f, g$ and take the limit as $r \to 1$:
$$
\langle f, g \rangle = \lim_{r \to 1} \int_{|z|=1} f(rz) \overline{g(rz)} \, |dz|.
$$

\paragraph*{Hardy spaces.} Recall from Theorem \ref{bourdon-shapiro} that the composition operator preserves the Hardy space $H^p$ for any $1 \le p < \infty$. Therefore, the transfer operator preserves $(H^p)^* \cong H^q$ where $1/p + 1/q = 1$. Since the transfer operator is a {\em real} operator, it must preserve $\overline{H^q}$ as well. In particular, the transfer operator preserves $L^q = H^q + \overline{H^q}$ for $1 < q \le \infty$.

\paragraph*{Weighted Bergman spaces.} From Theorem \ref{maccluer-saxe}, it follows that the composition operator $C_F$ has a spectral gap on the weighted Bergman space $A^p_\alpha$ for any $1 \le p < \infty$ and $\alpha > -1$. As explained in \cite[Section 1.4]{HKZ}, for $p > 1$, the dual of $A^p_\alpha$ with respect to the pairing $\langle \cdot, \cdot \rangle_{A^p_\alpha}$ is $A^q_\alpha$ where $q$ is the conjugate exponent, while for $p = 1$, the dual is the Bloch space $\mathcal B$. To obtain the dual with respect to the $L^2$-pairing on the circle, one needs to take a fractional integral of order $1 + \alpha$, i.e.~
$$
(A^p_\alpha)^* \cong {\bf I}_{1 + \alpha} (A^q_\alpha) \quad \text{for }p > 1, \qquad 
(A^1_\alpha)^* \cong {\bf I}_{1 + \alpha}(\mathcal B) \quad \text{for }p =1.
$$
Below, we consider the cases when $p =1,2$ in detail, where the dual spaces admit particularly simple descriptions.

 \medskip

{\em $p=2$.} Recall that in terms of power series expansions, the space $A^2_\alpha$ consists of holomorphic functions on the unit disk such that
$$
\sum_{n=0}^\infty n^{-1-\alpha} |a_n|^2 < \infty.
$$
Its dual is therefore the Dirichlet-type space $\mathcal D_{1+\alpha}$ which consists of holomorphic functions on the unit disk such that
$$
\sum_{n=0}^\infty n^{1+\alpha} |a_n|^2 < \infty.
$$
Again, using the fact that the transfer operator is a real operator, we see that it preserves the Sobolev space $W^{(1+\alpha)/2,2} = \mathcal D_{1+\alpha} + \overline{{\mathcal D}_{1+\alpha}}$. Since adjoint operators have the same spectral and essential spectral radii, $L_F$ has a spectral gap on $\mathcal D_{1+\alpha}$  and $W^{(1+\alpha)/2,2}(\partial \mathbb{D})$.  
Furthermore, as the adjoint operator has the same isolated eigenvalues with the same multiplicity, 1 is a simple isolated eigenvalue with eigenfunction ${\bf 1}$ in each of those spaces.

\smallskip

\begin{remark} 
The standard Dirichlet space $\mathcal D = \mathcal D_1$ corresponds to the case when $\alpha = 0$ and consists of holomorphic functions such that
$$
\int_{\mathbb{D}} |f'(z)|^2 dA(z) < \infty.
$$
The associated Sobolev space is $W^{1/2,2} = \mathcal{D} + \overline{\mathcal D}$.
\end{remark}

\medskip

{\em $p =1$.} By  \cite[Theorem 7.6]{zhu}, for $0 < \beta < 1$, the dual space
 $$(A^1_{\beta-1})^* \cong \hol(\mathbb{D}) \cap C^{\beta}(\partial \mathbb{D})$$
consists of holomorphic functions on the unit disk that admit H\"older continuous extensions to the closed unit disk. The same theorem identifies
 $$(A^1)^* \cong \hol(\mathbb{D}) \cap Z(\partial \mathbb{D})$$
with the space of holomorphic functions on the unit disk that admit Zygmund extensions. More generally, if we take duals of $A^1_\alpha$ with $\alpha > 0$, we get the holomorphic functions in 
$C^{n,\beta}$ for some integer $n \ge 0$ and $0 < \beta < 1$, as well as in $C^{n, Z}$, the space of $n$-fold integrals of Zygmund functions.
As above, the transfer operator has a spectral gap when acting on these spaces and their real-valued counterparts.

\subsection{Aleksandrov-Clark Measures}

Let $F$ be a holomorphic self-map of the unit disk. 
For $\alpha \in \partial \mathbb{D}$, consider the function
$$
h_\alpha(z) = \frac{\alpha+F(z)}{\alpha-F(z)}, \qquad z \in \mathbb{D}.
$$
Since $h_\alpha$ is holomorphic and has positive real part, it can be represented as the Herglotz extension of a positive measure on the unit circle plus a constant term:
\begin{equation}
\label{eq:ac-def}
h_\alpha(z) = \frac{\alpha+F(z)}{\alpha-F(z)} = \int_{\partial \mathbb{D}} \frac{\zeta+z}{\zeta-z} d\mu_\alpha(\zeta) + i \cdot \im h_\alpha(0).
\end{equation}
The collection of measures $\{ \mu_\alpha \}_{\alpha \in \partial \mathbb{D}}$ are known as the {\em Aleksandrov-Clark measures} of $F$. For basic properties of Aleksandrov-Clark measures, we refer the reader to \cite{cima, PS, saksman}.

\begin{itemize}
\item With the normalization $F(0) = 0$, each measure $\mu_\alpha$ has unit mass. When $F$ is an inner function, for each $\alpha \in \partial \mathbb{D}$, $\re h_\alpha$ has radial limit zero a.e.~on the unit circle, which means that the measure $\mu_\alpha$ is singular. 

\item The Aleksandrov-Clark measure $\mu_\alpha$ is concentrated on the set of points on the unit circle where $F$ has radial limit $\alpha$. For instance, 
$\mu_\alpha$ has a point mass at $\zeta \in \partial \mathbb{D}$ if and only if $F$ has an angular derivative at $\zeta$ with $F(\zeta) = \alpha$, in which case, $$\mu_\alpha(\{\zeta\}) = |F'(\zeta)|^{-1}.$$
In general, $\mu_\alpha$ may not be a discrete measure, i.e.~be a countable sum of point masses. In fact, there are holomorphic self-maps of the unit disk, even inner functions, that do not possess an angular derivative at any point on the unit circle.
 
\item The Aleksandrov-Clark measures $\{ \mu_\alpha \}$ vary continuously in $\alpha$ in the weak-$*$ topology. More generally, if $F_n \to F$ locally uniformly and $\alpha_n \to \alpha$ on $\partial \mathbb{D}$, then $\mu_{F_n, \alpha_n} \to \mu_{F, \alpha}$ weak-$*$.

\item Aleksandrov's disintegration theorem says that for any continuous function $f(\zeta)$ on the unit circle,
$$
\int_{\partial \mathbb{D}} \biggl ( \int_{\partial \mathbb{D}}  f(\zeta) d\mu_\alpha(\zeta) \biggr ) dm(\alpha) = \int_{\partial \mathbb{D}} f(\zeta)dm(\zeta).
$$
Equivalently, the family $\{ \mu_\alpha \}_{\alpha \in \partial \mathbb{D}}$ constitutes the canonical disintegration of the Lebesgue measure $m$ with respect to the measurable partition $F^{-1}(\epsilon)$, where $\epsilon$ denotes the partition of $\partial \mathbb{D}$ into singletons.
\end{itemize}

With help of the Aleksandrov-Clark measures, one can write down an explicit formula for the adjoint of the composition operator:
\begin{equation}
\label{eq:modern-transfer}
(L_F g)(\alpha) = \int_{\partial \mathbb{D}} g(z) d\mu_\alpha(z).
\end{equation}
Indeed, the above formula simplifies to (\ref{eq:classical-transfer}) for finite Blaschke products, while the general case follows from approximation. 

\subsection{Comparing Definitions} 

In the following theorem, we characterize inner functions for which the {\em adjoint} definition of the transfer operator coincides with the {\em classical} definition of the transfer operator (\ref{eq:classical-transfer}) a.e.
  
    \begin{theorem}
  \label{finite-angular-derivative}
Suppose $F$ is a centered inner function. The Alexandrov-Clark measures $\{\mu_{\alpha}\}$ are discrete for almost every $\alpha \in \partial \mathbb{D}$ if and only if $F$ has an angular derivative at a.e.~point on the unit circle.
  \end{theorem}
  
  For the proof, we will need the following lemma due to Craizer \cite[Lemma 5.4]{craizer}:
  
  \begin{lemma}
  \label{good-approximation}
If a centered inner function $F$ has an angular derivative at a.e.~point on the unit circle, then there exists a sequence of finite Blaschke products $\{ F_n \}_{n=1}^\infty$ converging to $F$ uniformly on compact subsets of the unit disk so that
$$
|F'_n(x)| < 2 |F'(x)|,
$$
for a.e.~$x\in \partial \mathbb{D}$ and all integers $n\ge 1$.
\end{lemma}
  
\begin{proof}[Proof of Theorem   \ref{finite-angular-derivative}]
    Suppose $F$ has an angular derivative at a.e.~point on the unit circle.
    We want to show that for almost every $\alpha \in  \partial \mathbb{D}$,
      $$
  \sum_{\substack{\zeta \in F^{-1}(\alpha)\\  |F'(\zeta)| < \infty}} \frac{1}{|F'(\zeta)|} = 1.
  $$
Let $F_n \to F$ be a sequence of finite Blaschke products from Lemma  \ref{good-approximation}.  Fix an $\varepsilon > 0$.
  Since $F$ has angular derivatives a.e., we can choose $N > 0$ sufficiently large so that
  $$
  m \bigl ( \{ |F'| > N/2 \} \bigr ) \le \varepsilon^2,
  $$
  which in turn, implies that
   \begin{equation}
   \label{eq:good-points}
  m \bigl ( \{ |F_n'| > N \} \bigr ) \le \varepsilon^2,
  \end{equation}
  for all $n \ge 1$.
  
   Call a point $\alpha \in \partial \mathbb{D}$ $(\varepsilon, N)$-{\em good} for $F_n$ if most of the mass of  $\mu_{F_n, \alpha}$ comes from atoms of size $\ge 1/N$\,:
  \begin{equation}
  \label{eq:many-large-atoms}
  \sum_{\substack{\zeta \in F_n^{-1}(\alpha) \\ |F'_n(\zeta)| \le N}} \frac{1}{|F_n'(\zeta)|} \ge 1 - \varepsilon.
  \end{equation}  
It follows from (\ref{eq:good-points}) that  $F_n$ is $(\varepsilon, N)$-good on a set $\mathcal G_n$ of Lebesgue measure $\ge 1-\varepsilon$. 
Let $\mathcal G$ be the set of points on the unit circle that lie in infinitely many $\mathcal G_n$. As
     $\mu_{F_n,\alpha} \to \mu_{F, \alpha}$ weakly, $m(\mathcal G) \ge 1-\varepsilon$ and almost every $\alpha \in \mathcal G$ is $(\varepsilon, N)$-good for $F$. For these $\alpha \in \partial \mathbb{D}$, the discrete part of $\mu_{F,\alpha}$ has mass at least $\ge 1 - \varepsilon$. The lemma follows since $\varepsilon > 0$ was arbitrary.
    \end{proof}
    
    For another proof, see \cite[Theorem 9.6.1]{cima}.

\begin{corollary}
Suppose $F$ is a centered inner function. If $F$ has an angular derivative at a.e.~point on the unit circle, then $F$ is essentially countable-to-one, i.e.~the set $F^{-1}(\alpha)$ is countable for a.e. $\alpha \in \partial \mathbb{D}$.
\end{corollary}

 \subsection{Equidistribution}
We keep assuming that $F$ is a centered inner function. Let $\alpha \in \partial \mathbb{D}$ be a point on the unit circle. The following lemma says that as $n \to \infty$, the pre-images $F^{-n}(\alpha)$ become equidistributed on the unit circle with respect to Lebesgue measure.
A similar statement holds for the periodic points $\{ x \in \partial \mathbb{D}: F^{\circ n}(x) = x \}$ of period $n$.

\begin{lemma}
\label{equidistribution-lemma}
For an inner function $F$ with $F(0) = 0$, we have:

{\em (i)} For every $\alpha \in \partial \mathbb{D}$ the sequence $\{ \mu_{F^{\circ{n}}, \alpha} \}_{n=1}^\infty$ converges weakly to the Lebesgue measure $m$  on the unit circle.

{\em (ii)} As $n \to \infty$, the measures $\{ \mu_{F^{\circ{n}}(z)/z, 1} \}_{n=1}^\infty$ converge weakly to the Lebesgue measure $m$ on the unit circle.
\end{lemma}

\begin{proof}
We only prove (i) as the proof of (ii) is similar. According to the definition of the Aleksandrov-Clark measures,
$$
\frac{\alpha+F^{\circ n}(z)}{\alpha-F^{\circ n}(z)} = \int_{\partial \mathbb{D}} \frac{\zeta+z}{\zeta-z} \, d\mu_{F^{\circ{n}}, \alpha}(\zeta).
$$
Since 0 is an attracting fixed point, as $n \to \infty$, the left hand side tends to 1 for any $z \in \mathbb{D}$. In turn, this implies that $\mu_{F^{\circ{n}}, \alpha}$ converge weakly to the Lebesgue measure.
\end{proof} 

\begin{corollary}
If $F$ is a centered inner function, then the set of its periodic points is dense in the unit circle.
\end{corollary}

 \subsection{Consequences of the Spectral Gap}

We have seen in Section \ref{sec:duals-of-some-function-spaces} that if $F$ is a centered inner function, then $L_F$ has a spectral gap on the spaces $W^{\beta,2}$ with $\beta > 0$ and $C^\alpha$ with $\alpha > 0$ 

Let $R$ be the Riesz projection onto the one-dimensional eigenspace spanned by the eigenfunction {\bf 1} (corresponding to the eigenvalue $1$) and $S = \id - R$ be the complementary Riesz projection onto the rest of the spectrum. From the general properties of Riesz projections, we know that:
\begin{itemize}
\item The Riesz projections $R$ and $S$ commute with $L_F$.
\item $RS = SR = 0$.
\item $\sigma(L_F |_{\im R}) = \{1\}$ and $\sigma(L_F |_{\im S}) = \sigma(L_F) \setminus \{ 1 \}$.
\end{itemize}
We write $L_F = L_F (R + S) = R + \Delta$ where $\Delta = L_F S$. As the spectral radius $r(\Delta) < 1$, we have
$$
\| \Delta^n g \| \le C\theta^n,
$$ 
for some constants $\theta\in (0,1)$, $C \in (0, +\infty)$ and all integers $n \ge 1$. Integrating both sides of the equality
$$
L_F^n g = Rg + \Delta^n g
$$
over the unit circle, and then taking $n \to \infty$, we get
$$
\lim_{n \to \infty} \int_{\partial \mathbb{D}} L_F^n g \, dm = Rg.
$$
Since $\int_{\partial \mathbb{D}} L_F^n g \cdot {\bf 1} \, dm = \int_{\partial \mathbb{D}} g \cdot C_F^n {\bf 1} \, dm$, the left hand side is simply $\int_{\partial \mathbb{D}} g dm$. Therefore, the Riesz projection is given by
$$
Rg = \overline{g} \cdot {\bf 1}, \qquad \text{where } \overline{g} =  \int_{\partial \mathbb{D}} g dm.
$$
We conclude that on the set of mean zero functions, $L_F = \Delta$ is a contraction.

We state the following lemma for Sobolev spaces, but it is also true for H\"older continuous functions with the same proof:

\begin{lemma}
\label{L-contraction}
For a centered inner function $F$, we have:

{\em (i)} For any $g  \in W^{\beta,2}$, the sequence of functions $ \{ L^n_F g \}_{n=1}^\infty$ converges to the constant function $\int_{\partial \mathbb{D}} g(x) dm$ in $W^{\beta,2}$.

{\em (ii)}
 There exist  constants $\theta\in (0,1)$ and $C \in (0,+\infty)$ such that for any two functions $g,h \in W^{\beta,2}$  with $\int_{\partial \mathbb{D}} g dm = 0$ or $\int_{\partial \mathbb{D}} h dm = 0$,
$$
\biggl |\int_{\partial \mathbb{D}}  g (h\circ F^{\circ n}) dm \biggr | < C \theta^n.
$$
\end{lemma}

Since $W^{\beta,2}$ is dense in $L^2(\partial \mathbb{D}, m)$, the above lemma implies that $F$ is mixing with respect to the Lebesgue measure $m$:

\begin{corollary}
\label{L-contraction2}
For centered inner function $F$, we have:

{\em (i)} For any $g  \in L^2(\partial \mathbb{D}, m)$, the sequence of functions $ \{ L^n_F g \}_{n=1}^\infty$ converges to the constant function $\int_{\partial \mathbb{D}} g(x) dm$ in $L^2(\partial \mathbb{D}, m)$.

{\em (ii)} For any $g,h \in L^2(\partial \mathbb{D}, m)$  with $\int_{\partial \mathbb{D}} g dm = 0$ or $\int_{\partial \mathbb{D}} h dm = 0$,
$$
\biggl |\int_{\partial \mathbb{D}}  g (h\circ F^{\circ n}) dm \biggr |  \to 0, \qquad \text{as }n \to \infty.
$$

\end{corollary}

\subsection{Composition Operators Acting on Measures}
\label{sec:composition-operators-on-measures}

We have seen that the composition operator $C_F$ acts isometrically on the space $L^2(\partial \mathbb{D}, m)$. We can extend $C_F$ to a continuous operator on the Borel probability measures on the unit circle equipped with the weak-$*$ topology by specifying the integrals of continuous functions:
\begin{equation}
\label{eq:composition-operator-on-measures}
\int_{\partial \mathbb{D}} \phi \, d(C_F\nu) \, = \, 
\int_{\partial \mathbb{D}} \biggl \{ \int_{F^{-1}(\alpha)} \phi d\mu_\alpha \biggr \} d\nu(\alpha) 
\, = \, \int_{\partial \mathbb{D}} L_F \phi \, d\nu.
\end{equation}
A monotone class argument similar to the one in \cite[Chapter 9.4]{cima} shows that the above formula is valid for any bounded Borel function $\phi$.

\begin{lemma}
\label{h-out}
For any bounded Borel function $h$ and Borel measure $\nu$ on $\partial \mathbb{D}$, we have $C_F(h \nu) = (h \circ F) \cdot C_F(\nu).$
\end{lemma}

\begin{proof}
By (\ref{eq:composition-operator-on-measures}), for any bounded Borel function $\phi$, we have
$$
\int_{\partial \mathbb{D}} \phi \cdot d C_F(h \nu) = \int_{\partial \mathbb{D}} L_F \phi \cdot h d\nu
$$
and
$$
\int_{\partial \mathbb{D}} \phi \cdot (h \circ F) \cdot dC_F(\nu) = \int_{\partial \mathbb{D}} L_F [\phi \cdot (h \circ F)] d\nu.
$$
Thus, to prove the lemma, it suffices to show that
$$
L_F [\phi \cdot (h \circ F)] = L_F \phi \cdot h.
$$
Indeed,
\begin{align*}
L_F [\phi \cdot (h \circ F)](\alpha) & = \int_{\partial \mathbb{D}} \phi(x) (h \circ F(x)) d\mu_\alpha(x) \\
& = h(\alpha)  \int_{\partial \mathbb{D}} \phi(x) d\mu_\alpha(x)  \\
& = (L_F \phi \cdot h)(\alpha),
\end{align*}
where in the second step, we used that the Aleksandrov-Clark measure $\mu_\alpha$ is supported on the set of points $x$ where $F(x) = \alpha$.
\end{proof}

Let $\mathcal R_F \subset \partial \mathbb{D}$ be the set of points where  $\lim_{r \to 1} F(r \zeta)$ exists and is unimodular. By \cite[Theorem 2.6]{CL},  $\mathcal R_F$ is an $F_{\sigma \delta}$ set and therefore Borel.
Plugging $\phi = \chi_{\mathcal R_F^c}$ into (\ref{eq:composition-operator-on-measures}) and using the fact that Aleksandrov-Clark measures do not charge $\mathcal R_F^c$ shows the following lemma:

\begin{lemma}
\label{composition-operators-and-radial-limits}
If $F$ a centered inner function, then
for any  Borel probability measure $\nu$ on the unit circle, the measure $C_F \nu$ is supported on $\mathcal R_F$.
\end{lemma}

\begin{lemma}
\label{dynamically-meaningful}
Let $F$ be a centered inner function and $\nu$ be a probability measure on the unit circle such that $C_F \nu = \nu$. Then,
$\nu$ is an $F$-invariant measure supported on the set of points whose forward orbits lie on the unit circle.
\end{lemma}

\begin{proof}
Let $E \subset \partial \mathbb{D}$ be a Borel set and $\phi = \chi_{F^{-1}(E)}$. By the integral representation (\ref{eq:modern-transfer}),
$L_F \phi = \chi_E$. Plugging $\phi$ into (\ref{eq:composition-operator-on-measures}), we get
$$
\nu(F^{-1}(E)) \, = \, C_F \nu (F^{-1}(E)) \, = \, \nu(E),
$$ which means that $\nu$ is $F$-invariant. By Lemma \ref{composition-operators-and-radial-limits}, $\nu(\mathcal R_F^c) = C_F \nu(\mathcal R_F^c) = 0$. From $F$-invariance, it follows that $\nu(F^{-n}(\mathcal R_F^c)) = 0$ for any $n \ge 0$.
\end{proof}

\section{Perturbative  Thermodynamic Formalism and the Central Limit Theorem for Centered Inner Functions}
\label{sec:CLT}

In this section, we use weighted composition operators to establish the rudiments of thermodynamic formalism for inner functions:
\begin{enumerate}
\item For a $W^{1/2,2}$ or $C^\alpha$ potential of sufficiently small norm, we construct conformal and equilibrium measures on the unit circle. 

\item We then discuss analytic families of weighted composition operators and show that their eigenvalues vary analytically, from which we deduce the Central Limit Theorem. 

\item We then give a sufficient condition for the asymptotic variance to not vanish, thereby completing the proof of Theorem \ref{main-thm}.
\end{enumerate}

For concreteness, we work with Sobolev potentials and leave H\"older potentials to the reader.

\subsection{Weighted Composition Operators}
\label{sec:weighted-composition-operators}

Let $F$ be a centered inner function. The {\em weighted composition operator} $C_{F, \gamma}$ is defined as
$$
C_{F, \gamma}(g) = \gamma(z)\cdot(g \circ F(z)), \  \  z\in\mathbb D.
$$
We first examine the action of $C_{F, \gamma}$ on various spaces of holomorphic functions. When $\gamma \in H^\infty$ is a bounded analytic function, multiplication by $\gamma$ preserves the Hardy space $H^2$ and the Bergman space $A^2$, so that  $C_{F, \gamma}$ acts on $H^2$ and $A^2$.
Taking the adjoint with respect to the $L^2(\partial \mathbb{D}, m)$-pairing on the unit circle, we see that the {\em weighted transfer operator}
\begin{equation}
\label{eq:weighted-transfer-def}
(L_{F, \gamma} g) (\alpha) 
:= L_F(\gamma g) (\alpha) 
= \int_{\partial \mathbb{D}} \gamma\cdot g\, d\mu_\alpha
\end{equation}
acts on the Hardy space $H^2$ and the Dirichlet space $\mathcal D$. 

For applications to dynamical systems, we need to consider the action of weighted composition operators on real Banach spaces. If $\gamma$ is an $L^\infty$ function, then the operator $C_{F, \gamma}$ preserves $L^2$. 
In order for $C_{F, \gamma}$ to preserve $W^{-1/2,2}$, it is enough to take $\gamma \in \mathcal M(W^{-1/2,2})$, which is a more stringent condition.
    
 \begin{remark} 
(i) When $F(z) = z$ is the identity mapping, $C_{F, \gamma}$ is bounded on $W^{-1/2,2}$ if and only if $\gamma \in \mathcal M(W^{-1/2,2})$. However, for other inner functions $F$, the set of admissible $\gamma$ may be strictly larger.

(ii) A simple duality argument shows that $\mathcal M(W^{-1/2,2}) = \mathcal M(W^{1/2,2})$. Any multiplier of $W^{-1/2,2}$ is a bounded function, but the converse is not true.
 \end{remark}

For future reference, we also make the following simple observations:

\begin{lemma}
\label{gFh-transfer}
If $F$ is a centered inner function and $\gamma$ is a  weight in $L^\infty(\partial \mathbb{D}, m)$, then for any functions $g \in L^\infty(\partial \mathbb{D}, m)$ and $h \in L^2(\partial \mathbb{D}, m)$, we have
$$
L_{F, \gamma} (g \circ F \cdot h) = g \cdot L_{F, \gamma} (h).
$$
\end{lemma}

\begin{proof} 
Since for any $\alpha\in \partial \mathbb{D}$, the Aleksandrov-Clark measure $\mu_\alpha$ is supported on the set of points $\zeta \in \partial \mathbb D$ for which $\lim_{r \to 1} F(r\zeta) = \alpha$, we have:
\begin{align*}
L_{F, \gamma} (g \circ F \cdot h)(\alpha) 
& = \int_{\partial \mathbb{D}} \gamma h \cdot g\circ F \, d\mu_\alpha  \\
& = g(\alpha) \int_{\partial \mathbb{D}} \gamma h \, d\mu_\alpha \\
& = g(\alpha) \cdot L_{F, \gamma} (h)(\alpha),
\end{align*}
as desired.
\end{proof}

\begin{lemma}
\label{iterated-AC}
If $F$ is a centered inner function and $\gamma$ is a  weight in $L^\infty(\partial \mathbb{D}, m)$, then for any function $h \in L^2(\partial \mathbb{D}, m)$ and $\alpha \in \partial \mathbb{D}$, we have
$$
L^n_{F, \gamma}(h)(\alpha) = \int_{\partial \mathbb{D}} \Pi_n \gamma \cdot h \, d\mu_{F^{\circ n}, \alpha},
$$
where $$\Pi_n \gamma(x) = \gamma(x) \gamma(F(x)) \dots \gamma(F^{\circ (n-1)}(x)).$$
\end{lemma}

\begin{proof}
For weighted composition operators, it is clear that $C_{F,\gamma}^n = C_{F^{\circ n}, \Pi_n\gamma}$. Taking adjoints, we get
$$
L^n_{F, \gamma}(h)(\alpha) = L_{F^{\circ n}, \Pi_n\gamma}(h)(\alpha).
$$
The lemma follows after expressing the weighted transfer operator on the right in terms of Aleksandrov-Clark measures as in  (\ref{eq:weighted-transfer-def}).
\end{proof}

In order to study stochastic laws via thermodynamic formalism, we need to work with families of weighted transfer operators. Suppose $s \to \gamma_s$ is a family of weights in $\mathcal M(W^{1/2,2})$, which depend analytically on a parameter $s \in U$, where $U \subset \mathbb{C}$ is an open set containing the origin and $\gamma_0 = {\bf 1}$. 
From the definition of the weighted transfer operators (\ref{eq:weighted-transfer-def}), it is not difficult to see that $s \to  L_{F, \gamma_s}$ defines an analytic function from $U$ to the space of bounded linear operators on $W^{1/2,2}$ with derivative $(d/ds) \, L_{F, \gamma_s} = L_{F, \dot \gamma_s}$.

In Section \ref{sec:duals-of-some-function-spaces}, we saw that $L_{F, \gamma_0}$ has a spectral gap on $W^{1/2,2}$\,: it has a simple isolated eigenvalue $\lambda_0 = 1$, while the rest of the spectrum is contained in a ball $B(0, r)$ with $r < 1$. From the Kato-Rellich Perturbation Theorem, it follows that for any complex number $s$ sufficiently close to $0$, $L_{F, \gamma_s}$ also has a spectral gap: it possesses a simple isolated eigenvalue $\lambda_s \in \mathbb{C}$, while the rest of the spectrum is contained in $B(0, r_s)$ with $r_s < |\lambda_s|$. Moreover, perturbation theory tells us that in the presence of a spectral gap, both the eigenvalue $\lambda_s$ and the Riesz projection $R_s$ onto the one-dimensional eigenspace associated to the eigenvalue $\lambda_s$ vary analytically in $s$. 

For $s$ close to 0, the function $\rho_s = R_s {\bf 1}$ is not identically zero and therefore spans the $\lambda_s$-eigenspace of $L_{F, \gamma_s}$.
The number 
$$
P(s) \, = \, \log r_{W^{1/2,2}}(L_{F, \gamma_s}) \, = \, \log \lambda_s
$$ 
is called the {\em topological pressure} of the weight $\gamma_s \in \mathcal M(W^{1/2,2})$, where the branch of the logarithm is chosen so that $P(0)=0$.
 
\subsection{Stochastic Laws with respect to the Lebesgue Measure}
\label{sec:stochastic-lebesgue}

We now change the notation and write $\mathcal L_{-\log |F'| + g} := L_{F, e^{g}}$, as is standard in thermodynamic formalism. In other words, we specify
the transfer operator by the potential  $-\log |F'|+g$ rather than by the weight $e^g$.

It is easy to see that if $g \in \mathcal M(W^{1/2,2})$ is a Sobolev multiplier, then for any complex number $s \in \mathbb{C}$,
$$
e^{sg} = 1 + sg + \frac{s^2}{2!} g^2 + \frac{s^3}{3!} g^2 + \dots
$$ is also a multiplier with 
$$
\|e^{sg}\|_{\mathcal M(W^{1/2,2})} \le \exp \bigl (|s| \cdot \| g\|_{\mathcal M(W^{1/2,2})} \bigr ).
$$
In this case, the weighted transfer operator 
$\mathcal L_s := \mathcal L_{-\log|F'| + sg}$ 
is bounded on $W^{1/2,2}$ and varies analytically with $s$.
From the discussion in the previous section, there exists a Riesz projection $R_s: W^{1/2,2} \to \mathbb{C} \cdot \rho_s$, defined for all complex numbers $s$ sufficiently close to $0$, which satisfies
$$
\mathcal L_{s} R_s = R_s \mathcal L_{s}
\qquad \text{and} \qquad \mathcal L_{s} \rho_s = e^{P(s)} \rho_s,
$$
where $\rho_s = R_s {\bf 1}$.

The following lemma \cite[Propositions 4.10 and 4.11]{PP} describes the first two derivatives of the topological pressure:

\begin{lemma}
\label{derivatives-of-pressure}
Let $g \in \mathcal M(W^{1/2,2})$ be a Sobolev multiplier. Then,
$$
\dot P(0) = \int_{\partial \mathbb{D}} g dm.
$$
If in addition $g$ has mean 0, i.e.~$\int_{\partial \mathbb{D}} g dm = 0$, then
$$
\ddot P(0) \, = \, \sigma^2(g) \, = \, \lim_{n \to \infty} \frac{1}{n} \int_{\partial \mathbb{D}}  (S_n g)^2 dm.
$$
\end{lemma}

\begin{proof}
(i) Let $f_s$ be the eigenfunction of the weighted composition operator $h \to e^{sg} h\circ F$ with eigenvalue $e^{P(s)}$. When $s = 0$, the eigenfunction $f_0 = {\bf 1}$ and $P(0) = 0$.
 Differentiating 
\begin{equation}
\label{eq:fs-equation}
 e^{sg(z)} f_s(F(z)) = e^{P(s)} f_s(z)
 \end{equation}
 with respect to $s$ at $0$, gives
$$
g(z) + \dot f_0(F(z)) = \dot P + \dot f_0(z).
$$
Integrating with respect to $m$, we get
$$
\int_{\partial \mathbb{D}} g(z) dm(z) = \dot P + \biggl \{ \int_{\partial \mathbb{D}} \dot f_0(z) dm(z) - \int_{\partial \mathbb{D}} \dot f_0(F(z)) dm(z) \biggr \}.
$$
Since $m$ is $F$-invariant, the two integrals cancel out.

(ii)
Differentiating (\ref{eq:fs-equation}) twice at $s = 0$ gives
$$
g(z)^2 + 2g(z) \dot f_0(F(z)) + \ddot f_0(F(z)) = \ddot P(0) + \ddot f_0(z).
$$
Integrating with respect to $m$, we get
$$
\int_{\partial \mathbb{D}}  g^2 dm + 2 \int_{\partial \mathbb{D}}  g (\dot f_0 \circ F)  dm = \ddot P(0).
$$
By instead differentiating the $n$-th iterate, we obtain
$$
\int_{\partial \mathbb{D}}  (S_n g)^2 dm + 2 \int_{\partial \mathbb{D}}  S_n g  \cdot (\dot f_0 \circ F) dm = n \ddot P(0).
$$
The statement follows after dividing by $n$, taking $n \to \infty$ and applying the ergodic theorem.
\end{proof}

From here, one can prove the Central Limit Theorem for real-valued observables in $\mathcal M(W^{1/2,2})$ as in \cite[Proposition 4.13]{PP}. Alternatively, one can appeal to the work of S.~Gou\"ezel \cite{gouezel} which shows that the process $\{ g \circ F^n : n \ge 1\}$ satisfies an {\em almost sure invariance principle} (ASIP). Loosely speaking, the ASIP says that one can redefine the process
$\{ g \circ F^n : n \ge 1\}$ on some probability space with a Brownian motion $B_t$, so that $S_n g$ is close to $B_{\sigma^2 n}$. 
With the help of the ASIP, one can obtain stochastic laws such as the Central Limit Theorem and the Law of the Iterated Logarithm 
for $g \circ F^n$ from the corresponding facts for Brownian motion, essentially for free.

\subsection{When is the Variance Non-Zero?}

To complete the proof of Theorem \ref{main-thm}, we need a mechanism for showing that the asymptotic variance of an observable $g \in \mathcal M(W^{1/2,2})$ with respect to the Lebesgue measure is not equal to zero:

\begin{theorem}\label{t120230828}
\label{no-eigenfunctions-thm2}
Suppose $F$ is a centered  inner function that is not a finite Blaschke product. If $g \in \mathcal M(W^{1/2,2})$ is a real-valued non-constant function of mean zero, then
$\sigma^2(g) \ne 0$.
\end{theorem}

Before providing the proof of the above theorem, we make the following preliminary observation:

\begin{lemma}
\label{boundedness-of-the-iterates-of-L}
Suppose $g \in \mathcal M(W^{1/2,2})$ is a real-valued function on the unit circle.
If $\| g \|_{\mathcal M(W^{1/2,2})}$ is sufficiently small, then for any $u \in W^{1/2,2}$, the sequence 
$\bigl \{ \mathcal L_{-\log|F'|+ig}^n (u) \bigr \}_{n=0}^\infty$ is bounded in $W^{1/2,2}$.
\end{lemma}

\begin{proof}
When $\| g \|_{\mathcal M(W^{1/2,2})}$ is small, we have $\sigma_{e, W^{1/2,2}}(\mathcal L_{-\log|F'|+ig}) < 1$.
Outside the essential spectral radius, 
 $\mathcal L_{-\log|F'|+ig}$ may only have countably many isolated eigenvalues of finite multiplicity. Therefore, to prove the lemma, it suffices to show that the operator $\mathcal L_{-\log|F'|+ig}$ acting on $W^{1/2,2}$ does not possess any eigenvalues with modulus greater than $1$. However, this is easy: as $\mathcal L_{-\log|F'|+ig}$ is an isometry on $L^2(\partial \mathbb{D}, m)$, it cannot have any $L^2(\partial \mathbb{D}, m)$ eigenvalues outside the closed unit disk, let alone any $W^{1/2,2}$ eigenvalues.
  \end{proof}

For the remainder of this section, we assume that $\| g \|_{\mathcal M(W^{1/2,2})}$ is sufficiently small so that the assumptions of Lemma \ref{boundedness-of-the-iterates-of-L} are satisfied.
The proof of Theorem \ref{no-eigenfunctions-thm2} relies on  \cite[Propositions 4.12 and 4.2]{PP}:

\begin{lemma}
\label{zero-asymptotic-variance}
 If $\sigma^2(g) = 0$, then $g$ is cohomologous to zero, that is, one can write
\begin{equation}
\label{eq:zero-av}
g(x) = \zeta(F(x)) - \zeta(x),
\end{equation}
for some real-valued function $\zeta \in L^2(\partial \mathbb{D}, m)$.
\end{lemma}

\begin{lemma}
\label{smoothness-of-eigenfunctions4}
Any $L^2(\partial \mathbb{D}, m)$ eigenfunction of $e^{-ig(x)} C_F$ belongs to $W^{1/2,2}$.
\end{lemma} 

While the two lemmas above are stated for different function spaces in \cite{PP}, the proofs are similar to the ones given in  \cite{PP}.
Below, we present the proof of Lemma \ref{smoothness-of-eigenfunctions4} as there are additional technicalities arising from working with Aleksandrov-Clark measures and
allow the interested reader to consult \cite{PP} for the proof of Lemma \ref{zero-asymptotic-variance}.
 
 \begin{proof}[Proof of Lemma \ref{smoothness-of-eigenfunctions4}]
Suppose $w(x)$ is an $L^2$ eigenfunction of the weighted composition operator $e^{-ig(x)} C_F$:
$$
e^{-ig(x)} w(F(x)) = \lambda w(x).
$$
As the Lebesgue measure is ergodic with respect to $F$, $|\lambda| = 1$ and $|w|$ is constant a.e. Since $w$ is not identically 0, by scaling it appropriately, we can make its absolute value 1 a.e. Multiplying the identity
$$
w(F^{\circ n}(x)) \cdot \overline{w(x)} = \lambda^n e^{iS_n g(x)}
$$
by a function $h \in C^\infty$, integrating with respect to the Aleksandrov-Clark measure $\mu_{F^{\circ n}, \alpha}$ and using Lemma \ref{iterated-AC}, we get
\begin{equation}
\label{eq:smoothness-eigenfunctions}
w(\alpha) \cdot \bigl [ \mathcal L_{-\log|F'|}^n (h \overline{w}) \bigr ](\alpha) = \lambda^n \cdot \bigl [ \mathcal L^n_{-\log|F'|+ig} h \bigr ] (\alpha).
\end{equation}
As the sequence of functions $\bigl \{ \mathcal L_{-\log|F'|+ig}^n h \bigr \}_{n=0}^\infty$ is bounded in $W^{1/2,2}$ by Lemma \ref{boundedness-of-the-iterates-of-L} and the inclusion 
$W^{1/2,2} \subset L^2$ is compact, we can pass to a subsequence so that 
$$
\mathcal L_{-\log|F'|+ig}^{n_k} h \to h^* \text{ in }L^2, \qquad \text{for some }h^* \in W^{1/2,2}.
$$
We pass to a further subsequence so that $\lambda^{n_k} \to \lambda^* \in \partial \mathbb{D}$.
In view of Lemma \ref{L-contraction2}, taking the $L^2$-limit as $n \to \infty$ in (\ref{eq:smoothness-eigenfunctions}) gives
$$
w\int_{\partial \mathbb{D}} \overline{w} h \, dm = \lambda^* h^*.
$$
Since $w$ is assumed to be non-zero, there exists an $h \in C^\infty$ with
$$\int_{\partial \mathbb{D}} \overline{w}h \, dm \ne 0,$$ from which it follows that $w$ is a scalar multiple of $h^*$ a.e.
\end{proof}

By Lemma \ref{zero-asymptotic-variance}, if $\sigma^2(g) = 0$,
then the weighted composition operator $e^{- i g} C_F$ possesses 
an $L^\infty(\partial \mathbb{D}, m)$ eigenfunction with eigenvalue 1, namely $e^{i \zeta(x)}$.
As $L^\infty(\partial \mathbb{D}, m) \subset L^2(\partial \mathbb{D}, m)$, Lemma \ref{smoothness-of-eigenfunctions4} tells us that $e^{i \zeta(x)}$ is actually a $W^{1/2,2}$ eigenfunction.
Therefore, to prove Theorem \ref{no-eigenfunctions-thm2}, we may show that $e^{-ig} C_F$ has no $W^{1/2,2}$ eigenfunctions.

\begin{lemma}
\label{dirichlet-composition}
Suppose $h \in \mathcal D$ is not constant and $F$ is an inner function that is not a finite Blaschke product. Then $h \circ F \notin \mathcal D$.
\end{lemma}

In the proof below, we will use the fact that if  $F$ is not a finite Blaschke product, then for a.e.~point $w \in \mathbb{D}$, the pre-image $F^{-1}(w)$ is an infinite set. Actually, Frostman's theorem \cite[Theorem 2.5]{mashreghi} says that the set of exceptional points has logarithmic capacity zero, but we will not need this stronger conclusion.

\begin{proof} The Dirichlet seminorm
$$
\int_{\mathbb D} |(h \circ F)'(z)|^2 \, |dz|^2
$$
computes the weighted area of the image of $h \circ F$. Since $F$ has infinite degree, almost every point of $h(\mathbb{D})$ is covered infinitely often. As non-constant holomorphic maps are open, $h(\mathbb{D})$ has positive area and so the Dirichlet seminorm of $h \circ F$ is infinite.
\end{proof}

In particular, the above lemma implies that if an inner function belongs to the Dirichlet space, then it is a finite Blaschke product.

\begin{proof}[Proof of Theorem \ref{no-eigenfunctions-thm2}]
As noted above, if $\sigma^2(g) = 0$, then the weighted composition operator $e^{-ig} C_F$
has a non-trivial eigenfunction in $W^{1/2,2}$ with eigenvalue 1:
$$
e^{-ig(x)} w(F(x)) = w(x).
$$
Since the right hand side is in $W^{1/2,2}$, to obtain a contradiction, we show that the left hand side is not in $W^{1/2,2}$.
Since $g \in \mathcal M(W^{1/2,2})$ implies that $e^{ig} \in \mathcal M(W^{1/2,2})$, this reduces our task to showing that $w(F(x)) \notin W^{1/2,2}$.
By decomposing $w$ into holomorphic and anti-holomorphic parts, its enough to consider the case when $w$ is a holomorphic function in $W^{1/2, 2}$, i.e.~an element of  the Dirichlet space 
$\mathcal D$. The existence of such a $w$ is ruled out by Lemma \ref{dirichlet-composition}.
\end{proof}

\subsection{Conformal and Equilibrium Measures}\label{CEM}

We now switch to H\"older potentials to avoid subtle issues arising from the fact that functions in the Sobolev space $W^{1/2,2}$  may be not be bounded. (One can resolve this technicality by working in the space $W^{\beta,2}$ for some $\beta > 1/2$, but by the Sobolev embedding theorem, these functions have some H\"older regularity, so we might as well work in $C^\alpha$.)

\begin{lemma}
\label{gamma-spectral-gap}
Suppose $\gamma: \partial \mathbb{D} \to (0, \infty)$ is a positive function of class $C^\alpha$ for some $0 < \alpha < 1$. When $\| \gamma - 1 \|_{C^\alpha}$ is sufficiently small, $L_{F, \gamma}$ has a spectral gap on $C^\alpha$ and the eigenspace associated to the dominant eigenvalue $\lambda > 0$ is spanned by a strictly positive function $\rho_\gamma$.
\end{lemma}

\begin{proof}
In Section \ref{sec:duals-of-some-function-spaces}, we used a duality argument to show that $L_{F}$ has a spectral gap on $C^\alpha$. It is not difficult to see the spectral gap persists for $L_{F, \gamma}$ when  $\| \gamma - 1 \|_{C^\alpha}$ is small by applying the Kato-Rellich Perturbation Theorem to the analytic family of operators $s \to L_{F, 1 + s(\gamma - 1)}$ and using the operator norm estimate $\| L_{F, \gamma_1} - L_{F, \gamma_2} \|_{\mathcal B(C^\alpha)} \lesssim \| \gamma_1 - \gamma_2 \|_{C^\alpha}$.

Since $L_{F, \gamma}$ is a real operator, the one-dimensional eigenspace associated to the dominant eigenvalue $\lambda$ is spanned by a real-valued function $\rho_\gamma$. 
(As $\overline{\rho_\gamma}$ is an eigenfunction with eigenvalue $\overline{\lambda}$, the spectral gap implies that $\overline{\rho_\gamma} = \rho_\gamma$ and $\overline{\lambda} = \lambda$.) Since the zero set of $\rho_\gamma$ is $F$-invariant and $\rho_\gamma$ is continuous, $\rho_\gamma$ never vanishes by Lemma \ref{equidistribution-lemma}. Multiplying $\rho_\gamma$ by $-1$ if necessary, we may assume that the eigenfunction $\rho_\gamma$ is strictly positive.
\end{proof}

Below, we construct conformal and equilibrium measures associated to a weight $\gamma$ satisfying the hypotheses of the above lemma.
We start with the conformal measure:

\begin{lemma}[Conformal measure]
\label{CEM-conformal}
There exists a unique probability measure $m_\gamma$ on $\partial \mathbb{D}$ such that $C_{F, \gamma} m_\gamma = \lambda  m_\gamma$.
\end{lemma}

In the lemma above, the action of the weighted composition operator on Borel measures on the unit circle is given by 
$
C_{F, \gamma} \nu := \gamma \, d(C_F \nu). 
$
The definition is chosen so that the following duality relation holds: for every bounded Borel function $\phi \in L^\infty(\partial \mathbb{D}, m)$ and Borel measure $\nu$,
\begin{equation}
\label{eq:composition-operator-on-measures2}
\int_{\partial \mathbb{D}} \phi \, d(C_{F, \gamma} \nu) = \int_{\partial \mathbb{D}} \gamma \phi \, d(C_{F} \nu)  = \int_{\partial \mathbb{D}} L_F(\gamma \phi) d\nu = \int_{\partial \mathbb{D}} L_{F, \gamma} \phi \, d\nu.
\end{equation}
The second equality above follows from the identity (\ref{eq:composition-operator-on-measures}).
The proof of Lemma~\ref{CEM-conformal} relies on two auxiliary lemmas:

\begin{lemma}
\label{l1Fg}
The operator $\lambda^{-1} L_{F, \gamma}$ is the adjoint of the composition operator $C_F$ in the Hilbert space $L^2(\partial \mathbb{D}, m_\gamma)$, i.e.~
$$
\int_{\partial \mathbb{D}} g \cdot  C_F h \, dm_\gamma = \int_{\partial \mathbb{D}} \lambda^{-1} L_{F, \gamma} g \cdot h \, dm_\gamma, \qquad g, h \in L^2(\partial \mathbb{D}, m_\gamma).
$$
\end{lemma}

\begin{proof}
By the duality relation (\ref{eq:composition-operator-on-measures2}) and Lemma \ref{h-out}, we have
\begin{align*}
\int_{\partial \mathbb{D}} L_{F,\gamma} g \cdot h \, dm_\gamma
&= \int_{\partial \mathbb{D}} g \, dC_{F,\gamma}(h m_\gamma) \\
&= \int_{\partial \mathbb{D}} g \cdot \gamma \, dC_F(h m_\gamma) \\
&= \int_{\partial \mathbb{D}} g \cdot \gamma \, (h \circ F) \, dC_F m_\gamma \\
&= \int_{\partial \mathbb{D}} g \cdot (h \circ F) \, dC_{F,\gamma} m_\gamma \\
&= \lambda \int_{\partial \mathbb{D}} g \cdot C_F h \, dm_\gamma,
\end{align*}
as desired.
\end{proof}

\begin{lemma}
\label{Lgamma-lemma}
Suppose $m_\gamma$ is a conformal measure on the unit circle. For any $g \in C^\alpha(\partial \mathbb{D})$, we have
$$
\lambda^{-n} L_{F,\gamma}^n g \ \xlongrightarrow[C^\alpha] \ \Biggl( \int_{\partial \mathbb{D}} g \, dm_\gamma \Biggr) \rho_\gamma
\qquad \text{ as } n \to \infty.
$$
where the eigenfunction $\rho_\gamma$ of $L_{F,\gamma}$ with eigenvalue $\lambda$ is normalized so that $\int_{\partial \mathbb{D}} \rho_\gamma dm_\gamma = 1$.
\end{lemma}

\begin{proof}
Taking $\nu = m_\gamma$ in the duality relation (\ref{eq:composition-operator-on-measures2}), we see that for any bounded Borel function $\phi \in L^\infty(\partial \mathbb{D}, m)$, we have
$$
\int_{\partial \mathbb{D}} L_{F, \gamma} \phi \, dm_\gamma = \lambda  \int_{\partial \mathbb{D}} \phi \, dm_\gamma.
$$ 
Since $L_{F,\gamma}$ has a spectral gap on $C^\alpha$ by Lemma \ref{gamma-spectral-gap}, it follows that
$$
\lambda^{-n} L_{F,\gamma}^n g \, \to_{C^\alpha} \, C_g \cdot \rho_\gamma
$$
for some constant $C_g \in \mathbb{C}$. To evaluate $C_g$, we integrate both sides against $m_\gamma$\,:
$$
\lim_{n \to \infty} \int_{\partial \mathbb{D}} \lambda^{-n} L_{F,\gamma}^n g dm_\gamma \, = \,
 C_g \int_{\partial \mathbb{D}} \rho_\gamma dm_\gamma \, = \, C_g.
$$
However, by Lemma \ref{l1Fg}, for any $n \ge 0$,
$$
\int_{\partial \mathbb{D}} \lambda^{-n} L_{F,\gamma}^n g dm_\gamma = \int_{\partial \mathbb{D}} g dm_\gamma.
$$
Hence, $C_g = \int_{\partial \mathbb{D}} g dm_\gamma$ as desired.
\end{proof}

\begin{proof}[Proof of Lemma \ref{CEM-conformal}]
{\em Step 1.} Let $\mathcal P(\partial \mathbb{D})$ denote the collection of Borel probability measures on the unit circle endowed with topology of weak convergence. Consider the map $T: \mathcal P(\partial \mathbb{D}) \to \mathcal P(\partial \mathbb{D})$ given by
$$
T\nu := \frac{ C_{F, \gamma}\nu}{(C_{F, \gamma}\nu)(\partial \mathbb{D})}.
$$
By the Schauder-Tychonoff Fixed Point Theorem, $T$ has a fixed point. Thus, there exists a measure $m_\gamma \in \mathcal P(\partial \mathbb{D})$ such that
$$
C_{F, \gamma} m_\gamma = \tilde \lambda  m_\gamma, \qquad \text{for some } \tilde \lambda > 0.
$$

{\em Step 2.} Together with (\ref{eq:composition-operator-on-measures2}), the equation
$C_{F, \gamma}^n m_\gamma = \tilde \lambda^n  m_\gamma$
shows that
\begin{equation}
\label{eq:comparing-lambdas}
 \tilde \lambda^n = \int_{\partial \mathbb{D}} L_{F,\gamma}^n {\bf 1} \, dm_\gamma.
\end{equation}
As $\gamma > 0$, $L_{F, \gamma}$ is a positive operator.
Since the eigenfunction $\rho_\gamma$ is pinched between two positive constants, i.e.~
$$
C_1 \cdot {\bf 1} \, \le \, {\rho_\gamma} \, \le \, C_2 \cdot {\bf 1},
$$
we have
$$
C_1 \cdot L_{F,\gamma}^n {\bf 1} \, \le \, \lambda^n {\rho_\gamma} \, \le \, C_2 \cdot L_{F,\gamma}^n {\bf 1},
$$
so the right hand side of (\ref{eq:comparing-lambdas}) is comparable to $\lambda^n$. As a result, $\tilde \lambda = \lambda$.

\medskip

{\em Step 3.} Let $m_\gamma$ be a probability measure on the unit circle which satisfies $C_{F, \gamma} m_\gamma = \lambda  m_\gamma$. By Lemma \ref{Lgamma-lemma}, for any $g \in C^\alpha$, the  sequence of functions $\{ \lambda^{-n} L_{F, \gamma}^n g\}$ converges in $C^\alpha$ to $(\int_{\partial \mathbb{D}} g dm_\gamma) \rho_\gamma$. As $C^\alpha$ is dense in the continuous functions on the unit circle, this determines the measure $m_\gamma$ uniquely.
\end{proof}

We now turn to the construction of the equilibrium measure.

\begin{lemma}
The measure $\mu_\gamma = \rho_\gamma m_\gamma$ is mixing.
\end{lemma}

\begin{proof}
Applying Lemma \ref{l1Fg} $n$ times gives
$$
\int_{\partial \mathbb{D}} g \cdot  C^n_F h \, dm_\gamma = \int_{\partial \mathbb{D}} \lambda^{-n} L_{F, \gamma}^n g \cdot h \, dm_\gamma.
$$
By Lemma \ref{Lgamma-lemma}, it follows that
$$
\int_{\partial \mathbb{D}} g \cdot  C^n_F h \, dm_\gamma \, \to \, \int_{\partial \mathbb{D}} g dm_\gamma  \int_{\partial \mathbb{D}} h \rho_\gamma dm_\gamma
\, = \, \int_{\partial \mathbb{D}} g dm_\gamma  \int_{\partial \mathbb{D}} h d\mu_\gamma,
$$
as $n \to \infty$.
Replacing $g$ by $\rho_\gamma g$ yields
$$
\int_{\partial \mathbb{D}} g \cdot  C^n_F h \, d\mu_\gamma \to  \int_{\partial \mathbb{D}} g d\mu_\gamma  \int_{\partial \mathbb{D}} h d\mu_\gamma, \qquad \text{as }n \to \infty,
$$
as desired.
\end{proof}

\begin{lemma}[Equilibrium measure]
 There exists a unique $F$-invariant probability measure $\mu_\gamma$ on $\partial \mathbb{D}$ which is absolutely continuous with respect to $m_\gamma$.
The Radon-Nikodym derivative is $d\mu_\gamma/dm_\gamma = \rho_\gamma$.
\end{lemma}

\begin{proof}
{\em Step 1.} It is easy to see that the {\em equilibrium measure}
$\mu_\gamma = \rho_\gamma dm_\gamma$ is invariant under $F$: if $g \in C(\partial \mathbb{D})$, then
\begin{align*}
\int_{\partial \mathbb{D}} g \circ F \, d\mu_\gamma 
& = \int_{\partial \mathbb{D}} g \circ F \cdot \rho_\gamma \, dm_\gamma \\
& = \lambda^{-1} \int_{\partial \mathbb{D}} g \circ F \cdot \rho_\gamma \, d (C_{F,\gamma} m_\gamma) \\
& =  \lambda^{-1}  \int_{\partial \mathbb{D}} L_{F, \gamma}(g \circ F \cdot \rho_\gamma) \, dm_\gamma \\
& =  \lambda^{-1} \int_{\partial \mathbb{D}} g  L_{F, \gamma} \rho_\gamma \, dm_\gamma \\
& = \int_{\partial \mathbb{D}} g \, d\mu_\gamma,
\end{align*}
where we have used the duality relation (\ref{eq:composition-operator-on-measures2}) and Lemma \ref{gFh-transfer}. 

\medskip

{\em Step 2.}
The duality relation (\ref{eq:composition-operator-on-measures2}) implies an exponential decay of correlations for $C^\alpha$ observables with respect to the measures $m_\gamma$ and $\mu_\gamma = \rho_\gamma \, dm_\gamma$ as in Lemma \ref{L-contraction}. From the density of $C^\alpha(\partial \mathbb{D}) \subset L^2(\partial \mathbb{D})$, it follows $F$ is mixing with respect to $\mu_\gamma$.

\medskip

{\em Step 3.}
If there was another $F$-invariant measure $\tilde \mu$ absolutely continuous with respect to $m_\gamma$, then the Radon-Nikodym derivative
$$
g = \frac{d\tilde \mu}{d\mu_\gamma}
$$
would be an $F$-invariant function. Since $F$ is ergodic with respect to $\mu_\gamma$ (being mixing), $g$ must be constant $\mu_\gamma$-a.e. This proves uniqueness.
\end{proof}

\begin{remark}
We use the term equilibrium measure to give $\mu_\gamma$ a name. We do not claim that the equilibrium measure maximizes $h_{ \mu}(F) + \int_{\partial \mathbb{D}} \gamma \, d\mu$ over all $F$-invariant probability measures. Nor is it clear that the maximizer is unique as $h_{ \mu_\gamma}(F) + \int_{\partial \mathbb{D}} \gamma \, d\mu_\gamma$ could be infinite.
\end{remark}

Previously in Section \ref{sec:stochastic-lebesgue}, we have established the Central Limit Theorem and other stochastic laws with respect to the Lebesgue measure. Similar arguments show that the Central Limit Theorem also holds with respect to the equilibrium measures constructed above.

\begin{remark}When $F$ is an inner function that is not a finite Blaschke product and $\gamma$ is non-constant, it is not possible to normalize the transfer operator by setting
$$
\widehat{L}_{F, \gamma} g = \frac{1}{\lambda_\gamma \rho_\gamma} \cdot L_{F, \gamma} (\rho_\gamma g),
$$
so that $\widehat{L}_{F, \gamma} {\bf 1} = {\bf 1}$, since this amounts to replacing $\gamma$ with
$$
\widehat{\gamma}(x) = \frac{\gamma(x)}{\lambda_\gamma} \cdot \frac{\rho_\gamma(x)}{\rho_\gamma(F(x))},
$$
which is not in $C^\alpha$. While this is not an essential obstruction, it does make computing the derivatives of pressure and proving the CLT slightly more cumbersome. We leave the details to the reader.
 \end{remark}
 
 \addtocontents{toc}{\protect\newpage}
 
\part{Life on the Shift Space}
\label{shift-space}

\section{Refined Thermodynamic Formalism and Stochastic Laws for Countable Alphabet Subshifts of Finite Type}
\label{sec:SFTTFSL}

In this section, we recall the rudimentary notions of thermodynamic formalism on the shift space from \cite{PP, MU, URM22}. We then define the classes of {\em normal} and {\em robust} potentials that mimic the properties of the potential $- \log |F'|$ associated to a one component inner function $F$. One advantage of this somewhat abstract approach of working on the shift space, rather than directly on the unit circle, is that the arguments are applicable to a plethora of other dynamical systems, not just to inner functions.

A normal or a robust potential can be put into a holomorphic family of potentials indexed by the right half-plane $\mathbb{C}_1^+ = \{ s\in\C:\re s > 1 \}$. In order to study the differentiability properties of the transfer operator $\mathcal L_{s\psi}$ near the line $s = 1$, we introduce the modified transfer operators $\mathcal L_{s,p}$ and develop some of their basic properties.

\subsection{Shift Space}
\label{sec:shift-shape}

Let $E$ be a countable (finite or infinite) alphabet and $A: E \times E \to \mathbb{R}$ be a matrix whose entries are either 0 or 1.
The {\em shift space} $E_A^\infty \subset E^\infty$ consists of all infinite words $\xi = \xi_0 \xi_1 \xi_2 \dots$ with $\xi_i \in E$ such that $A(\xi_i, \xi_{i+1}) = 1$. We call words in  $E_A^\infty$ {\em admissible}\/. The shift map $\sigma: E_A^\infty \to E_A^\infty$ takes a word and removes its first letter.

We let $E_A^*$ denote the collection of all finite admissible words in the alphabet $E$. We denote the length of $\tau \in E_A^*$ by $|\tau|$. For $\tau \in E_A^*$, we let $[\tau]$ denote the {\em cylinder} of all infinite words that begin with $\tau$.

Below, we will always assume that the incidence matrix $A$ is {\em finitely primitive}\/, i.e.~there exists a finite collection of words $\Lambda \subset E_A^*$ of the same length such that for any two letters $a$ and $b$ in $E$ there exists $\tau\in\Lambda$ such that the word $a\tau b$ is admissible.

We endow the shift space $E_A^\infty$ with the metric $$d(\xi, \xi') = 2^{-|\xi \wedge \xi'|},$$ where $|\xi \wedge \xi'|$ denotes
the length of the common prefix of $\xi, \xi'$. We denote the space of (complex-valued) continuous functions on $E_A^\infty$ by  $C(E_A^\infty)$ and the space of bounded continuous functions by $C_{\bounded}(E_A^\infty)$.

\begin{remark}
If the alphabet is infinite, then the two spaces are usually not the same, i.e.~$C_{\bounded}(E_A^\infty) \ne C(E_A^\infty)$. For instance, if $E = \mathbb{N}$ and $A(i, j) = 1$ for all $i, j \in \mathbb{N}$, then the function $f(\xi) = \xi_1$ is continuous but not bounded.
\end{remark}

 For $\alpha > 0$, we endow the space of H\"older continuous functions $C^\alpha(E_A^\infty)$ with the norm
$$
\| f \|_\alpha := \| f \|_\infty + v_\alpha(f),
$$
where 
$$
v_\alpha(f):= \sup\left\{\frac{|f(\xi') - f(\xi)|}{d(\xi, \xi')^\alpha}:\xi\ne\xi'\in E_A^\infty\right\}.
$$

\begin{lemma}
\label{basic-holder}
The space of H\"older continuous functions $C^\alpha(E_A^\infty)$ forms an algebra. In fact, 
$$
\|f g \|_\alpha \le 3 \| f \|_\alpha \cdot \| g\|_\alpha.
$$
\end{lemma}

We leave the proof as an exercise for the reader. In particular, the above lemma implies that the exponential of a H\"older continuous function is H\"older continuous.

\subsection{Transfer Operators}

To a potential $\psi: E_A^\infty \to \mathbb{C}$, one can associate the Perron-Frobenius or {\em transfer} operator
$$
\mathcal L_{\psi} g(\omega) = \sum_{a \in E} e^{\psi(a \omega)} \cdot g(a \omega),
$$
where we use the convention that a term is 0 if it features an inadmissible word.
We say that a potential $\psi$ is {\em summable} if $$\sum_{a \in E}\, \sup_{[a]} |e^{\psi}|,$$
and {\em level 1 H\"older continuous}\footnote{The term ``H\"older continuous on cylinders'' is also used.} with exponent $\alpha > 0$ if
$$
\sup_{a \in E} v_\alpha(\psi|_{[a]}) < \infty.
$$ 

It is explained in \cite[Chapter 2.3]{MU} or \cite[Lemma 17.6.2]{URM22} that if $\psi$ is a real-valued summable potential, then $\mathcal L_\psi$ preserves the space $C_{\bounded}(E_A^\infty)$. Furthermore, according to \cite[Chapter 2.4]{MU}, comp.~also \cite[Theorems~8.1.12 and 18.1.9]{URM22}, if one additionally assumes that $\psi$ is level 1 H\"older continuous, then $\mathcal L_\psi$ also preserves $C^\alpha(E_A^\infty)$, where it has a spectral gap:

\begin{theorem}[Ruelle-Perron-Frobenius]
\label{rpf-symbolic-dynamics}
For a real-valued  summable level 1 H\"older continuous potential $\psi$ on $E_A^\infty$, we have:

{\em (i)} The operator $\mathcal L_\psi: C^\alpha(E_A^\infty) \to C^\alpha(E_A^\infty)$ has a maximal positive eigenvalue
$\lambda_\psi$.

{\em (ii)} The eigenvalue $\lambda_\psi$ is simple and the rest of the spectrum of
$\mathcal L_\psi$ is contained in a ball centered at the origin of radius strictly less than $\lambda_\psi$.

{\em (iii)} The dual operator  $\mathcal L^*_\psi = C_{\sigma, \psi}: \mathcal M(E_A^\infty) \to \mathcal M(E_A^\infty)$ has a unique probability eigenmeasure $m_\psi$ and the corresponding eigenvalue is also 
$\lambda_\psi$.

{\em (iv)} The eigenspace of the eigenvalue $\lambda_\psi$ is generated by a strictly positive H\"older continuous eigenfunction $\rho_\psi$ with $m_\psi(\rho_\psi)=1$.

{\em (v)} The Borel probability measure $\mu_\psi = \rho_\psi \, dm_\psi$ is invariant under the shift map $\sigma: E_A^\infty \to E_A^\infty$.

{\em (vi)} For any $g \in L^1(E_A^\infty)$, we have
\begin{equation}
\label{eq:conformal-measures-and-transfers}
\int_{E_A^\infty} \mathcal Lg\, dm = \lambda_\psi \int_{E_A^\infty} g dm_\psi.
\end{equation}

{\em (vii)} For any function $g \in C^\alpha(E_A^\infty)$,
\begin{equation}
\label{eq:conformal-measures-and-transfers2}
\lambda_\psi^{-n} \int_{E_A^\infty} \mathcal L^n g \,dm \longrightarrow \biggl ( \int_{E_A^\infty} g dm_\psi \biggr )\cdot \rho_\psi.
\end{equation}

\end{theorem}

We refer to $m_\psi$ as the {\em conformal measure} associated to $\psi$, while $\mu_\psi$ is known as the
{\em invariant Gibbs state} or {\em equilibrium measure}\/ for $\psi$. From property (iv), it follows that 
the spaces $L^p(E_A^\infty, m_\psi) = L^p(E_A^\infty, \mu_\psi)$ are equal for any $1 \le p \le \infty$.
The {\em topological pressure} $P(\psi)$ of $\psi$ is defined as $\log \lambda_\psi$.

The following theorem describes the spectrum of the transfer operator associated to a complex-valued potential:

\begin{theorem}
\label{complex-potentials}
Let $\psi = u + iv$ be a complex-valued summable level 1 H\"older continuous potential on $E_A^\infty$ with $L_u {\bf 1} = {\bf 1}$.
The spectrum of $\mathcal L_\psi$ acting on $C^\alpha(E^\infty_A)$ is contained in the closed unit disk. On the unit circle, $\mathcal L_\psi$
may have at most one simple eigenvalue. Apart from this possible eigenvalue, the rest of the spectrum is contained in a disk of strictly smaller radius.
The following conditions are equivalent:

{\em (i)} $e^{-iv} C_\sigma$ does not have an $L^2(m_u)$ eigenvalue.

{\em (ii)} $e^{-iv} C_\sigma$ does not have a $C^\alpha$ eigenvalue.

{\em (iii)}  the spectral radius of $\mathcal L_\psi$ is less than
 $\lambda_{u}$.

{\em (iv)} $v$ is cohomologous to a potential $v_1$ which takes values in $2\pi \mathbb{Z}$, that is,
$$v - v_1 = w - w \circ \sigma, \qquad \text{for some } w \in C^\alpha(E^\infty_A).$$

{\em (v)} the {\em length spectrum} of $v$
$$
\bigl \{ S_n v(\xi) : \sigma^n(\xi) = \xi \text{ for some }n \ge 1 \bigr \} \subset 2\pi \mathbb{Z}.
$$
\end{theorem}
The proofs of the equivalences are scattered in Sections 4 and 5 of \cite{PP}. While \cite{PP} deals with finite alphabets, the proofs carry over to countable alphabets.

 We say that a potential $\psi = u + iv$ is {\em D-generic} if $e^{-iav} C_\sigma$ fails to have a $C^\alpha$ eigenvalue for any $a \ne 0$, or alternatively, if the length spectrum of $\psi$ is not contained in a discrete subgroup of $\mathbb{R}$.

\subsection{Families of Potentials}

We say that a potential $\psi: E_A^\infty \to \mathbb{R}$  is {\em normal} if it is 
\begin{itemize}
\item level 1 H\"older continuous for some exponent $\alpha > 0$,
\item summable,
\item negative: $\sup \psi < 0$,
\item centered: $P(\psi) = 0$.
\end{itemize}

Recall that $\mathbb{C}_1^+$ is the right half-plane $\{s\in\C: \re s > 1 \}$ and $\overline{\mathbb{C}_1^+}$ is its closure. Given a normal potential $\psi$, we embed $\mathcal L_\psi$ in a holomorphic family of operators $\mathcal L_{s\psi}$ indexed by $s \in \mathbb{C}_1^+$.
When the potential is clear from context, we write $\mathcal L_s$ instead of $\mathcal L_{s\psi}$.
The following result can be found in \cite[Chapter 2.6]{MU} or \cite[Lemma~20.2.2, Theorems~20.2.3 and 20.1.12]{URM22}:
  
\begin{theorem}
\label{continuity-and-holomorphy-Ls}
If $\psi: E_A^\infty \to \mathbb{R}$ is a normal potential, then the function 
$$
s \to \mathcal L_{s}, \qquad \overline{\mathbb{C}_1^+} \to  \mathcal B(C^\alpha(E_A^\infty))
$$ 
is continuous on $\overline{\mathbb{C}_1^+}$ and holomorphic on  $\mathbb{C}^+_1$. Moreover, there exists an open set $U \subset \mathbb{C}_1^+$ containing $(1,+\infty)$ such that for all $s \in U$, the operator $\mathcal L_{s}$ has a simple isolated  eigenvalue $\lambda_s$ whose modulus is equal to the spectral radius of $\mathcal L_{s}$, and $\lambda_s$ depends holomorphically on $s \in U$.
\end{theorem}

\subsection{Summability conditions}

Before continuing, we prove the following simple lemma:

\begin{lemma}
\label{sp-summability-lemma}

Suppose $\psi$ is a negative level 1 H\"older continuous potential. For any fixed $s, p \ge 0$, the quantities
$$
\sum_{a \in E} \, \inf_{[a]} |\psi^p e^{s\psi}|, \qquad \sum_{a \in E}\, \sup_{[a]} |\psi^p e^{s\psi}|, \qquad \int_{E_A^\infty} |\psi^p| dm_{s \psi}
$$
are comparable. In particular, if one of the quantities is finite, then they all are.
\end{lemma}

\begin{proof}
As $\psi < -C_1$ is bounded above by a negative constant and $\sup_{[a]} \psi - \inf_{[a]} \psi \le v_\alpha^{(1)}(\psi)$ for any $a \in E$, 
$$
\inf_{[a]} |\psi^p| \asymp \sup_{[a]} |\psi^p| \qquad \text{and} \qquad \inf_{[a]} |e^{s\psi}| \asymp \sup_{[a]} |e^{s\psi}|,
$$
which shows that the first two quantities are comparable. Since
$$
m_{s\psi}([a]) \, 
= \, \int_{E_A^\infty} \chi_{[a]} dm_{s\psi} \, 
= \, \lambda_{s\psi}^{-1}\int_{E_A^\infty} \mathcal L_{s\psi}  \chi_{[a]}  dm_{s\psi} \, \asymp \, \sup_{[a]} |e^{s\psi}|,
$$
the integral
$$
\int_{E_A^\infty} |\psi^p| dm_{s \psi} \, \asymp \, \sum_{a \in E} \sup_{[a]} |\psi^p| \cdot m_{s\psi}([a])
\, \asymp \,  \sum_{a \in E} \sup_{[a]} |\psi^p| \cdot \sup_{[a]} |e^{s\psi}|
$$
as desired. Here, we have used that the incidence matrix is finitely primitive to guarantee that $m_{s\psi} \bigl ( \{ \tau \in E^\infty_A : a \tau \text{ is admissible}\} \bigr )$  is bounded below by a positive constant, independent of $a \in E$.
\end{proof}

If one (and hence all) of the quantities in the above lemma is finite, then we say that the potential $\psi: E_A^\infty \to \mathbb{R}$ is {\em $(s,p)$-summable}\/. We will often write $p = n+\varepsilon$ where $n = \lfloor p \rfloor$ and $0 \le \varepsilon < 1$.

\subsection{Modified Transfer Operators}
\label{sec:sp-transfer}

In order to study the regularity of the operators $\mathcal L_s$ near the vertical line $\{\re s = 1\}$, we introduce the (modified) transfer operators
\begin{equation}
\label{eq:Lsp-def}
\mathcal L_{s,p} g \, = \, \mathcal L_{s} (\psi^p g), \qquad s \in \overline{\mathbb{C}_1^+}, \qquad p \ge 0.
\end{equation}
Using Lemma \ref{sp-summability-lemma}, it is not difficult to show that for a normal $(s, p)$-summable potential, the operator $\mathcal L_{s,p}$ is bounded on $ C_{\bounded}(E_A^\infty)$.
The following theorem describes the behaviour of $\mathcal L_{s,p}$ on the space of H\"older continuous functions $C^\alpha(E_A^\infty)$:

\begin{theorem}
\label{continuity-and-holomorphy-Lsp}
Suppose $\psi: E_A^\infty \to \mathbb{R}$ is a normal $(1, n+\varepsilon)$-summable potential.
Then, $s \to \mathcal L_s$ is

{\em (i)} a $C^{n+\varepsilon}_{\loc}$ mapping from $\overline{\mathbb{C}^+_1}$ to $\mathcal B(C^\alpha(E_A^\infty))$.

{\em (ii)} a holomorphic mapping from $\mathbb{C}_1^+$ to $\mathcal B(C^\alpha(E_A^\infty))$,
with 
\begin{equation}
\label{eq:differentiating-mto}
\frac{d^k}{ds^k}\, \mathcal L_{s} =  \mathcal L_{s, k}, \qquad k \in \mathbb{N}.
\end{equation}
\end{theorem}

\begin{proof}
The proof of (i) is quite cumbersome and is deferred to Appendix \ref{sec:continuity-of-sp-transfer-operators}. Fix a real number $q \ge 0$. Below, we show that the operator
$$
\mathcal L_{s,q}: \mathbb{C}_1^+ \to \mathcal B(C^\alpha(E_A^\infty))
$$ is holomorphic and $\mathcal L'_{s,q} = \mathcal L_{s,q+1}$.
As the series 
$$
\mathcal L_{s,q} g(\omega) \, =  \, \sum_{a \in E} (\psi^q e^{s\psi}) (a \omega) \cdot g(a \omega)
$$
 converges locally uniformly on the open half-plane $\mathbb{C}^+_1$, it can be differentiated term-by-term for any $s \in \mathbb{C}^+_1$ and $g \in C^\alpha(E_A^\infty)$. Hence, the {\em pointwise derivative}
 $$
\frac{d}{ds} \bigl ( \mathcal L_{s,q} g(\omega) \bigr ) = \mathcal L_{s,q+1} g(\omega).
$$

By Morera's theorem, to verify that the operator
$\mathcal L_{s,q}$ depends holomorphically on $s \in  \mathbb{C}_1^+$, it suffices to check that for every loop $\gamma \subset \mathbb{C}_1^+$, the integral
$
\int_\gamma \mathcal L_{s,q} ds
$
is the zero operator\footnote{In Appendix  \ref{sec:continuity-of-sp-transfer-operators}, we will see that
$s \to \mathcal L_{s,q}$ is continuous on $ \mathbb{C}_1^+$, so the operator-valued integral is well-defined.}.
 This is easy: as $\mathcal L_{s,q}$ is linear and bounded,
 $$
\biggl [
 \biggl ( \int_\gamma \mathcal L_{s,q} ds \biggr ) (g) 
 \biggr ] (\omega) \, = \,  \sum_{a \in E} \biggl ( \int_\gamma \psi^q e^{s\psi} ds \biggr ) (a \omega) \cdot g(a \omega) \, = \, 0,
 $$
by Cauchy's theorem. Having satisfied ourselves that $s \to \mathcal L_{s,q}$ is holomorphic as a mapping valued in $\mathcal B(C^\alpha(E_A^\infty))$,
it has a {\em Banach space derivative}\/, which must coincide with the pointwise derivative computed above.
\end{proof}

\begin{remark}
In general, the existence of the pointwise derivative does not imply the existence of a Banach space derivative, which is a much more stringent condition.
\end{remark}

\subsection{Projections and Eigenvalues}

Let $\psi: E_A^\infty \to \mathbb{R}$ be a normal potential. By Theorem \ref{continuity-and-holomorphy-Ls}, the operator $\mathcal L_1$ has a maximal positive eigenvalue $\lambda_1$. Since $\mathcal L_s$ is continuous in $s\in \overline{\mathbb{C}_1^+}$, simple isolated eigenvalues vary continuously in $s$. Therefore, we can follow $\lambda_s$ in a half-ball $\overline{\mathbb{C}_1^+} \cap B(1,\delta)$ for some $\delta > 0$.
For these parameters $s$, we can define the Riesz projection $R_s$ onto the corresponding 1-dimensional eigenspaces. We normalize the eigenfunction $\rho_1$ of $\mathcal L_1$ so that $\int_{E_A^\infty} \rho_1dm= 1$ and define 
$$
\rho_s:= R_s \rho_1.
$$

\begin{lemma}
\label{resolvent-diff-lemma}
If $\psi$ is a normal $(1, n+\varepsilon)$-summable potential, then $s \to R_s$ is a $C^{n+\varepsilon}$ mapping from $\overline{\mathbb{C}_1^+} \cap B(1, \delta)$ to $\mathcal B(C^\alpha(E_A^\infty))$.
\end{lemma}

\begin{proof}
We explicitly show the lemma for $p = n + \varepsilon \in (0, 2]$, which suffices for Orbit Counting. The general case is similar but more computationally involved.

\medskip

($p = 1$)  Since $\psi$ is centered, $\lambda_1 = 1$. After decreasing $\delta > 0$ if necessary, we may choose $\eta > 0$ so that  
for every $s \in \overline{\mathbb{C}_1^+} \cap B(1, \delta)$, 
\begin{equation}
\label{120251114}
\sigma(\mathcal L_s) \cap \overline{B(1, \eta)} = \{ \lambda_s \}
\end{equation} 
consists of a simple isolated eigenvalue.
Set $\gamma = \partial B(1, \eta)$. For any $s \in \overline{\mathbb{C}_1^+} \cap B(1, \delta)$, the Riesz projection onto the $1$-dimensional eigenspace of $\lambda_s$ is given by
$$
R_s := \frac{1}{2\pi i} \int_{\gamma} (z \id - \mathcal L_s)^{-1} dz,
$$
and this definition is independent of the choice of $\eta$ as long as (\ref{120251114}) holds. By the second resolvent formula,
\begin{equation}
\label{eq:Rts-difference}
R_t - R_s = \frac{1}{2\pi i} \int_{\gamma} (z \id - \mathcal L_t)^{-1} (\mathcal L_t - \mathcal L_s) (z \id - \mathcal L_s)^{-1} dz,
\end{equation}
for all $s, t \in \overline{\mathbb{C}_1^+} \cap B(1, \delta)$. Using Theorem \ref{continuity-and-holomorphy-Lsp} to justify differentiation under the integral sign, we get
\begin{equation}
\label{eq:first-derivative-resolvent}
R'_s \, := \, \lim_{t \to s} \frac{R_t - R_s}{t - s} \, = \, \frac{1}{2\pi i} \int_{\gamma} (z \id - \mathcal L_s)^{-1}  \mathcal L'_s (z \id - \mathcal L_s)^{-1} dz.
\end{equation}
Since $s \to \mathcal L_s$ is a $C^1$ mapping from $\overline{\mathbb{C}_1^+}$ to $\mathcal B(C^\alpha(E_A^\infty))$ by Theorem \ref{continuity-and-holomorphy-Lsp}, the expression above shows that the mapping
$s \to R'_s$
from $\overline{\mathbb{C}_1^+} \cap B(1, \delta)$ to $\mathcal B(C^\alpha(E_A^\infty))$ is continuous. 

\medskip

($0 < p < 1$) 
Using the representation (\ref{eq:Rts-difference}) for $R_t - R_s$, we can write
$$
\frac{R_t - R_s}{(t - s)^\varepsilon} = \frac{1}{2\pi i} \int_{\gamma} (z \id - \mathcal L_t)^{-1} \, \frac{\mathcal L_t - \mathcal L_s}{(t-s)^\varepsilon} \, (z \id - \mathcal L_s)^{-1} dz.
$$
By virtue of Theorem \ref{continuity-and-holomorphy-Lsp}, the integrand is uniformly bounded in $\mathcal B(C^\alpha(E_A^\infty))$ and converges to $0$ in $\mathcal B(C^\alpha(E_A^\infty))$ as $t\to s$. Consequently, the mapping $s\to R_s$ from $\overline{\mathbb{C}_1^+} \cap B(1, \delta)$ to $\mathcal B(C^\alpha(E_A^\infty))$ is  $C^\varepsilon$.


\medskip

$(p=2)$ 
By (\ref{eq:first-derivative-resolvent}) and the second resolvent identity, we can write
\begin{align*}
R'_t - R'_s 
&= \frac{1}{2\pi i} \int_{\gamma} (z \id - \mathcal L_t)^{-1} \mathcal L'_t 
      \bigl[ (z \id - \mathcal L_t)^{-1} - (z \id - \mathcal L_s)^{-1} \bigr] dz
\notag\\[2mm]
&\quad + \frac{1}{2\pi i} \int_{\gamma} (z \id - \mathcal L_t)^{-1} 
      (\mathcal L'_t - \mathcal L'_s) (z \id - \mathcal L_s)^{-1} dz
\notag\\[2mm]
&\quad + \frac{1}{2\pi i} \int_{\gamma} 
      \bigl[ (z \id - \mathcal L_t)^{-1} - (z \id - \mathcal L_s)^{-1} \bigr] 
      \mathcal L'_s (z \id - \mathcal L_s)^{-1} dz
\notag\\[1mm]
&= \frac{1}{2\pi i} \int_{\gamma} (z \id - \mathcal L_t)^{-1} \mathcal L'_t 
      (z \id - \mathcal L_t)^{-1} (\mathcal L_t - \mathcal L_s) (z \id - \mathcal L_s)^{-1} dz
\notag\\[1mm]
&\quad + \frac{1}{2\pi i} \int_{\gamma} (z \id - \mathcal L_t)^{-1} 
      (\mathcal L'_t - \mathcal L'_s) (z \id - \mathcal L_s)^{-1} dz
\notag\\[1mm]
&\quad + \frac{1}{2\pi i} \int_{\gamma} (z \id - \mathcal L_t)^{-1} 
      (\mathcal L_t - \mathcal L_s) (z \id - \mathcal L_s)^{-1} 
      \mathcal L'_s (z \id - \mathcal L_s)^{-1} dz.
\end{align*}
Dividing by $t-s$ and using Theorem \ref{continuity-and-holomorphy-Lsp} to justify taking the limit as $t \to s$, we obtain
\begin{equation}
\label{eq:second-derivative-resolvent}
\begin{aligned}
R''_s  
&= \frac{2}{2\pi i} \int_{\gamma} 
      (z \id - \mathcal L_s)^{-1} \, \mathcal L'_s \, 
      (z \id - \mathcal L_s)^{-1} \, \mathcal L'_s \, (z \id - \mathcal L_s)^{-1} \, dz
\notag\\[1mm]
&\quad + \frac{1}{2\pi i} \int_{\gamma} 
      (z \id - \mathcal L_s)^{-1} \, \mathcal L''_s \, (z \id - \mathcal L_s)^{-1} \, dz.
\end{aligned}
\end{equation}
Applying Theorem \ref{continuity-and-holomorphy-Lsp} again, we conclude that the mapping  $s\to R_s''$ from $\overline{\mathbb{C}_1^+} \cap B(1, \delta)$ to $\mathcal B(C^\alpha(E_A^\infty))$ is continuous.


\medskip

$(1 < p < 2)$ Using the formula for $R'_t - R'_s$ above, we can express
$$
\frac{R'_t - R'_s}{(t - s)^\varepsilon} 
$$
as an integral whose integrand is uniformly bounded in $\mathcal B(C^\alpha(E_A^\infty))$ and converges to $0$ in $\mathcal B(C^\alpha(E_A^\infty))$ as $t\to s$. Therefore, the mapping $s\to R_s$ from $\overline{\mathbb{C}_1^+} \cap B(1, \delta)$ to $\mathcal B(C^\alpha(E_A^\infty))$ is  $C^{1+\varepsilon}$.
\end{proof}

\begin{corollary}
\label{eigenvalue-diff-lemma}
If $\psi$ is a normal $(1, n+\varepsilon)$-summable potential, then $s \to \lambda_s$ is a $C^{n+\varepsilon}$ function on $\overline{\mathbb{C}_1^+} \cap B(1,\delta)$.
\end{corollary}

\begin{proof}
We explicitly show the corollary for when $p = n$ is an integer, and when $p = n + \varepsilon \in (1, 2)$, which suffices for Orbit Counting. The general case follows the same pattern.

($\varepsilon = 0$) Let $F: C^\alpha(E_A^\infty) \to \mathbb{C}$ be a bounded linear functional for which $F(\rho_1) \ne 0$.
The equality $\mathcal L_s R_s \rho_1 = \lambda_s R_s \rho_1$ tells us that
$$
\lambda_s = \frac{F(\mathcal L_s R_s \rho_1)}{F(R_s \rho_1)}.
$$
The corollary now follows from differentiability properties of $\mathcal L_s$ and $R_s$ discussed in Theorem \ref{continuity-and-holomorphy-Lsp} and Lemma \ref{resolvent-diff-lemma}.

($1 < n+\varepsilon < 2$) To see that $s \to \lambda_s$ is $C^{1+\varepsilon}$ function on $\overline{\mathbb{C}_1^+} \cap B(1, \delta)$, it suffices to show that the derivative is
$$
\lambda'_s = \frac{F(\mathcal L'_s (R_s \rho_1) \cdot R'_s \rho_1) F(R_s \rho_1) - F(\mathcal L_s R_s \rho_1) F(R'_s \rho_1)}{F(R_s \rho_1)^2}.
$$
a $C^\varepsilon$ function. This follows from the fact that $C^\varepsilon$ functions and operators form an algebra.
\end{proof}

\subsection{Stochastic Laws}

We compute the one-sided derivative of pressure at $s=1$ by following the argument in  \cite[Proposition 4.10]{PP}:

\begin{lemma}
\label{first-derivative-of-pressure}
If $\psi$ is a normal $(1, 1)$-summable potential, then
\begin{equation}
\label{eq:dlambda-1}
\lambda'(1) = P'(1) := \lim_{s \to 1, \, s \in \mathbb{C}_1^+} P'(s) = \int_{E_A^\infty}  \psi d\mu_1.
\end{equation}
\end{lemma}

\begin{proof}
Since $\rho_s = R_s \rho_1$, we can use Lemma \ref{resolvent-diff-lemma} to conclude the existence of the one-sided derivative
$$
\rho'_1(\xi) := \lim_{s \to 1, \, s \in \mathbb{C}_1^+} \rho'_s(\xi).
$$
Differentiating $\mathcal L_s \rho_s = e^{P(s)} \rho_s$ at $s \in \mathbb{C}_1^+$ and taking $s \to 1$, we get
$$
\mathcal L'_1 \rho_1 + \mathcal L \rho'_1 = P'(1) \rho_1 + \rho'_1.
$$
After integrating the above equation with respect to $m_1$ and using (\ref{eq:conformal-measures-and-transfers}), we obtain
$$
\int_{E_A^\infty}  \mathcal L'_1 \rho_1 dm_1 = P'(1) \int_{E_A^\infty}  \rho_1 dm_1.
$$
By (\ref{eq:differentiating-mto}) and the fact that $\mu_1 = \rho_1 dm_1$ was designed to be a probability measure, the above equality simplifies to
$$
\int_{E_A^\infty}  \mathcal L_1 (\psi \rho_1) dm_1  = P'(1).
$$
Applying the identity (\ref{eq:conformal-measures-and-transfers}) one more time, we end up with
$$
\int_{E_A^\infty}  \psi d\mu_1 \, = \, \int_{E_A^\infty}  \psi \rho_1 dm_1 \, = \,  P'(1),
$$
which is what we wanted to show.

The equality $\lambda'(1) = P'(1)$ follows from the definition of pressure as $\log \lambda$ and the normalization $P(1) = 0$.
\end{proof}

In a similar fashion, one can follow the proof of \cite[Proposition 4.11]{PP} to compute the one-sided second derivative of pressure at $s=1$:

\begin{lemma}
\label{second-derivative-of-pressure}
Suppose $\psi$ is a normal $(1, 2)$-summable potential on $E_A^\infty$. Let $\overline{\psi} = \int_{E_A^\infty} \psi d\mu_1$
and $\phi = \psi - \overline{\psi}$.
Consider the family of potentials $\psi_s = \psi + (s-1) \phi$ defined on $\overline{\mathbb{C}^+_1}$. The one-sided second
derivative of pressure is
\begin{equation}
\label{eq:dlambda-2}
P''(1) := \lim_{s \to 1, \, s \in \mathbb{C}_1^+} P''(s) = \sigma^2(\phi),
\end{equation}
where
$$
\sigma^2(\phi) = \lim_{n \to \infty} \frac{1}{n} \int_{\partial \mathbb{D}}  (S_n \phi)^2 d\mu
$$
is the {\em asymptotic variance} of $\phi$.
\end{lemma}

\begin{proof}
Arguing as in the proof of Lemma \ref{first-derivative-of-pressure} shows that $P'(1) = \int_{E_A^\infty}  \phi d\mu_1 = 0$.
In view of the $(1,2)$-summability hypothesis,  Lemma \ref{resolvent-diff-lemma} implies the existence of the one-sided derivatives
$$
\rho'_1(\xi) := \lim_{s \to 1, \, s \in \mathbb{C}_1^+} \rho'_s(\xi), \qquad \rho''_1(\xi) := \lim_{s \to 1, \, s \in \mathbb{C}_1^+} \rho''_s(\xi).
$$
Taking the second derivative of  $\mathcal L_s \rho_s = e^{P(s)} \rho_s$ and tending $s \to 1$ in $\mathbb{C}_1^+$ gives
$$
\mathcal L''_1 \rho_1 + 2 \mathcal L'_1 \rho'_1 + \mathcal L_1 \rho''_1 = P''(1) \rho_1 + \rho''_1.
$$
Integrating with respect to $m_1$, we obtain
$$
\int_{E_A^\infty}  \bigl \{ \mathcal L''_1 \rho_1 + 2 \mathcal L'_1 \rho'_1  \bigr \} dm_1 = P''(1),
$$
which simplifies to
$$
\int_{E_A^\infty} \phi^2 d\mu_1 + 2 \int_{E_A^\infty} \rho'_1 \cdot  \phi dm_1 = P''(1).
$$
Replacing $\phi$ by $S_n \phi$ and dividing by $n$, we get
$$
\frac{1}{n} \int_{E_A^\infty} (S_n \phi)^2 d\mu_1 + 2 \int_{E_A^\infty} \frac{\rho'_1}{\rho_1} \cdot  \frac{S_n \phi}{n} d\mu_1 = P''(1).
$$
The result follows after taking $n \to \infty$ and applying the ergodic theorem.
\end{proof}

Once we know the existence of the second derivative of the pressure along the imaginary axis, one can conclude the Central Limit Theorem for $\phi$ as in \cite[Theorem 13.9.15]{URM22}. For a $(1,3)$-summable potential, one could follow the argument of \cite[Theorem 4.13]{PP} to obtain the Central Limit Theorem for $\phi$ with an $O(1/\sqrt{n})$ error term.

\section{Orbit Counting in Symbolic Dynamics}
\label{sec:orbit-counting}

In this section, we prove our main results in symbolic dynamics, Theorems \ref{orbit-counting-tdf} and \ref{orbit-counting-tdf2}.
We follow the approach in \cite{PU}, paying closer to attention to the behaviour of the Poincar\'e series near the line $\re s = 1$.

\subsection{Poincar\'e Series}
\label{sec:poincare}

For a normal potential $\psi: E^\infty_A \to (-\infty, 0)$ and a point $\xi \in E^\infty_A$, the {\em counting function with offset $\varphi  \in C^\alpha(E^\infty_A)$} is given by
$$
N^{\varphi, \psi}_\xi(T) := \# \bigl \{ \omega \in E^*_\xi \, :\, - S_{|\omega|} \psi (\omega \xi) - \varphi(\omega \xi) \le T \bigr \},
$$
where $E^*_\xi \subset E^*_A$ is the set of finite words $\omega$ such that $\omega \xi$ is admissible. We will often suppress $\varphi$ and $\psi$ from the notation and abbreviate $f = f_1 = e^\varphi$ and $f_s = e^{s\varphi}$.

\begin{example}
(i) When the offset $\varphi = {\bf 0}$, $N^{\varphi, \psi}_\xi(T)$ counts the pre-images of $\xi$ under the shift map $\sigma: E_A^\infty \to E_A^\infty$:
$$
N_\xi(T) := \# \bigl \{ \omega \in E^*_\xi \, : \, S_{|\omega|} (-\psi) (\omega \xi) \le T \bigr \}.
$$
(ii)
Let $\tau \in E_A^*$ be a finite word. Taking $\varphi(\zeta) = S_{|\tau|}\psi ( \tau \zeta)$ in the definition above, we get
 $$
 N^{[\tau]}_\xi(T) := \# \bigl \{ \omega \in E^*_\xi \text{ starting with } \tau \, : \, S_{|\omega|} (-\psi) (\omega \xi) \le T \bigr \}.
 $$
Up to  bounded error, this counts the number of pre-images of $\xi$ under the shift map $\sigma: E_A^\infty \to E_A^\infty$ that lie in $[\tau]$:
 $$
 \# \bigl \{ \omega \in E^*_\xi \, : \, \omega \xi \in [\tau] \text{ and } S_{|\omega|} (-\psi) (\omega \xi) \le T \bigr \}.
 $$
(Since $\psi$ is negative and summable, only finitely many shorts words $\omega \in E^*_\xi$ of length less than $|\tau|$ satisfy $S_{|\omega|} (-\psi) (\omega \xi) \le T$.)
\end{example}

Following \cite{PU}, to study the asymptotics of $N_\xi(T)$ as $T \to \infty$, we examine the  {\em Poincar\'e series}
$$
\eta(s) \, := \, \sum_{n=0}^\infty \mathcal L_s^n (e^{s\varphi}) 
\, = \,  \int_0^\infty e^{-sT} dN_\xi(T)
\, = \, \sum_{ \omega \in E^*_A} \exp \bigl ( s ( S_{|\omega|} \psi + \varphi ) \bigr ) (\omega \xi).
$$

\begin{lemma}
\label{behaviour-of-poincare-series}
Suppose $\psi$ is a robust $(1,\,1+\varepsilon)$-summable potential, with $0 \le \varepsilon \le 1$. For any $\xi \in E^\infty_A$, the following statements hold:

{\em (1)} The series defining $\eta(s)(\xi)$ converges uniformly and absolutely on compact subsets of $\mathbb{C}_1^+$ and defines a holomorphic function.

{\em (2)} The function $s \to\eta(s)(\xi)\in \mathbb{C}$ has a continuous extension to 
$\overline{\mathbb{C}_1^+} \setminus \{1 \}$.

{\em (3, $\varepsilon = 0$)} As $s \to 1$ in $\mathbb{C}_1^+$,
$$
\eta(s)(\xi) - \frac{(-\lambda'_1)^{-1} R_1f(\xi) }{s-1} = o(|s-1|^{-1}).
$$

{\em (3, $0 < \varepsilon < 1$)} As $s \to 1$ in $\mathbb{C}_1^+$,
$$
\eta(s)(\xi) - \frac{(-\lambda'_1)^{-1} R_1f(\xi) }{s-1} = O(|s-1|^{\varepsilon-1}),
$$

{\em (3, $\varepsilon = 1$)} As $s \to 1$ in $\mathbb{C}_1^+$,
$$
\eta(s)(\xi) - \frac{(-\lambda'_1)^{-1} R_1f(\xi) }{s-1}
$$
tends to a finite limit.
\end{lemma}

\begin{remark}As $\psi$ is a negative potential, $\lambda'_1 \ne 0$ by Lemma \ref{first-derivative-of-pressure}, so the expression in (3) is well-defined.
\end{remark}

We temporarily assume Lemma \ref{behaviour-of-poincare-series} and prove Theorem \ref{orbit-counting-tdf}:

\begin{proof}[Proof of Theorem \ref{orbit-counting-tdf} assuming Lemma \ref{behaviour-of-poincare-series}]

By Lemma \ref{behaviour-of-poincare-series} and the hypotheses of the theorem, for any $\xi \in E_A^\infty$,
$$
\eta(s)(\xi) - \frac{(-\lambda'_1)^{-1} R_1f(\xi) }{s-1}
$$
 extends to an $L^1_{\loc}$ function on the vertical line $\{\re s=1\}$.
By the Wiener-Ikehara Tauberian theorem (Theorem \ref{wiener-ikehara}),
\begin{equation*}
\lim_{T \to \infty} \frac{N_\xi(T)}{e^T} = -(\lambda'_1)^{-1} R_1f(\xi).
\end{equation*}
Inserting the expression for $\lambda'_1$ from Lemma \ref{first-derivative-of-pressure}, we get
\begin{equation}
\label{eq:orbit-counting-proof}
\lim_{T \to \infty} \frac{N_\xi(T)}{e^T}  = 
\frac{ \rho_\psi(\xi)}{\int_{E_A^\infty} (-\psi) d\mu_\psi } \cdot \int_{E_A^\infty} e^{\varphi (\zeta)} dm_\psi(\zeta).
\end{equation}
Setting $\varphi = {\bf 0}$ in (\ref{eq:orbit-counting-proof}) shows that the theorem holds for $B = E_A^\infty$:
$$
\lim_{T \to \infty} \frac{N_\xi(T)}{e^T} = \frac{ \rho_\psi(\xi) }{\int_{E_A^\infty} (-\psi) d\mu_\psi}.
$$
More generally, setting $\varphi(\zeta) = S_{|\tau|}\psi ( \tau \zeta)$ into (\ref{eq:orbit-counting-proof}), shows that the theorem holds for cylinders $B = [\tau]$:
$$
\lim_{T \to \infty} \frac{N^{[\tau]}_\xi(T)}{e^T} =  \frac{ \rho_\psi(\xi) }{\int_{E_A^\infty} (-\psi) d\mu_\psi} \cdot m_\psi([\tau]),
$$

Clearly, the theorem  is also valid when $B  \subset E_A^\infty$ is a finite union of disjoint cylinders. We next consider the case when $B$ is an open set with $m _\psi(\partial B) = 0$. For any $\varepsilon > 0$, we can find a set $B_1 \subset B$ which is a finite union of disjoint cylinders with
$m_\psi(B_1) \ge m_\psi(B) - \varepsilon$. Clearly,
$$
\lim_{T \to \infty} \frac{N^{B}_\xi(T)}{e^T} \ge  \frac{ \rho_\psi(\xi) }{\int_{E_A^\infty} (-\psi) d\mu_\psi} \cdot (m_\psi(B) - \varepsilon).
$$
Taking $\varepsilon \to 0$ shows the lower bound
$$
\lim_{T \to \infty} \frac{N^{B}_\xi(T)}{e^T} \ge  \frac{ \rho_\psi(\xi) }{\int_{E_A^\infty} (-\psi) d\mu_\psi} \cdot m_\psi(B).
$$
The above formula for the open set $E_A^\infty \setminus \overline{B}$ gives the corresponding upper bound.

Since complements of open sets are closed sets, the theorem also holds for closed sets $B$ with $m_\psi(\partial B) = 0$.
From here, it is easy to see that the theorem holds for an arbitrary measurable set $B \subset E_A^\infty$ with $m_\psi(\partial B) = 0$ by squeezing $B$ between its interior and closure.
\end{proof}

The argument above also shows Theorem \ref{orbit-counting-tdf2} for D-generic potentials, provided that one uses Hardy-Littlewood Tauberian theorem (Theorem \ref{hardy-littlewood}) instead of the Wiener-Ikehara Tauberian theorem (Theorem \ref{wiener-ikehara}). Note that the $\varepsilon = 0$ case of Lemma \ref{behaviour-of-poincare-series} implies that
$$
\eta(s)(\xi) \sim \frac{(-\lambda'_1)^{-1} R_1f(\xi) }{s-1}, 
$$
as the numerator is non-zero.

\subsection{Orbit Counting for D-Generic Potentials}
\label{sec:behaviour-near-1}

To complete the proof of Theorem \ref{orbit-counting-tdf}, it remains to prove Lemma \ref{behaviour-of-poincare-series}.

\begin{proof}[Proof of Lemma \ref{behaviour-of-poincare-series}]
By Theorem~\ref{complex-potentials} and the D-genericity of $\psi$, the spectral radius of $\mathcal L_s$ is less than $1$ for every $s \in \overline{\mathbb{C}_1^+} \setminus \{1 \}$. Since parts (1) and (2) follow immediately, we focus on part (3).
 
 \medskip
  
{\em Step 1.} Consider the operator
$$
\Delta_s = \mathcal L_s - \lambda_s R_s.
$$
Since $\mathcal L_s R_s = R_s \mathcal L_s$ and $R_s$ is the Riesz projection onto $\ker(\lambda_s \id - \mathcal L_s)$,
$$
R_s \Delta_s = \Delta_s R_s = 0.
$$
Consequently,
$$
\mathcal L_s^n = \lambda_s^n R_s + \Delta_s^n, \qquad n \ge 0.
$$
Since $  \| \Delta_1 \|_\alpha < 1$, after shrinking $\delta > 0$ if necessary, we  may assume that $ \|  \Delta_s \|_\alpha < 1$ for all $s \in \overline{\mathbb{C}^+_1} \cap B(1, \delta)$. This ensures that the series
$$
s \to \Delta_\infty(s) :=   \sum_{n=0}^\infty \Delta_s^n (f_s)
$$
converges uniformly in $\mathcal B(C^\alpha(E^\infty_A))$ and thus defines a continuous function from $\overline{\mathbb{C}^+_1} \cap B(1, \delta)$ to $\mathcal B(C^\alpha(E^\infty_A))$.

\medskip

{\em Step 2.}
Since $|\lambda_s| < 1$ for all $\mathbb{C}^+_1 \cap B(1, \delta)$, we can write
$$
   \eta(s) = \sum_{n=0}^\infty \lambda_s^n R_s f_s + \sum_{n=0}^\infty \Delta_s^n(f_s) = (1-\lambda_s)^{-1} R_sf_s + \Delta_\infty(s).
   $$
A computation shows that
   \begin{align*}
   \eta(s) - \frac{(-\lambda'_1)^{-1} R_1f_1}{s-1} & = \frac{(1-\lambda_s)^{-1} R_sf_s(s-1) + \Delta_\infty(s)(s-1) + (\lambda'_1)^{-1} R_1f_1}{s-1} \\
 & = \frac{(1-\lambda_s)^{-1}R_sf_s(s-1) + (\lambda'_1)^{-1}R_1f_1}{s-1} + \Delta_\infty(s).
   \end{align*}
 We denote the first summand in this formula by $A(s)$. Since $\Delta_\infty(s)$ is continuous at $s =1$, to prove the lemma, 
 we only need to analyze $A(s)$.
  
  \medskip
  
  {\em Step 3.}
 We compute:
$$
  \begin{aligned}
A(s) & = \frac{(R_s f_s - R_1 f_ 1)(s-1) + R_1f_1 \bigl ((s-1) + (\lambda'_1)^{-1}(1-\lambda_s) \bigr )}{(s-1)(1-\lambda_s)} 
\\
  & = \frac{R_s(f_s-f_1)}{s-1} \cdot \frac{s-1}{1-\lambda_s} + 
  \frac{(R_s-R_1)f_1}{s-1} \cdot \frac{s-1}{1-\lambda_s}
  \\
& \  \  \  \  \  \  \  \  \  \  \  \  \  \  \ + R_1f_1 \cdot \frac{s-1+(\lambda'_1)^{-1}(1-\lambda_s)}{(s-1)(1-\lambda_s)} \\
& = A_1(s) + A_2(s) + A_3(s).
  \end{aligned}
$$
In view of Lemmas \ref{resolvent-diff-lemma} and \ref{eigenvalue-diff-lemma}, we have
  $$
  \lim_{\overline{\mathbb{C}_1^+} \cap B(1,\delta_1) \ni s \to 1} \frac{R_s(f_s-f_1)}{s-1} = R_1(f_1 \log f_1) \in C^\alpha(E^\infty_A),
  $$
$$
  \lim_{\overline{\mathbb{C}_1^+} \cap B(1,\delta_1) \ni s \to 1}  \frac{(R_s - R_1)f_1}{s-1} = R_1' f_1 \in C^\alpha(E^\infty_A)
$$
and
$$
  \lim_{\overline{\mathbb{C}_1^+} \cap B(1,\delta_1) \ni s \to 1}  \frac{s-1}{1-\lambda_s} = - \frac{1}{\lambda'_1},
$$
so that $A_1(s)$ and $A_2(s)$ are continuous at $s=1$.
Hence,
$$
 \eta(s) - \frac{(-\lambda'_1)^{-1} R_1f_1}{s-1} =  R_1f_1 \cdot \frac{s-1+(\lambda'_1)^{-1}(1-\lambda_s)}{(s-1)(1-\lambda_s)}
 + B(s),
 $$
 for some function $B(s)$ defined on $\overline{\mathbb{C}^+_1} \cap B(1, \delta)$ which is continuous at $s = 1$.

\medskip

{\em Step 4.} Finally, we use Lemma \ref{eigenvalue-diff-lemma} to expand $\lambda_s$ near $s = 1$: 

\medskip

$(\varepsilon = 0)$ If $\psi$ is $(1,1)$-summable, then
$$
\lambda_s = 1 + \lambda'_1(s-1) + o(|s-1|), \qquad \text{as }s \to 1 \text { in } \mathbb{C}^+_1,
$$
and
$$
\frac{s-1+(\lambda'_1)^{-1}(1-\lambda_s)}{(s-1)(1-\lambda_s)} = o(|s-1|^{-1}).
$$

$(0 < \varepsilon < 1)$ If $\psi$ is $(1, 1+\varepsilon)$-summable for some $0 < \varepsilon < 1$, then
$$
\lambda_s = 1 + \lambda'_1(s-1) + O(|s-1|^{1+\varepsilon}),  \qquad \text{as }s \to 1 \text { in } \mathbb{C}^+_1,
$$
and
$$
\frac{s-1+(\lambda'_1)^{-1}(1-\lambda_s)}{(s-1)(1-\lambda_s)} = O(|s-1|^{\varepsilon-1}).
$$

$(\varepsilon = 1)$
If $\psi$ is (1, 2)-summable, then
$$
\lambda_s = 1 + \lambda'_1(s-1) + \frac{1}{2}\bigl ( \lambda''_1 + o(1) \bigr ) (s-1)^2,  \qquad \text{as }s \to 1 \text { in } \mathbb{C}^+_1,
$$
and
$$
\lim_{\overline{\mathbb{C}_1^+} \cap B(1,\delta_1) \ni s \to 1}  \frac{s-1+(\lambda'_1)^{-1}(1-\lambda_s)}{(s-1)(1-\lambda_s)} = \frac{1}{2} \cdot \frac{\lambda''_1}{(\lambda'_1)^2}.
$$
The proof is complete.
\end{proof}

\subsection{Orbit Counting for Non-D-Generic Potentials}

We now turn to the proof of Theorem \ref{orbit-counting-tdf2}. Since we have already established the theorem for D-generic potentials in Section \ref{sec:poincare}, we now address the situation when $\psi$ is not D-generic.

Suppose $\psi$ is a normal (1,\,1)-summable potential. If $\psi$ is not D-generic, then the limit
 $$
 \lim_{T \to \infty} \frac{N^{B}_\xi(T)}{e^T}
 $$
 may not exist in the traditional sense. Below, we show that the limit exists after Ces\`aro averaging as in Theorem \ref{orbit-counting-tdf2}. 
 
According to Theorem \ref{complex-potentials},
$$
\psi = \phi + w - w \circ \sigma,
$$
for some function $\phi$ which takes values in a discrete subgroup $a\mathbb{Z} \subset \mathbb{R}$. We assume that $a$ has been chosen so that $\phi$ does not take values in any proper subgroup of  $a\mathbb{Z}$.
Since $\psi$ is negative, $\phi$ cannot be identically zero and hence $a \ne 0$.
Theorem \ref{complex-potentials} also tells us that the operator $\mathcal L_\psi$ has spectral radius less than 1 for all $s \in \overline{\mathbb{C}_1^+} \setminus \{1 + (2\pi i/a) \mathbb{Z} \}$.

Arguing as in the proof of Lemma \ref{behaviour-of-poincare-series} shows:

\begin{lemma}
\label{behaviour-of-poincare-series2}
Suppose $\psi$ is a normal $(1,\,1)$-summable potential, which is cohomologous to a function $\phi$ which takes values in a discrete subgroup $a\mathbb{Z} \subset \mathbb{R}$ (and in no smaller discrete subgroup). For any $\xi \in E^\infty_A$, the following statements hold:

{\em (1)} The series defining $\eta(s)(\xi)$ converges uniformly and absolutely on compact subsets of $\mathbb{C}_1^+$ and defines a holomorphic function.

{\em (2)} The function $s \to\eta(s)(\xi)\in \mathbb{C}$ has a continuous extension to 
$\overline{\mathbb{C}_1^+} \setminus \{1 + (2\pi i /a)k : k \in \mathbb{Z} \}$.

{\em (3)} As $s \to 1$ in $\mathbb{C}_1^+$,
$$
\eta(s)(\xi) - \frac{(-\lambda'_1)^{-1} R_1f(\xi) }{s-1} = o(|s-1|^{-1}).
$$
\end{lemma}

From here, one can deduce Theorem \ref{orbit-counting-tdf2} from the Hardy-Littlewood Tauberian theorem as in the D-generic case. We leave the details to the reader.

\part{One Component Inner Functions}
\label{COCIFSP}

\section{Fundamental Properties and Examples}

Following W.~Cohn \cite{cohn}, an inner function $F$ is a {\em one component inner function} if the set $\{ z \in \mathbb{D} : |F(z)| < r \}$ is connected for some $0 < r < 1$.

\begin{figure}[h]
\centering
\includegraphics[scale=0.65]{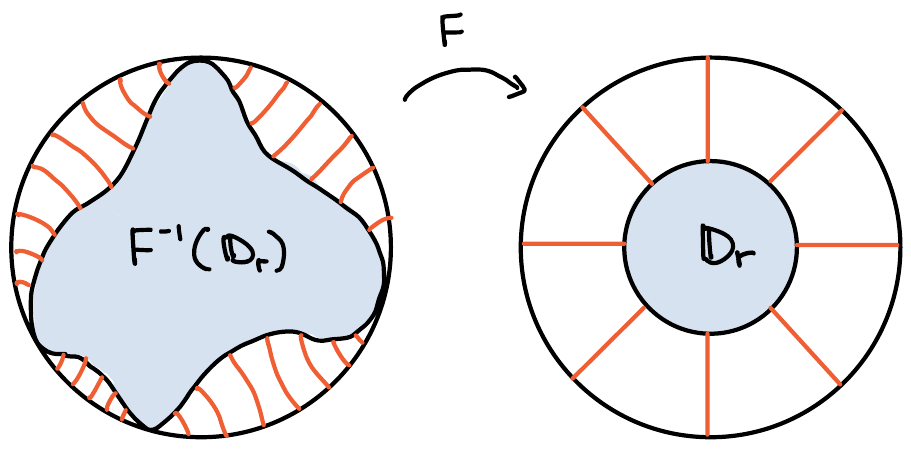}
 \caption{Anatomy of a one component inner function.}
\label{fig:one-component-anatomy}
\end{figure}

In his paper, Cohn gave a geometric description of one component inner functions, which is more suitable for applications in dynamical systems, see Fig.~\ref{fig:one-component-anatomy} above:
 
\begin{lemma}[Cohn]
\label{one-component-anatomy}
Let $F$ be a one component inner function and assume that $0 < r < 1$ is chosen so that  $F^{-1}(\mathbb{D}_r)$ is connected.
Then the region $F^{-1}(\mathbb{D}_r)$ is a Jordan domain. If $F^{-1}(\mathbb{D}_r)$ is compactly contained in the unit disk, then $F$ is a finite Blaschke product. Otherwise, $\mathbb{D} \setminus F^{-1}(\mathbb{D}_r)$ consists of countably many simply-connected regions $\{\Omega_i\}$. On each region $\Omega_i$,  $F$ acts as a universal covering map to the annulus $A(0; r, 1)$. The map $F$ is also a universal covering map from $\widetilde{\Omega}_i$ to the annulus $A(0; r, 1/r)$, where $\widetilde{\Omega}_i$ is the double of $\Omega_i$ across the unit circle.
\end{lemma}

We say that an inner function $F$ is {\em singular} at a point $\zeta \in \partial \mathbb{D}$ if it does not admit an analytic extension to a neighbourhood of $\zeta$. 
Let $\Sigma \subset \partial \mathbb{D}$ be the set of analytic singularities of $F$. It is clear from this definition that $\Sigma$ is a closed set.
While one usually thinks of inner functions as holomorphic self-maps of the unit disk, one may also view $F$ as a meromorphic function 
on $ \hat{\mathbb{C}} \setminus \Sigma$.

The following theorem provides a converse to Lemma \ref{one-component-anatomy}, as well as a number of other characterizations of one component inner functions:

\begin{theorem}
\label{characterizations-of-one-component-functions}
Let $F: \mathbb{D} \to \mathbb{D}$ be an inner function. The following conditions are equivalent:

{\em (a)} The set $F^{-1}(B(0,r))$ is connected for some $0 < r < 1$.

{\em (b)} There is an annulus $A = A(0; \rho, 1)$ such that $F: \mathbb{D} \to \mathbb{D}$ is a covering map over $A$.

{\em (c)} There is an annulus $\widetilde{A} = A(0; \rho, 1/\rho)$ such that $F: \sph \setminus \Sigma \to \sph$ is a covering map over $\widetilde{A}$.

{\em (d)} The singular set $\Sigma \subset \partial \mathbb{D}$ has Lebesgue measure $0$, the derivative $F'(\zeta) \to \infty$ as $\zeta \in \partial \mathbb{D}\setminus \Sigma$ approaches $\Sigma$, and
\begin{equation}
\label{eq:analytic-definition}
\biggl | \frac{F''(\zeta)}{F'(\zeta)^2} \biggr | \le C, \qquad \zeta \in \partial \mathbb{D} \setminus \Sigma,
\end{equation}
for some constant $C > 0$.

If any of the above conditions hold, then $\Sigma$ is equal to the set of points on the unit circle where the radial limit of $F$ does not exist or is not unimodular.
\end{theorem}

\begin{remark}
(i) We say that  $v \in \mathbb{D}$ is a {\em regular value} of an inner function $F: \mathbb{D} \to \mathbb{D}$ if $F$ is a covering map over some open neighbourhood of $v$.
A point in the unit disk which is not a regular value is called a {\em singular value}\/. 

(ii) The equivalence (a) $\Leftrightarrow$ (d)  is due to Aleksandrov  \cite{aleksandrov}. 
 For additional characterizations of one component inner functions in terms of Carleson boxes and Aleksandrov-Clark measures, we refer the reader to \cite{NR}
and \cite{bessonov} respectively.
\end{remark}

To prove Theorem \ref{characterizations-of-one-component-functions}, we will use several facts about inner functions and cluster sets.
We begin by recalling \cite[Lemma 1.3]{IK22}:

\begin{lemma}
\label{preimage-components}
Let $F$ be an inner function. If $V$ is a Jordan domain compactly contained in the unit disk, then any connected
component of $F^{-1}(V)$ is a Jordan domain.
\end{lemma}

For a function $g$ on the unit disk and a point $\zeta \in \partial \mathbb{D}$, the {\em cluster set} $\Cl(g, \zeta)$ consists of all possible limits
of $g(z_n)$, for sequences $z_n \in \mathbb{D}$ converging to $\zeta$. The following lemma is well known, e.g.~see \cite[Theorem 2.6]{mashreghi}:

\begin{lemma}
\label{cluster-set}
Let $F$ be an inner function. A point $\zeta \in \partial \mathbb{D}$ belongs to $\Sigma$ if and only if the
cluster set  $\Cl(F, \zeta)$ is the closed unit disk.
\end{lemma}

We will also need the following elementary observation:

\begin{lemma}
\label{covering-lemma}
Suppose $g: \mathbb{D} \to \mathbb{D}$ is a inner function, which extends continuously to $\partial \mathbb{D} \setminus \{-1,1\}$.
If $g$ preserves the upper and lower semicircles, then $g$ is a M\"obius transformation.
\end{lemma}

\begin{proof}
We begin by observing that $g$ is injective on $\partial \mathbb{D} \cap \mathbb{H}$ and  $\partial \mathbb{D} \cap \overline{\mathbb{H}}$, being orientation-preserving but not surjective.
 Let $\alpha$ be a point on the unit circle, which is not the radial limit of $g$ at $-1$ or $+1$, if they exist. As the Aleksandrov-Clark measure $\mu_\alpha$ can be supported on at most one point, (\ref{eq:ac-def}) tells us that $g$ is either a constant or a M\"obius transfomation. The assumption rules out the former case, so that $g$ is a M\"obius transformation.
\end{proof}

\begin{proof}[Proof of Theorem \ref{characterizations-of-one-component-functions}]

We show (a) $\Rightarrow$ (b) $\Rightarrow$ (c) $\Rightarrow$ (a) and (c) $\Leftrightarrow$ (d):

\medskip

(a) $\Rightarrow$ (b). By Lemma \ref{preimage-components}, $F^{-1}(\mathbb{D}_r)$ is a Jordan domain. If $F^{-1}(\mathbb{D}_r)$ is compactly contained in the unit disk, then $F$ is a finite Blaschke product. Otherwise, $\mathbb{D} \setminus  F^{-1}(\mathbb{D}_r)$ consists of countably many simply-connected regions $\Omega_i$, bounded by a curve $\gamma_i$ in the unit disk and an arc $I_i$ on the unit circle. Usually $\gamma_i$ and $I_i$ share two distinct endpoints, however, when there is only one complementary region, then $\gamma_i$ and $I_i$ meet twice at a single point.

In view of Lemma \ref{cluster-set}, $F$ extends continuously to each open arc $I_i$.
Therefore, $F$ is continuous on $\overline{\Omega_i}$, with possible exception of one or two points where $\gamma_i$ meets the unit circle. 
To see that $F$ acts as a universal covering map
from $\Omega_i$ onto the annulus $A = A(0; r, 1)$, we lift $F$ to the universal cover of $A = A(0; r, 1)$, compose with conformal maps as necessary and apply Lemma \ref{covering-lemma}.

\medskip

(b) $\Rightarrow$ (c). Suppose $F: \mathbb{D} \to \mathbb{D}$ is a covering map over an annulus $A(0; \rho, 1)$. Let $\Omega_1, \Omega_2, \dots$ denote the connected components of $F^{-1}(A(0; \rho, 1))$. For each $i = 1, 2, \dots$, let $\widetilde{\Omega}_i$ be the double of $\Omega_i$ across the unit circle.
In view of Lemma \ref{cluster-set}, $F$ extends holomorphically across $I_i = \widetilde{\Omega}_i \cap \partial \mathbb{D}$. From the Schwarz reflection principle, it follows that $F$ is a covering map from $\widetilde{\Omega}_i$ to $A(0; \rho, 1/\rho)$.
 
\medskip

(c) $\Rightarrow$ (a). Take a slightly smaller annulus $A(0; r, 1/r) \subset A(0; \rho, 1/\rho)$. Each connected component of $F^{-1}(A(0; r, 1/r))$ is a Jordan domain $\widetilde{\Omega}_i$ which is symmetric with respect to the unit circle. The boundaries of the connected components can only meet at points of $\Sigma$. As such, $\mathbb{D} \setminus F^{-1}(A(0; r, 1/r))$ is connected so that $F$ is a one component inner function.

\medskip

(c) $\Rightarrow$ (d). Let $\zeta \in \partial \mathbb{D}$ be a point at which the radial limit of $F$ belongs to the unit circle. By applying Koebe's Quarter Theorem to the holomorphic branch of $F^{-1}$ on the ball $B \bigl (F(\zeta), 1-\rho \bigr )$ that maps $F(\zeta)$ to $\zeta$, we get
$$
\mathscr B \, = \, B \biggl (\zeta, \, \frac{1-\rho}{4 |F'(\zeta)|} \biggr ) \, \subset \,\sph \setminus \Sigma.
$$
In particular, $\Sigma$ has Lebesgue measure $0$ and $F'(\zeta) \to \infty$ as $\dist(\zeta, \Sigma) \to 0$. 

Recall that the class $\mathcal S$ of conformal maps $\varphi: \mathbb{D} \to \mathbb{C}$ with $\varphi(0) = 0$ and $\varphi'(0) = 1$ is compact in the topology of uniform convergence on compact subsets. In particular, there exists a constant $C > 0$ so that the second derivative $|\varphi''(0)| \le C$ for every $\varphi \in \mathcal S$. After rescaling appropriately, we may apply this bound to $F$ on $\mathscr B$ to get (\ref{eq:analytic-definition}). 

\medskip

(d) $\Rightarrow$ (c).
Let $\zeta \in \partial \mathbb{D} \setminus \Sigma$. Together with Gronwall's inequality, (\ref{eq:analytic-definition}) implies that there exists $c > 0$ so that 
$$|F'(\omega)| < 2 |F'(\zeta)| \ \text{on the arc } \ I = I \biggl (\zeta,\, \frac{c}{|F'(\zeta)|} \biggr )$$ of length $2c/|F'(\zeta)|$ centered at $\zeta$. The assumption that $F'(\omega) \to \infty$ as $\dist(\omega, \Sigma) \to 0$ assures us that $I \subset \partial \mathbb{D} \setminus \Sigma$. By shrinking $c > 0$ if necessary, we can guarantee that $F|_I$ omits an arc of length $\pi$.

In particular, $F: \mathbb{D} \cup I \cup (\sph \setminus \mathbb{D}) \to \sph$ is contained in a compact family
of meromorphic functions. An argument involving rescaling and normal families shows that
 $F$ is injective on the ball
$$
\mathscr B \, = \, B \biggl (\zeta, \,  \frac{c_2}{|F'(\zeta)|} \biggr ) \, \subset \,\sph \setminus \Sigma,
$$
for some $0 < c_2 < c$. Koebe's Quarter Theorem then shows that $F(\mathscr B)$ contains a ball centered at $F(\zeta)$ whose radius bounded is bounded below by $c_3 = c_2/4$. It remains to set $\rho = 1 - c_3$.
\end{proof}

\begin{remark}The proof of (c) $\Rightarrow$ (d) shows that for every integer $m \ge 2$ there exists a constant $C_m\in(0,+\infty)$ such that
\begin{equation}
\label{eq:analytic-definition-hd}
\biggl | \frac{F^{(m)}(\zeta)}{F'(\zeta)^m} \biggr | \le C_m, \qquad \zeta \in \partial \mathbb{D} \setminus \Sigma.
\end{equation}
For the original proof, see \cite[Corollary 1]{aleksandrov}. 
\end{remark}

We now describe two illuminating classes of examples of one component inner functions.

\paragraph*{Character-Automorphic Functions.} Let $\Gamma \subset \aut(\mathbb{D})$ be a Fuchsian group which possesses a fundamental polygon $P$ with finitely many sides, including at least one ideal side. In other words, $\Gamma$ is a geometrically finite Fuchsian group of co-infinite area. We can form an infinite Blaschke product whose zeros constitute an orbit of $\Gamma$:
$$
g(z) \, = \,  \prod_{\gamma \in \Gamma} -\frac{\gamma(0)}{|\gamma(0)|} \cdot \frac{z-\gamma(0)}{1-\overline{\gamma(0)}z}.
$$
The function $g$ is related to the Green's function of the Riemann surface $\mathbb{D}/\Gamma$.

Contrary to what one may initially expect, the function $g$ is not automorphic under $\Gamma$ but only character-automorphic. A {\em character} $v$ of a Fuchsian group $\Gamma \subset \aut(\mathbb{D})$ is a homomorphism from $\Gamma$ to the unit circle. A function $f$ on the unit disk is called {\em character automorphic} if 
$$
f(\gamma(z)) = v(\gamma) \cdot f(z), \qquad \gamma \in \Gamma.
$$

It is easy to see that if  $B(0, c)$ is a ball which meets all the finite sides of $P$, then the inverse image $F^{-1}(B(0,c))$ is connected. For further properties of the Blaschke product $g$, we refer the reader to \cite{pom2}.

\paragraph*{Stephenson's Cut and Paste Construction.} Another natural way of building inner functions is Stephenson's cut and paste construction, see \cite{stephenson, bishop}. We confine ourselves to a particular example but the reader can easily modify it to construct a plethora of other examples.

Take an infinite collection of tiles $T_j$ of the form $\mathbb{D} \setminus [1/2, 1)$, indexed by the integers. We form a simply-connected Riemann surface $S$ by gluing the lower side of $[1/2,1)$ in $T_{j}$ to the upper side of $[1/2,1)$ in $T_{j+1}$. The surface $S$ comes equipped with a natural projection to the disk $\mathbb{D}$ which sends a point in a tile $T_j$ to its representative in $\mathbb{D} \setminus [1/2,1)$. We may uniformize $S \cong \mathbb{D}$ by taking $0$ in the base tile $T_0$ to $0$. In this uniformizing coordinate, the projection $F:S\to \mathbb{D} \setminus \{1/2\}$ becomes a holomorphic self-map  of the disk. Since all the slits have been glued up, $F$ is an inner function, and a little thought shows that it is the universal covering map of $\mathbb{D} \setminus \{1/2\}$.

It is clear from this construction, that if $B(0,c)$ is a ball with radius $c\in (1/2,1)$, then the inverse image $F^{-1}(B(0,c))$ is connected.
 
\section{Thermodynamic Formalism for Centered One Component Inner Functions}
\label{sec:one-component}

In this section, we discuss centered one component inner functions from a dynamical point of view. We show that a one component inner function admits a Markov partition, which allows us to code the unit circle by a countable alphabet shift space. For one component inner functions of finite entropy, we prove a variational principle and show that the potential $\psi(z) = - \log|F'|$ has a mixing property needed for Orbit Counting.

\subsection{Markov Partitions}
\label{sec:markov-partitions}

Let $X \subset \mathbb{C}$ be a closed set, $\mu$ be a measure on $X$ and $F: X \to X$ be a conformal dynamical system defined $\mu$ almost everywhere. A {\em Markov partition} for $F$ is a countable collection of closed subsets $\{ X_i \}_{i \in I}$ called {\em tiles} which satisfy the following axioms:
\begin{enumerate}
\item The tiles $X_i$ cover $X$ up to a set of $\mu$ measure zero.
\item As subspaces of $X$, each tile $X_i$ is the closure of its interior.
\item The interiors of the tiles are disjoint.
\item $F|_{\Interior X_i}$ is injective and extends to a conformal map on a neighbourhood of $X_i$.
\item Up to a set of measure zero, $F(X_i)$ is a union of tiles.
\end{enumerate}
When working with Graph Directed Markov systems (GDMS), one additionally requires that each inverse branch $F^{-1}: F(X_i) \to X_i$ extends to a conformal map defined on a neighbourhood of $F(X_i)$.

We first construct a {\em classical Markov partition} for a centered one component inner function (with respect to the Lebesgue measure) and then modify it slightly in order to represent the unit circle as the limit set of a Graph Directed Markov System (GDMS), the concept defined and explored in \cite{MU}. For more recent expositions on GDMS, we refer the reader to \cite{URM22,KU1}.

Below, we will use the fact that a centered inner function is uniformly expanding on the unit circle:
\begin{equation}
\label{eq:uniformly-expanding}
\inf_{\zeta \in \partial \mathbb{D}} |F'(\zeta)|> 1,
\end{equation}
where we use the convention that $|F'(\zeta)| = \infty$ if $F$ does not have an angular derivative at $\zeta$ in the sense of Carath\'eodory. For a proof, see \cite[Theorem 4.15]{mashreghi}. For one component inner functions, this convention amounts to setting $|F'(\zeta)| = \infty$ on $\Sigma \subset \partial \mathbb{D}$.

\paragraph*{Finite Blaschke Products.}
Suppose first that $F$ is a finite Blaschke product of degree $d \ge 2$, for example $z \mapsto z^d$. Let $p$ be one of the $d-1$ repelling fixed points of $F$ on the unit circle.
The set $F^{-1}(p)$ divides the unit circle into $d$ open arcs $I_1, I_2, \dots, I_d$, which are mapped by $F$ bijectively onto $\partial \mathbb{D} \setminus \{ p \}$.

This Markov partition allows us to code points on the unit circle by infinite words in $\{1, 2, \dots, d\}^{\mathbb{N}}$.
More precisely, we say that a point $\zeta \in \partial \mathbb{D}$ is coded by 
$$
e_0 e_1 e_2 \dots
$$
 if $e_n = i$ whenever $F^{\circ n}(\zeta) \in \overline{I_i}$. It is easy to see that any infinite word $\{1, 2, \dots, d\}^{\mathbb{N}}$ arises as the code of some point on the unit circle, and all but countably many
 points are coded by unique words. The {\em exceptional set} $\mathscr E$ consists of the endpoints of the arcs $I_i$ and their iterated pre-images under $F$. Points in $\mathscr E$ have two codings, while all other points have exactly one code.

\paragraph*{Infinite-Degree Inner Functions.}
We now assume that $F$ is a centered one component inner function that is not a finite Blaschke product. 
In this case, the singular set $\Sigma \subset \partial \mathbb{D}$ of $F$ is not empty.

To construct a Markov partition of the unit circle for $F$, fix an arbitrary point $p \in \partial \mathbb{D}$ on the unit circle. Let $\{ J_k \}_{k=1}^\infty$ be the collection complementary arcs in $\partial \mathbb{D} \setminus \Sigma$. Since for each $k\in\mathbb N$, the restriction $F|_{J_k}$ is an infinite degree covering map of $\partial \mathbb{D}$, the set $F^{-1}(p)$ partitions each $J_k$ into infinitely many arcs: $J_k = \bigcup_{l=-\infty}^\infty \overline{I_{k,l}}$. For convenience, we enumerate the arcs $\{I_i\} = \{ I_{k,l} \} $ by a single index $i\in\mathbb N$.

As $F(I_i) = \partial \mathbb{D} \setminus \{ p \}$,  each closed arc $\overline{I_i}$ contains a fixed point of $F$. In particular, $F$ has infinitely many fixed points. We may therefore assume that $p$ is one of these fixed points.

We again have a near-bijection between infinite words in $\{1, 2, \dots \}^{\mathbb{N}}$ and points in $\partial \mathbb{D}$. This time, the exceptional set $\mathscr E$ of points without unique codings consists of iterated pre-images of endpoints of the arcs $I_i$, $i\in\mathbb N$, as well as points of $\Sigma$. The former have two codings, while the latter have no codings. Since the singular set $\Sigma$ has Lebesgue measure zero, so does $\mathscr E$.

\paragraph*{Graph Directed Markov Systems.}

From the perspective of GDMS, the Markov partitions constructed above are slightly inadequate. Arbitrarily choose a point $q \in \Sigma$. The points $p$ and $q$ divide the unit circle into two arcs $X_1$ and $X_2$, which we call {\em tiles} of the GDMS. For each tile $X_j$, $j=1,2$, choose an open set $X_j \subset U_j \subset A(0; r, 1/r)$, where $A(0; r, 1/r)$ is an annulus over which $F$ is a covering map.

By construction, each arc $I_i$ in the classical Markov partition is wholly contained in one of the two tiles. Recall that
$F(I_i) = \partial \mathbb{D} \setminus \{ p \}$.
We can 
 split $I_i$ into two smaller arcs $I_{i,1}$ and $I_{i,2}$ with $F(I_{i,1}) = X_1$ and $F(I_{i,2}) = X_2$. We refer
 to the collection of intervals $\{ I_n \} = \{I_{i, j}\}$ as the {\em Markov partition for the GDMS}.
 
 We write $\phi_{i,j}: U_j \to \mathbb{C}$ for the holomorphic branch of $F^{-1}$ which maps $X_j$ onto $I_{i,j}$. We define the {\em alphabet} as the set of these  {\em contractions} $\{ \phi_{n} \} = \{ \phi_{i,j} \}$.
Two contractions $\phi_m$ and $\phi_n$ can only be composed if the image of $\phi_n$ is contained in the tile on which $\phi_m$ is defined. We say that a word (finite or infinite)
$$
e_0 e_1 e_2 \dots
$$
is {\em admissible} if for any $n \ge 0$, the composition $\phi_{e_n} \circ \phi_{e_{n+1}}$ makes sense. It is easy to see that the incidence matrix which indicates admissible compositions is finitely primitive (see Section \ref{sec:shift-shape} for the definition).

The normalization $F(0) = 0$ guarantees that the inverse branches $\phi_{e_n}$ contract by a definite factor. In particular, admissible infinite words code single points. The set of admissible infinite words is almost in bijection with points on the unit circle: the only obstruction being that points in $\mathscr E$ do not have a unique coding. 

\medskip

From the topological description of a one component inner function (Lemma \ref{one-component-anatomy}), it follows that if $F$ has a radial boundary value at a point $\zeta \in \Sigma$, then it cannot lie on the unit circle.

\begin{lemma}
Let $F$ be a one-component centered inner function.
For any $\alpha \in \partial \mathbb{D}$, the Aleksandrov-Clark measure $\mu_\alpha$ does not charge the set $\Sigma$. In particular, each measure $\mu_\alpha$ is discrete and the adjoint definition of the transfer operator coincides with the classical one.
\end{lemma}

As a result, little is lost when passing to the subset of points on the unit circle which have unique codes.

\subsection{Some Properties of the Potential \texorpdfstring{$-\log |F'|$}{-log |F'|}}
\label{sec:some-properties-of-the-geometric-potential}

We now record some properties of the potential 
$$
\psi(z) = -\log|F'(z)|
$$ 
associated to a centered one component inner function $F$ which will be used in the sequel. In the language of Part \ref{shift-space}, one would say that $\psi$ is a normal potential.

\begin{enumerate}
\item (Summable) The sum over the intervals $I_n$ in the Markov partition (for the GDMS),
$$
\sum_{n=1}^\infty \sup_{I_n} |e^\psi|   < \infty.
$$
Indeed, by Koebe's Distortion Theorem, the above sum is essentially
$$
\sum_{n=1}^\infty \sup_{I_n} |e^\psi| \, \asymp \, \sum_{n=1}^\infty | I_n| \, = \, 2\pi.
$$
More generally, for any integer $p \ge 0$, we say that a potential $\psi$ is {\em $(1,p)$-summable} if
$$
\sum_{n=1}^\infty  \sup_{I_n} |\psi^p e^\psi|  < \infty.
$$
By Koebe's Distortion Theorem, the above  condition is equivalent to
$$
\int_{\partial \mathbb{D}} |\psi|^p dm < \infty.
$$

\item (Negative) Since $F(0) = 0$, the derivative $|F'(z)| > c > 1$ on the unit circle $\partial \mathbb{D}$ by (\ref{eq:uniformly-expanding}). As a result,
$$\sup_{z \in \partial \mathbb{D}} \psi(z) < 0.$$

\item (Centered) In Section \ref{sec:composition-operators}, we saw that the spectral radius of the composition operator $C_F$ was 1. 
As $\mathcal L_{-\log|F'|}$ is its adjoint, we have 
$$P(- \log |F'|) = \log 1 = 0.$$

\item (Level $1$ Lipschitz continuity)

Unless $F$ is a finite Blaschke product, the potential $\psi(z) = -\log |F'(z)|$ is not Lipschitz on the unit circle, let alone bounded. Nevertheless, there is a constant $C\in [1,+\infty)$ such on each interval $I_n$ from the Markov partition (for the GDMS), one has
\begin{equation}
\label{eq:level1-lipschitz}
| \psi(x) - \psi(y) | \le C \cdot |F(x) - F(y)|,
\end{equation}
which we call {\em level 1 Lipschitz continuity}\/. In other words, $\psi$ is Lipschitz on each interval $I_n$ from the Markov partition when it is rescaled to definite size, with a uniform Lipschitz constant.

To verify (\ref{eq:level1-lipschitz}), we use the fundamental theorem of calculus:
$$| \psi(x) - \psi(y) | \le \int_x^y \biggl | \frac{F''(z)}{F'(z)} \biggr | \, |dz|, \qquad x, y \in I_n.$$
From Koebe's Distortion Theorem, we know that
$$
|F'(z)| \asymp \frac{|F(x) - F(y)|}{|x-y|}, \qquad x, y, z \in I_n.
$$
Therefore,
$$| \psi(x) - \psi(y) | \le \frac{|F(x) - F(y)|}{|x-y|} \int_x^y \biggl | \frac{F''(z)}{F'(z)^2} \biggr | \, |dz|.$$
Finally, by the analytic definition of one component inner functions (\ref{eq:analytic-definition}), the integrand is bounded, which shows (\ref{eq:level1-lipschitz}).
\end{enumerate}

\subsection{Variational Principle}

To proceed, we impose an additional assumption on our inner function $F$: we assume that its derivative lies in the Nevanlinna class:
$$
\int_{\partial \mathbb{D}} \log|F'| dm < \infty.
$$
In a beautiful work, M.~Craizer \cite{craizer} showed that this condition is equivalent to the Lebesgue measure $m$ having finite measure-theoretic entropy with respect to the action of $F$ on the unit circle.

Let $\mathcal P = \{I_n\}_{n\in\N}$ be the classical Markov partition of the unit circle defined in Section \ref{sec:markov-partitions}.
Since each arc $I_n \in \mathcal P$ is mapped onto $\partial \mathbb{D} \setminus \{ p\}$, by Koebe's Distortion Theorem, $|F'(\zeta)| \asymp 1/|I_n|$ for $\zeta \in I_n$. Therefore,
$$
\sum |I_n| \log \frac{1}{|I_n|} \asymp \int_{\partial \mathbb{D}} \log |F'| dm < \infty.
$$

Notice that each element of the refined partition $$\mathcal P^{(n)} = \mathcal P \vee \mathcal F^{-1}(\mathcal P) \vee \mathcal F^{-2}(\mathcal P) \vee \dots \vee \mathcal F^{-(n-1)}(\mathcal P)$$ is an arc and that the maximal size of arcs decays exponentially. As such, $\mathcal P$ is a one-sided generator in the sense that $\bigvee \mathcal P^{(n)}$ generates the $\sigma$-algebra of Borel subsets of the circle. 
Since $\mathcal P$ is a generating partition of finite entropy, by \cite[Theorem 1.9.7]{conformal-fractals}, the measure theoretic entropy is given by
$$
h(F, m) = \int_{\partial \mathbb{D}} \log J_F \, dm,
$$
where $J_F$ is the measure-theoretic Jacobian of $F$. By definition, the measure-theoretic Jacobian is a non-negative function in $L^1(\partial \mathbb D, m)$ such that
$$m(F(E)) = \int_{E} J_F \, dm$$ for any measurable set $E$ on which $F$ is injective. In the setting of one component inner functions, $J_F(\zeta) = |F'(\zeta)|$ a.e.~on the unit circle.

\begin{theorem} Suppose that $F$ is a centered one component inner function whose derivative lies in the Nevanlinna class.
If $\mu$ is an $F$-invariant probability measure on the unit circle, then its measure-theoretic entropy
$$
h(F, \mu) \le \int_{\partial \mathbb{D}} \log |F'(\zeta)| d\mu.
$$
Equality holds only for the normalized Lebesgue measure $m$.
\end{theorem}

The inequality in the variational principle above is neatly explained in \cite[Section 2.1]{MU}.
To see that the Lebesgue measure is the only measure for which equality holds, we refer the reader to Proof II of \cite[Theorem 4.6.2]{conformal-fractals}.

\subsection{Spectral Gap on Smooth Functions}
\label{sec:smooth-spectral-gap}

Below, we continue to abbreviate $\psi(x) = - \log|F'(x)|$ for simplicity.
From the estimate (\ref{eq:analytic-definition-hd}), it is not difficult to see that for every integer $m \ge 1$ there exists a constant ${\tilde C}_m\in(0,+\infty)$ such that 
\begin{equation}
\label{eq:Am-bound2}
\biggl | \frac{ \psi^{(m)}(x) }{F'(x)^{m}} \biggr | \le {\tilde C}_m, \qquad x \in \partial \mathbb{D} \setminus \Sigma,
\end{equation}
where $\psi^{(m)}$ denotes the $m$-th derivative of $\psi$ along the unit circle.

\begin{lemma}
\label{smooth-spectral-gap}
Suppose that $F$ is a centered one component inner function, which is not a rotation. If $\re s \ge 1$, then
$$
\mathcal L_{s} g(x) = \sum_{F(y) = x} e^{s \psi(y)} g(y)
$$ 
satisfies the two-norm inequality of Ionescu-Tulcea and Marinescu:
There exist constants $\theta\in(0,1)$ and $C\in(0,+\infty)$, which may depend on $k$ and $s$ such that
\begin{equation}
 \label{eq:itm1}
\|  \mathcal L_{s} g(x) \|_{C^k} \le  \theta \cdot \| g \|_{C^k} + C \cdot \| g \|_{C^{k-1}}.
\end{equation}
In particular, the norms
\begin{equation}
\label{eq:itm2}
\|  \mathcal L^n_{s}  g(x) \|_{C^k} \le  \theta^n \cdot \| g \|_{C^k} + \frac{C}{1-\theta} \cdot \| g \|_{C^{k-1}}
\end{equation}
are uniformly bounded for $n \ge 1$.
\end{lemma}

\begin{proof}
We treat the case when $k = 1$ as the general case is similar.  Differentiating the definition of
$\mathcal L_{s} g$, we get
$$
\frac{d}{dx} \mathcal L_{s} g(x) = \sum_{F(y) = x} \frac{1}{|F'(y)|} \biggl \{ e^{s \psi(y)} g'(y) + s e^{s \psi(y)} \psi'(y) \cdot g(y) \biggr \}.
$$
Thus,
$$
\biggl | \frac{d}{dx}  \mathcal L_{s} g(x) \biggr | \le  \sum_{F(y) = x} \biggl \{ |F'(y)|^{-1-\re s} \, \| g \|_{C^1} + |s| \cdot |F'(y)|^{-\re s} \, \biggl | \frac{F''(y)}{F'(y)^2} \biggr| \, \| g \|_{C^0} \biggr \},
$$
which simplifies to (\ref{eq:itm1}). In the last step, we have used that
$$
\sum_{F(y) = x} |F'(y)|^{- \re s} \le 1 \quad \text{and} \quad \sum_{F(y) = x} |F'(y)|^{-1 - \re s} \le \theta < 1,
$$
which can be proved in the same way as Corollary \ref{spectral-gap-bergman-space}.
\end{proof}

As in \cite[Proposition 2.1 and Theorem 2.2]{PP}, the above lemma implies that $\mathcal L_1$ 
has a spectral gap on $C^k(\partial \mathbb{D})$ for every integer $k \ge 1$:

\begin{corollary}
\label{smooth-spectral-gap2}
If  $F$ is a centered one component inner function, then the operator $\mathcal L_1$ has a simple eigenvalue 1 with eigenfunction ${\bf 1}$, while the rest of the spectrum is contained in a smaller ball centered at the origin. Moreover, the spectral radius $r_{C^k}(\mathcal L_{1+ia}) \le 1$ for any $a \in \mathbb{R}$.
\end{corollary}

\subsection{Genericity and Orbit Counting}
\label{sec:weak-mixing}

Suppose $F$ is a centered one component inner function with derivative in the Nevanlinna class. Taking $s = 1+ia$ in Lemma \ref{smooth-spectral-gap}, we see that for any function $u \in C^k$, the sequence 
$\bigl \{ \mathcal L_{(-1+ia)\log|F'|}^n (u) \bigr \}_{n=0}^\infty$ is bounded in $C^k$, which is an analogue of Lemma \ref{boundedness-of-the-iterates-of-L}.
Adapting the proof of Lemma \ref{smoothness-of-eigenfunctions4} to our current setting, we obtain:

\begin{lemma}
\label{smoothness-of-eigenfunctions2}
Any $L^2$ eigenfunction of $e^{-i a \psi} C_F$ is $C^\infty$.
\end{lemma} 

We now show:

\begin{lemma}
\label{non-existence-of-eigenfunctions2}
Suppose $F$ is an infinite degree centered one component  inner function of finite entropy. For any $a \ne 0$, the operator $e^{-i a \psi} C_F$ does not possess any continuous eigenfunctions, let alone any $C^\infty$ eigenfunctions.
\end{lemma}

\begin{proof}
 
We want to show that any continuous solution $w \ne 0$ of the eigenfunction equation
\begin{equation}
\label{smooth-eigenfunction-equation2}
w(F(z)) = e^{i a \log |F'(z)|} \cdot w(z), \qquad z \in \mathbb{D},
\end{equation}
is constant, which would imply that $a = 0$. Since $w$ is not identically zero, we may scale $w$ so that $|w(z)| = 1$. Let $d \in \mathbb{Z}$ be the topological degree of $w: \partial \mathbb{D} \to \partial \mathbb{D}$. As $F$ is an infinite degree one component inner function, the singular set $\Sigma \ne \emptyset$. Pick an arbitrary complementary arc 
$$J \, = \, \bigcup_{k \in \mathbb{Z}} I_k \subset \partial \mathbb{D} \, \setminus \, \Sigma,
$$ 
where the sub-arcs $I_k$ are elements of the classical Markov partition defined in Section \ref{sec:markov-partitions}.
Since $J$ is simply-connected, $\arg w(x)$ and $\arg w(F(x))$ admit continuous branches on $J$. We consider two cases:
 
 \medskip
 
{\em Case I. $d = 0$.} Let $I_k = [z_k, z_{k+1}]$. As $z$ moves from $z_k$ to $z_{k+1}$, the net change of the argument of $w(F(z))$ is zero. Consequently, $\arg w(F(z))$ is bounded on $J$. As $\arg w(z)$ is bounded on $J$, equation (\ref{smooth-eigenfunction-equation2}) implies that $a \log |F'(z)|$ must also be bounded on $J$. This contradicts Theorem \ref{characterizations-of-one-component-functions}(d), which says that $|F'(z)| \to \infty$ as $z \to \Sigma$.
 
  \medskip

{\em Case II. $d \ne 0$.} Below, we assume that $d > 0$, as the case $d < 0$ is similar. As $z$ moves from $z_k$ to $z_{k+1}$, the argument of $w(F(z))$ increases by $d$, so that
$$
\arg w(F(z_k)) - \arg w(F(z_0)) = kd + O(1).
$$
In particular, this implies that $\log |F'(z_k)| \to -\infty$ as $k \to -\infty$. This once again contradicts Theorem \ref{characterizations-of-one-component-functions}(d).
\end{proof}

In the language of Part \ref{shift-space}, Lemma \ref{non-existence-of-eigenfunctions2} says that the potential $\psi(z) = - \log|F'(z)|$ is D-generic.
Together with the discussion in Section \ref{sec:some-properties-of-the-geometric-potential}, this completes the verification that $\psi$ is a robust potential when $F$ is an infinite degree one component inner function of finite entropy. For the case of finite Blaschke products, we refer the reader to \cite[Section 7]{ivrii-wp}.

To deduce Theorems \ref{orbit-counting1a} and \ref{orbit-counting2a} as special cases of  Theorems \ref{orbit-counting-tdf} and \ref{orbit-counting-tdf2},
we code the unit circle as in Section \ref{sec:markov-partitions}.
In this case, the conformal and equilibrium measures for the potential $\psi \circ \pi: E_A^\infty \to \mathbb{R}$ are equal to the pullback of the normalized Lebesgue measure on the unit circle under the coding map $\pi : E^\infty_A \to \partial \mathbb{D}$. Inserting these quantities into Theorems \ref{orbit-counting-tdf} and  \ref{orbit-counting-tdf2}, one obtains Theorems \ref{orbit-counting1a} and \ref{orbit-counting2a} for points on the unit circle with unique codings, which are dense in the unit circle.
To obtain the full statements of Theorems \ref{orbit-counting1a} and \ref{orbit-counting2a}, one can approximate a general point $x \in \partial \mathbb{D}$ by a sequence of points  $x_n \in \partial \mathbb{D}$ in the image of $\pi$ and apply Koebe's Distortion Theorem. We leave the details to the reader.

\section{Thermodynamic Formalism for Parabolic One Component Inner Functions}
\label{POCIF}

An inner function $f: \mathbb{D} \to \mathbb{D}$ is {\em parabolic} if its Denjoy-Wolff fixed point $p \in \partial \mathbb{D}$ and
$f'(p) := \lim_{r \to 1} f'(rp) = 1$. 
It is convenient to think of parabolic inner functions as holomorphic self-maps of the upper half-plane with a Denjoy-Wolff point at infinity.
For this purpose, take a M\"obius transformation $M$ which maps $\mathbb{D}$ to $\mathbb{H}$ and sends $p$ to $\infty$. Then,
$F = M \circ f \circ M^{-1}$ is a holomorphic self-map of the upper half-plane such that for Lebesgue almost every point $x \in \mathbb{R}$, the vertical boundary value
$\lim_{y \to 0} F(x+iy)$ exists and is real.

We say that an inner function $f$  is {\em doubly parabolic} if $f$ is holomorphic in a neighbourhood of $p$ and 
$$
f(z) = z + c (z-p)^3 + \dots, \qquad c \ne 0,
$$ 
as $z \to p$.
Translating to the upper half plane, this means that $F$ is doubly parabolic if
$$
F(z) = z - a/z + \dots, \qquad a > 0,
$$ 
near infinity.
By contrast, an expansion
$$
f(z) = z + c(z-p)^2 + \dots, \qquad c \ne 0,$$
 or $$F(z) = z + T - a/z + \dots, \qquad T \ne 0, \qquad a > 0,$$ 
suggests that $f$ and $F$ are {\em simply parabolic}\/.

We say that $F$ is a {\em parabolic one component inner function} if one of the following equivalent conditions is satisfied:

\begin{theorem}
\label{characterizations-of-one-component-functions2}
Suppose $F$ is a  parabolic inner function, viewed as a map of the upper half-plane to itself, with a parabolic fixed point at infinity. 
The following are equivalent:

{\em (a)} The set $F^{-1}(\{\im z > y \})$ is connected for some $y > 0$.

{\em (b)} There is a strip $S = \{ 0 < \im z < \tau\}$ such that $F: \mathbb{H} \to \mathbb{H}$ is a covering map over $S$.

{\em (c)} There is a strip $\tilde{S} =  \{ -\tau < \im z < \tau\}$ such that $F: \mathbb{C} \setminus \Sigma_\infty \to \mathbb{C}$ is a covering map over $\tilde S$.

{\em (d)} The singular set $\Sigma_\infty \subset \mathbb{R}$ has Lebesgue measure $0$, the derivative $F'(\zeta) \to \infty$ as $\zeta \in \mathbb{R}$ approaches $\Sigma$ and
\begin{equation}
\label{eq:analytic-definition2}
\biggl | \frac{F''(\zeta)}{F'(\zeta)^2} \biggr | \le C, \qquad \zeta \in \mathbb{R} \setminus \Sigma_\infty,
\end{equation}
for some constant $C > 0$.
\end{theorem}

In the theorem above, $\Sigma_\infty \subset \mathbb{R}$ is the set of points $x \in \mathbb{R}$ where $F$ does not admit a finite analytic extension to a neighbourhood of $x$. By definition, $\Sigma_\infty$ is a closed subset of the real line.
The proof of Theorem \ref{characterizations-of-one-component-functions2} is essentially the same as that of Theorem \ref{characterizations-of-one-component-functions}. We leave the details to the reader.

In this section, we prove an orbit counting theorem for doubly parabolic one component inner functions (dp1c inner functions).  In order to apply the machinery developed in Section \ref{sec:orbit-counting}, we work with the first return map.

\subsection{Expansion}

The lemma below says that parabolic inner functions are expanding on the real line:

\begin{lemma}
Suppose $F$ is a  parabolic inner function, viewed as a map of the upper half-plane to itself, with a parabolic fixed point at infinity. Then, $F'(x) > 1$ on $\mathbb{R}$.
\end{lemma}

\begin{proof}
An inner function, viewed as self-map of the upper half-plane, has the form
$$
F(z) = \alpha z + \beta + \int_{\mathbb{R}} \frac{1 + zw}{w-z} d\mu(w),
$$
for some constants $\alpha > 0$, $\beta \in \mathbb{R}$ and finite positive singular measure $\mu$ on the real line, e.g.~see \cite{tsuji}.
Differentiating, we get
\begin{equation}
\label{eq:tsuji-derivative}
F'(z) = \alpha + \int_{\mathbb{R}} \frac{w^2+1}{(w-z)^2} d\mu(w).
\end{equation}
From the above representation, it is readily seen that $\alpha = \lim_{t \to \infty} F'(it)$. For $F$ to have a parabolic fixed point at infinity, we must have $\alpha = 1$.
Since the integrand in (\ref{eq:tsuji-derivative}) is positive, $F'(z) > 1$ for any $z \in \mathbb{R}$.
\end{proof}

\subsection{Markov Partitions}

Let $F$ be a dp1c inner function.
The set $F^{-1}(\infty)$ partitions $(\mathbb{R} \cup \{ \infty \}) \setminus \Sigma$ into countably many intervals $\{I_i\}_{i \in \Lambda}$ with $F(I_i) = \mathbb{R}$. The collection of intervals $\{I_i\}_{i \in \Lambda}$ is finite if and only if $F$ is a finite Blaschke product.
We refer to the collection $\{I_i\}_{i \in \Lambda}$ as the {\em basic Markov partition} for $F: \mathbb{R} \cup  \{ \infty \} \to \mathbb{R} \cup  \{ \infty \}$.

Suppose $I_- = (-\infty, p_1^-)$ and $I_+ = (p_1^+, \infty)$ are the unbounded intervals in $\Lambda$. As the dynamics of $F$ is repelling from $\infty$ in both directions, one can construct inverse orbits
$$
 \dots < p_3^- < p_2^- < p_1^-
 \qquad \text{and} \qquad   p_1^+ < p_2^+ < p_3^+ < \dots
$$
We refine the Markov partition $\Lambda$ by subdividing $I_+$ and $I_-$ into 
countably many intervals:
$$
I_+ = \bigcup_{n = 1}^\infty J^+_n, \qquad I_{-} = \bigcup_{n = 1}^\infty J^-_{n},
$$
where $J_n^{\pm}$, $n\in\N$, is the segment between $p_n^{\pm}$ and $p_{n+1}^{\pm}$.
By construction, $F$ maps 
\begin{itemize}
\item $J_{n+1}^{\pm}$ onto $J_n^{\pm}$ for any $n \ge 1$,
\item $J_1^+$ onto $(-\infty, p_{1})$,
\item $J_1^-$ onto $(p_{-1}, \infty)$,
\item each bounded basic interval $I_i \subset (p_1^-, p_1^+)$ onto $\mathbb{R}$.
\end{itemize}
Since $\infty$ is a doubly parabolic fixed point, $\diam J^{\pm}_n \asymp 1/\sqrt{n}$ and is located $\asymp \sqrt{n}$ away from the origin.

For $N \ge 0$, define
$
X_N = [p_{N+1}^-, p_{N+1}^+].
$
Given a bounded subset $B$ of the real line, we choose $N \ge 1$ sufficiently large so that $B \subset X_N$. From now on, we simply write $X$ instead of $X_N$.

\subsection{The First Return Map}

We define the {\em first return map} $\hat F: X \to X$  as the first iterate of $F$ which lands in $X$, i.e.~
$$
\hat F(x):= F^{\circ N(x)}(x),
$$ 
where $N(x) \ge 1$ is the smallest positive integer such that $F^{\circ N(x)}(x) \in X$. (Since every point in $\mathbb{R} \setminus X$ eventually maps to $J_1^{\pm}$ under the dynamics of $F$, the first return map is well-defined outside the measure zero set $\Sigma_\infty \subset X$. In fact, any doubly parabolic Blaschke product is recurrent by \cite[Theorem 4.2]{doering-mane}.)

We may view $\hat{F} : X \to X$ as a GDMS with the tiles $J_1^{\pm}, J_2^{\pm}, \dots J_N^{\pm}$ and $[p^-_1, p^+_1]$ by splitting $X$ into intervals that map univalently onto one of the tiles under $\hat{F}$. We leave it to the reader to check that the associated incidence matrix is finitely primitive.

Let $\ell_X$ be the normalized Lebesgue measure on $X$. As $\ell$ is invariant for $F: \mathbb{R} \cup \{\infty\} \to \mathbb{R} \cup \{\infty\}$, $\ell_X$ is invariant for $\hat F: X \to X$. We now prove two lemmas which relate the properties of $F$ and $\hat{F}$\,:

\begin{lemma}
\label{lyapunov1}
The Lyapunov exponents of $\hat{F}: X \to X$ and $F: \mathbb{R} \to \mathbb{R}$ are equal:
$$
\int_{X} \log |{\hat F}'(x)| d\ell = \int_{\mathbb{R}} \log |F'(x)| d\ell.
$$
\end{lemma}

The above lemma follows from a generalization of Kac's Lemma \cite[Proposition 10.2.5]{URM22a} and the identity $\log|(F^{\circ n})'(x)| = S_n \log |F'(x)|$.

\begin{lemma}
\label{lyapunov2}
For any $p \ge 1$, $\hat F: X \to X$ is $(1,\,p)$-integrable if and only if $F: \mathbb{R} \to \mathbb{R}$ is $(1,\,p)$-integrable:
$$
\int_{X} \bigl ( \log |{\hat F}'(x)| \bigr)^p d\ell < \infty \quad \Longleftrightarrow \quad \int_{\mathbb{R}} \bigl ( \log |F'(x)|  \bigr )^p d\ell < \infty.
$$
\end{lemma}

\begin{proof}
For $x \in \mathbb{R} \setminus X$, we continue to use the notation $N(x)$ for the smallest positive integer such that 
$\hat F(x):= F^{\circ N(x)}(x) \in X$. 

Since $\diam J_n^\pm \asymp n^{-1/2}$, the derivative $|(F^{\circ N(x)})'(x)| \asymp n^{1/2}$ for $x \in J_n^{\pm}$, whence
$$
 \int_{J_n^- \cup J_n^+} \bigl ( \log |(F^{\circ n})'(x)|  \bigr )^p d\ell(x) \, \lesssim \, n^{-1/2} \cdot (1 + \log n)^p.
 $$
Summing over $n = 1, 2, \dots$, we obtain
$$
\int_{\mathbb{R} \setminus X} \bigl ( \log |(F^{\circ N(x)})'(x)| \bigr)^p d\ell(x) \, \lesssim \, \sum_{n=1}^\infty n^{-1/2} \cdot (1 + \log n)^p \, < \, +\infty.
$$
From the above equation, the $(\Rightarrow)$ implication is immediate since 
$$
\log |{\hat F}'(x)| \ge  \log |F'(x)|
$$
for every $x \in \mathbb{R}$.
By the elementary inequality $(A+B)^p \le 2^p(A^p+B^p)$ and the invariance of the Lebesgue measure, we get
\begin{equation*}
\int_{X} \bigl ( \log |{\hat F}'(x)| \bigr)^p d\ell  \lesssim
 \int_{X} \bigl ( \log |F'(x)|  \bigr )^p d\ell +  \int_{\mathbb{R} \setminus X} \bigl ( \log |(F^{\circ N(x)})'(x)|  \bigr )^p d\ell,
 \end{equation*}
which proves the  $(\Leftarrow)$ implication.
\end{proof}

\subsection{Proof of the Orbit Counting Theorem~\ref{parabolic-orbit-counting}}

In order to apply the abstract symbolic results developed in Part \ref{shift-space} of this manuscript to the first return map $\hat F: X \to X$, we show that the potential
$$
\psi(x):= - \log |\hat F'(x)|
$$
is  robust under the integrability hypothesis 
\begin{equation}
\label{eq:integrability-hypothesis}
\int_{ \mathbb{R}} \bigl ( \log |F'(z)| \bigl )^{1+\varepsilon} d\ell < \infty, \qquad \varepsilon > 0.
\end{equation}
In view of Lemma \ref{lyapunov2}, the assumption (\ref{eq:integrability-hypothesis}) is equivalent to the potential $\psi$ being $(1,1+\varepsilon)$-summable. 
From formula (\ref{eq:tsuji-derivative}), it is clear that $\psi$ is negative as $X \subset \mathbb{R}$ is compact. An argument similar to the one in Section \ref{sec:some-properties-of-the-geometric-potential} shows that $\psi$ is level 1 Lipschitz continuous (and hence  level 1 $\alpha$-H\"older continuous for any $0 < \alpha < 1$).

Since $\psi$ is level 1 H\"older continuous and summable, the conformal and equilibrium measures for $\psi$ are unique by Theorem \ref{rpf-symbolic-dynamics}. As $\ell_X$ is a conformal measure for $\psi$ (with eigenvalue 1), $m_\psi = \ell_X$.
Since the equilibrium measure $\mu_\psi$ is the unique $\hat{F}$-invariant measure on $X$ that is absolutely continuous with respect
to $m_\psi$, we also have $\mu_\psi = \ell_X$. In particular, the eigenfunction $\rho_\psi = \mu_\psi/m_\psi = {\bf 1}$ and the potential $\psi$ is centered. Summarizing, we have proved that $\psi$ is a normal $(1, 1+\varepsilon)$-summable potential.

\paragraph*{D-Genericity.}

According to \cite[Theorem 9.7]{PU}, parabolic dynamical systems are automatically D-generic. For convenience of the reader, we give a brief sketch of the argument in the present setting: for $n \ge 2$,
let $q_n \in J_{n-1}^+$ be the unique periodic point of period $n$ with itinerary
$$
J_{n-1}^+ \, \to \, \dots \to J_1^{+} \, \to \, J_1^{-} \, \to \, J_n^+ \, \to \, J_{n-1}^+.
$$
Inspection shows that the logarithms of the multipliers $$L_n = \log |(F^{\circ n})'(q_n)|$$ satisfy
$$
L_{n+1} - L_n \to 0, \qquad L_n \to \infty, \qquad \text{as } n \to \infty,$$ so they cannot be contained in a discrete subgroup of $\mathbb{R}$. It remains to notice that any periodic orbit of $F: \mathbb{R} \to \mathbb{R}$ which passes through $X$ can be viewed as a periodic orbit of the induced map $\hat{F} : X \to X$ with the same multiplier.

\paragraph*{Conclusion.} Applying Theorem \ref{orbit-counting-tdf} to the induced map $\hat{F} : X \to X$ and using Lemmas \ref{lyapunov1} and \ref{lyapunov2}, we get
$$
\lim_{T \to \infty} \frac{n_{\widehat F:\, X \to X}(x, T, B)}{e^T} \, = \, \frac{\ell_X(B)}{\int_{X} \log|\hat{F}'(z)| d\ell_X} \, = \, \frac{\ell(B)}{\int_{\mathbb{R}} \log|F'(z)| d\ell},
$$
for $x \in X$ outside a Lebesgue measure zero set of points that do not have unique codings.
Replacing $n_{\widehat F:\, X \to X}(x, T, B)$ with $n_{F:\, \mathbb{R} \to \mathbb{R}}(x, T, B)$, we see that
Theorem \ref{parabolic-orbit-counting} holds for a.e.~$x \in X$.
 
Recalling that $X = X_N$ and taking $N \to \infty$ shows that Theorem \ref{parabolic-orbit-counting} holds for a.e.~$x \in \mathbb{R}$. Approximating a general point $x \in \mathbb{R}$ by a sequence of points $x_n$ for which the conclusion of Theorem \ref{parabolic-orbit-counting} is valid and using Koebe's Distortion Theorem as in 
Section \ref{sec:weak-mixing} shows that Theorem \ref{parabolic-orbit-counting} holds for all $x \in \mathbb{R}$.

\appendix

\section[Appendix. Continuity of Transfer Operators]{Continuity of Transfer Operators}
\label{sec:continuity-of-sp-transfer-operators}

In this appendix, $q \ge 0$ will a positive real number and $0 \le \varepsilon < 1$.
We show that if $\psi$ is a normal $(1, q+\varepsilon)$-summable potential, then the modified transfer operator $s \to \mathcal L_{s,q}$ given by
(\ref{eq:Lsp-def}) defines a $C^\varepsilon_{\loc}$ function from $\overline{\mathbb{C}_1^+}$ to $\mathcal B(C^\alpha(E_A^\infty))$. We will often abbreviate $p = q + \varepsilon$.

The proof is not difficult but is somewhat tedious because one has to deal with the $C^{\varepsilon}$ and $C^\alpha$ H\"older norms, both of which consist of two pieces (the $L^\infty$ norm and the H\"older variation). 

\subsection{\texorpdfstring{Some Properties of the Function $x \mapsto G(x) = x^p e^{-sx}$}{Some properties of the Function G(x) = x\textasciicircum p e\textasciicircum \{-sx\}}}

Fix a real number $c > 0$. Loosely speaking, the following lemma says that for $x \in [c, \infty)$ and $s \in \overline{\mathbb{C}_1^+}$,
the function 
$G(x) = x^p e^{-sx}$ behaves similarly to $G_0(x) = e^{-sx}$\,:

\begin{lemma}
\label{exp-bounds}
 The function $G(x,s) = x^p e^{-sx}$ enjoys the following properties on $[c,\infty) \times 
\overline{\mathbb{C}_1^+}$\,:

\begin{enumerate}
\item 
If $1 \le \sigma_1 \le \sigma_2$ and $x_1 \le x_2$ then $G(x_2, \sigma_2) \lesssim G(x_1, \sigma_1).$

\item $|G_x(x,s)| \lesssim |s| \cdot |G(x,s)|.$

\item $|G_s(x,s)| \lesssim x \cdot |G(x,s)|.$

\item $|G_{xs}(x,s)| \lesssim x \cdot |s| \cdot |G(x,s)|.$
\end{enumerate}
The implicit constants depend on the parameters $p \ge 0$ and $c > 0$.
\end{lemma}

\begin{corollary}
\label{exp-bounds2}
Suppose $(x_1, s_1), (x_2, s_2) \in (c,\infty) \times \overline{\mathbb{C}_1^+}$.

{\em (i)}  If  $x_{\min} = \min(x_1,x_2)$, then
$$
\Delta_x  \, = \,  |G(x_2, s) - G(x_1,s)| \, \lesssim \, G(x_{\min}, s) \cdot |x_2 - x_1| \cdot |s|.
$$

{\em (ii)}  If $\sigma_{\min} = \min(\re s_1, \re s_2)$, then
$$
\Delta_s  \, = \,  |G(x, s_2) - G(x,s_1)| \, \lesssim \, G(x, \sigma_{\min}) \cdot x \cdot |s_2 - s_1|.
$$

{\em (iii)} With $x_{\min}$ and $\sigma_{\min}$ above, we have
\begin{align*}
\Delta_{xs} & \, = \, |G(x_2, s_2) - G(x_1,s_2) - G(x_2, s_1) + G(x_1,s_1)| \\
& \, \lesssim \, G(x_{\min}, \sigma_{\min}) \cdot  |x_2 - x_1| \cdot \max(|s_1|, |s_2|) \cdot \min \bigl (1, |s_2 - s_1| \cdot \max(x_1, x_2) \bigr ).
\end{align*}
\end{corollary}

\begin{proof}
(i, ii) By the fundamental theorem of calculus, $\Delta_x = \bigl | \int_{x_1}^{x_2} G_x(x,s) dx \bigr |$. By the mean value theorem, there exists $x^* \in [x_1, x_2]$ for which
$$
\Delta_x \, = \, |G_x(x^*, s)(x_2 - x_1) | \, \le \, |G_x (x^*, s)| \cdot |x_2 - x_1|,
$$
Part (i) now follows from Properties 1 and 2. 
The proof of (ii) is similar.

(iii) To prove the estimate with the minimum replaced by 1, we write 
$$
\Delta_{xs} = \int_{x_1}^{x_2} \bigl \{ G_x(x,s_2) - G_x(x,s_1) \bigr \} \, dx.
$$
and proceed as above. On the other hand, the representation 
$$\Delta_{xs} = \biggl | \int_{s_2}^{s_1} \int_{x_2}^{x_1}  G_{xs}(x,s) dx ds \biggr |,$$
 tells us that
$$
\Delta_{xs} = |G_{xs} (x^*, s^*)| \cdot |x_2 - x_1| \cdot |s_2 - s_1|,
$$
for some $x^* \in [x_1, x_2]$ and $s^* \in [s_1,s_2]$. The estimate with the minimum replaced by $|s_2 - s_1| \cdot \max(x_1, x_2)$ now follows from Properties 1 and 4.
\end{proof}

\subsection{Proof of Boundedness}

We first show that if $\psi$ is a normal (1,\,$p$)-summable potential on $E_A^\infty$, then for any $s \in \overline{\mathbb{C}^+_1}$, the linear operator $\mathcal L_{s,p}$ is bounded on $C^\alpha(E_A^\infty)$. For this purpose, we examine the behaviour of $ \psi^p e^{s\psi}$ on a cylinder $[a] \subset E^\infty_A$:

\begin{lemma}
\label{spe-holder-estimate}
For any $p \ge 0$, $s \in \overline{\mathbb{C}^+_1}$ and normal $(1,p)$-summable potential $\psi$, we have 
\begin{equation}
\label{eq:spe-holder-estimate}
\| \psi^p e^{s\psi} \circ a \|_\alpha \le C(s, p, \psi) \cdot \sup_{[a]} |\psi^p e^{s\psi}|, \qquad a \in E.
\end{equation}
\end{lemma}

\begin{proof}
The $L^\infty$ norm
$$
\| \psi^p e^{s\psi} \circ a \|_\infty = \sup_{[a]} |\psi^p e^{s\psi}|
$$
is trivially bounded by the right side of (\ref{eq:spe-holder-estimate}). To estimate the $\alpha$-H\"older variation of
$\psi^p e^{s\psi} \circ a$, we use Corollary \ref{exp-bounds2}(i):
$$
|\psi^p e^{s\psi} (a\omega) - \psi^p e^{s\psi} (a\tau)|  \lesssim |\psi(a\omega) - \psi(a\tau)| \cdot |s| \cdot \sup_{[a]} |\psi^p e^{s\psi}|.
$$
Dividing both sides by $d(\omega, \tau)^\alpha$ and taking the supremum over $\omega, \tau \in E_A^\infty$, we obtain
$$
v_\alpha \bigl ( \psi^p \exp(s\psi) \circ a  \bigr ) \lesssim v_\alpha^{(1)}(\psi) \cdot |s| \cdot \sup_{[a]} |\psi^p e^{s\psi}|.
$$
Absorbing $v_\alpha^{(1)}(\psi) \cdot |s|$ into the implicit constant completes the proof.
\end{proof}

\begin{corollary}
For any $p \ge 0$ and $s \in \overline{\mathbb{C}^+_1}$, the operator $\mathcal L_{s,p}$ is bounded on $C^\alpha(E_A^\infty)$.
\end{corollary}

\begin{proof}
Suppose $g \in C^\alpha(E_A^\infty)$. We may write
$$
\mathcal L_{s,p} g = \sum_{e \in E} (\psi^p e^{s\psi} \circ a) \cdot (g \circ a).
$$
By Lemma \ref{basic-holder},
$$
\| \mathcal L_{s,p} g \|_\alpha \, \le \, 3 \sum_{a \in E}  \| \psi^p e^{s\psi} \circ a \|_\alpha \cdot \|g \|_\alpha
\, \lesssim \,  \|g \|_\alpha \cdot \sum_{a \in E} \sup_{[a]} |\psi^p e^{s\psi}|,
$$
which is finite by Lemma \ref{sp-summability-lemma}.
\end{proof}

\subsection{Proof of Continuity (Part 1)}
\label{sec:continuity-I}

We first examine the case when $\varepsilon > 0$:

\begin{lemma}
\label{continuity-Ia}
If the potential $\psi$ is $(1,\, q+\varepsilon)$ summable with $q \ge 0$ and $0 < \varepsilon < 1$, then $s \to \mathcal L_{s, q}$ is a $C^\varepsilon_{\loc}$ mapping from $\overline{\mathbb{C}_1^+}$ to $\mathcal B(C(E^\infty_A))$.
\end{lemma}

\begin{proof}
We show that the quotient
\begin{equation}
\label{eq:pointwise-Ceps}
\frac{\| \mathcal L_{s, q} g - \mathcal L_{t, q} g \|_{C(E^\infty_A)}}{|s-t|^\varepsilon}.
\end{equation}
is uniformly bounded above over all $s \ne t \in \overline{\mathbb{C}_1^+}$. The numerator of (\ref{eq:pointwise-Ceps}) evaluated at a point $\omega \in E^\infty_A$ is
$$
 \biggl | \sum_{a \in E} (\psi^q e^{s\psi}\circ  a -\psi^q e^{t\psi} \circ a)(\omega) \cdot g(\omega) \biggr |,
$$
which by Corollary \ref{exp-bounds2}(ii) and the inequality $\min(1, x) \le x^\varepsilon$, $x \ge 0$ is
\begin{align*}
&\lesssim \| g \|_\infty \cdot \sum_{a \in E}  \sup_{[a]} |\psi^q e^{\sigma_{\min} \psi}| \cdot \min \bigl (1, |(s-t) \psi(a \omega)| \bigr ) \\
& \lesssim \| g \|_\infty \cdot \sum_{a \in E}  \sup_{[a]} |\psi^q e^{\sigma_{\min} \psi}| \cdot  |(s-t) \psi(a \omega)|^\varepsilon \\
& \lesssim \| g \|_\infty \cdot \sum_{a \in E}  \sup_{[a]} |\psi^{q + \varepsilon} e^{\sigma_{\min} \psi}| \cdot  |s-t|^\varepsilon,
\end{align*}
where $\sigma_{\min} = \min(\re s, \re t)$.
In view of Lemma \ref{sp-summability-lemma}, the quotient (\ref{eq:pointwise-Ceps}) is
$$
\lesssim \,  \| g \|_\infty \cdot \sum_{a \in E}  \sup_{[a]} |\psi^{q + \varepsilon} e^{\sigma_{\min} \psi}| \, < \, \infty,
$$
as desired.
\end{proof}

We now treat the case when $\varepsilon = 0$:

\begin{lemma}
\label{continuity-Ib}
If the potential $\psi$ is $(1,\, q)$ summable for some $q \ge 0$, then the mapping $s \to \mathcal L_{s,q}$ is continuous from $ \overline{\mathbb{C}_1^+}$  to  $\mathcal B(C(E^\infty_A))$.
\end{lemma}

\begin{proof}
Fix an $s \in \overline{\mathbb{C}_1^+}$. Evidently, for any $t \in \overline{\mathbb{C}_1^+}$,
\begin{align*}
(\mathcal L_{s, q} g - \mathcal L_{t, q})(\omega) & = 
\sum_{a \in E} (\psi^q e^{s\psi}\circ  a -\psi^q e^{t\psi} \circ a)(\omega) \cdot g(\omega) \\
& \le 2 \| g \|_\infty \cdot \sum_{a \in E}  \sup_{[a]} |\psi^q e^{\sigma_{\min} \psi}|
\end{align*}
is dominated by a convergent series. Since the individual terms tend to 0  uniformly in $\omega$ as $t \to s$, 
$$(\mathcal L_{s, q} g - \mathcal L_{t, q})(\omega) \to 0, \qquad \text{uniformly in }\omega \in E_A^\infty, \qquad \text{as }t \to s,$$ which is the desired continuity statement.
\end{proof}

\subsection{Proof of Continuity (Part 2)}
\label{sec:continuity-II}

As before, we first examine the case when $\varepsilon > 0$:

\begin{lemma} If the potential $\psi$ is $(1,\, q+\varepsilon)$ summable with $q \ge 0$ and $0 < \varepsilon < 1$, then $s \to \mathcal L_{s, q}$ is a $C^\varepsilon_{\loc}$ mapping  from $\overline{\mathbb{C}_1^+}$ to $\mathcal B(C^\alpha(E^\infty_A))$.
\end{lemma}

\begin{proof}
Since we have already proved that $s \to \mathcal L_{s,q}$ is a $C^\varepsilon_{\loc}$ mapping  from $\overline{\mathbb{C}_1^+}$ to $\mathcal B(C(E^\infty_A))$ in Lemma \ref{continuity-Ia}, it remains to estimate the $\alpha$-variation, which amounts to giving an upper bound for
$$
\frac{ \bigl | \sum_{a \in E} (\psi^q e^{s\psi}\circ  a -\psi^q e^{t\psi} \circ a)(\omega) \cdot g(\omega) - (\psi^q  e^{s\psi}\circ  a -\psi^q e^{t\psi} \circ a)(\tau) \cdot g(\tau) \bigr |}
{|s-t|^\varepsilon \cdot d(\omega, \tau)^\alpha},
$$
for a fixed $s \in \overline{\mathbb{C}_1^+}$. Below, we assume $|s - t| \le 1$ so that $|s| \asymp |t|$. We split the numerator of the above expression into two parts $A + B$, which we estimate separately. 

For the first summand, the estimate is similar to the one in the proof of Lemma \ref{continuity-Ia}:
\begin{align*}
A & \le  \sum_{a \in E} \bigl | (\psi^q e^{s\psi}\circ  a -\psi^q e^{t\psi} \circ a)(\tau) \cdot (g(\tau) - g(\omega)) \bigr | \\
& \le \| g \|_\alpha \cdot d(\omega, \tau)^\alpha \cdot \sum_{a \in E} \bigl | (\psi^q e^{s\psi}\circ  a -\psi^q e^{t\psi} \circ a)(\tau) | \\
& \le \| g \|_\alpha \cdot d(\omega, \tau)^\alpha \cdot \sum_{a \in E} |\psi^q e^{\sigma_{\min} \psi}| \cdot   \min \bigl (1, |(s-t) \psi(a \tau)| \bigr ) \\
& \le \| g \|_\alpha \cdot d(\omega, \tau)^\alpha \cdot \sum_{a \in E} \sup_{[a]}  |\psi^q e^{\sigma_{\min} \psi}| \cdot  |(s-t) \psi(a \tau)|^\varepsilon \\
& \lesssim \, \| g \|_\alpha \cdot d(\omega, \tau)^\alpha \cdot |s-t|^\varepsilon \cdot \sum_{a \in E}  \sup_{[a]} |\psi^{q + \varepsilon} e^{s_{\min} \psi}|,
\end{align*}
where as usual, $\sigma_{\min} = \min(\re s, \re t)$.

For the other summand, we use Corollary \ref{exp-bounds2}(iii):
\begin{align*}
B & \le \sum_{a \in E} \bigl |(\psi^q e^{s\psi}\circ  a -\psi^q e^{t\psi} \circ a)(\omega) \cdot g(\omega) - (\psi^q e^{s\psi}\circ  a -\psi^q e^{t\psi} \circ a)(\tau) \cdot g(\omega) \bigr | \\
& \le \|g \|_\infty \cdot \sum_{a \in E} \bigl |(\psi^qe^{s\psi}\circ  a -\psi^q e^{t\psi} \circ a)(\omega)  - (\psi^q e^{s\psi}\circ  a -\psi^q e^{t\psi} \circ a)(\tau) \bigr | \\
& \le \|g \|_\infty \cdot \sum_{a \in E} \sup_{[a]} |\psi^q e^{\sigma_{\min} \psi}| \cdot |s| \cdot |\psi(a\omega) - \psi(a\tau)| \cdot
\min \Bigl (1, \, |s-t| \cdot \sup_{[a]} |\psi| \Bigr ) \\
& \le  |s| \cdot \|g \|_\infty \cdot d(\omega, \tau)^\alpha \cdot v_\alpha^{(1)}(\psi) \cdot |s-t|^{\varepsilon} \cdot \sum_{a \in E}  
\sup_{[a]} |\psi^{q+\varepsilon} e^{s_{\min} \psi}|,
\end{align*}
as desired.
\end{proof}

We now treat the case when $\varepsilon = 0$:

\begin{lemma}
If the potential $\psi$ is $(1,\, q)$ summable for some $q \ge 0$, then the mapping $s \to \mathcal L_{s, q}$ is continuous from $\overline{\mathbb{C}_1^+}$ to $\mathcal B(C^\alpha(E^\infty_A))$.
\end{lemma}

\begin{proof}
Fix an $s \in \overline{\mathbb{C}_1^+}$. Since we have already established that $s \to \mathcal L_{s, q}$ is continuous from $\overline{\mathbb{C}_1^+}$ to $\mathcal B(C(E^\infty_A))$ in Lemma \ref{continuity-Ib}, it remains to control the $\alpha$-variation. Namely, we need to check that
\begin{equation*}
\sup_{\omega \ne \tau} \, \frac{(\mathcal L_{s, q} g - \mathcal L_{t, q} g)(\omega) - (\mathcal L_{s, q} g - \mathcal L_{t, q} g)(\tau) }{d(\omega, \tau)^\alpha} \, \to \, 0, \qquad \text{as }\overline{\mathbb{C}_1^+} \ni t \to s.
\end{equation*}
We assume that $|s - t| \le 1$ and write the quotient on left hand side as
 $$
\sum_{a \in E} \, \biggl \{ \frac{(\psi^q e^{s\psi}\circ  a -\psi^q e^{t\psi} \circ a)(\omega) - (\psi^q  e^{s\psi}\circ  a -\psi^q e^{t\psi} \circ a)(\tau)}
{d(\omega, \tau)^\alpha} \cdot  g(\omega)
$$
$$
+ \, (\psi^q  e^{s\psi}\circ  a -\psi^q e^{t\psi} \circ a)(\tau) \cdot \frac{ g(\omega) - g(\tau)}
{d(\omega, \tau)^\alpha} \biggr \}.
$$
As each individual term tends to 0 as $t \to s$ and the sum is dominated by the convergent series
$$
\Bigl ( |s| \cdot \| g \|_\infty \cdot v_\alpha^{(1)}(\psi) + \| g \|_\alpha \Bigr ) \cdot \sum_{a \in E} |\psi^q e^{\sigma_{\min} \psi}|,
$$
the entire sum tends to 0 as $t \to s$, uniformly in $\omega \ne \tau \in E^\infty_A$.
\end{proof}

\subsection*{Acknowledgements}
This research was supported by the Israel Science Foundation (grant no.~3134/21) and the Simons Foundation (grant no.~581668). The authors also wish to thank 
the organizers of the thematic research programme ``Modern holomorphic dynamics and related fields'' at the University of Warsaw for their hospitality.

\bibliographystyle{amsplain}

\end{document}